\documentclass{amsart}
\usepackage{verbatim}
\usepackage{hyperref}
\usepackage{geometry}
\usepackage{amsfonts}
\usepackage{amssymb}
\usepackage{url}
\usepackage{amsmath}
\usepackage{amsthm}
\usepackage{ulem}
\usepackage{todonotes}

\newcommand{\xTo}[1]{\stackrel{#1}{\To}}
\newcommand{\knf}{K^{\mathrm{inv}}}

\renewcommand{\todo}[1]{}

\newcommand{\can}{\mathrm{can}}
\newcommand{\THH}{\mathrm{THH}}
\newcommand{\HH}{\mathrm{HH}}

\newcommand{\TP}{\mathrm{TP}}
\renewcommand{\hom}{\mathrm{Hom}}
\newcommand{\HC}{\mathrm{HC}}
\newcommand{\md}{\mathrm{Mod}}
\newcommand{\Sp}{\mathrm{Sp}}
\newcommand{\fun}{\mathrm{Fun}}

\usepackage{xy}
\input xy
\newcommand{\spec}{\mathrm{Spec}}
\renewcommand{\sp}{\mathrm{Sp}}
\xyoption{all}

\renewcommand{\mod}{\mathrm{Mod}}

\theoremstyle{definition}
\newtheorem{definition}{Definition}[section]
\newtheorem{example}[definition]{Example}
\newtheorem{proposition}[definition]{Proposition}

\newtheorem{lemma}[definition]{Lemma}
\newtheorem{corollary}[definition]{Corollary}
\newtheorem{cons}[definition]{Construction}

\newtheorem{remark}[definition]{Remark}
\newtheorem{theorem}[definition]{Theorem}

\newtheorem*{var}{Variant}
\newtheorem*{question}{Question}
\newcommand{\lotimes}{\otimes^{\mathbb{L}}}
\newtheorem{thm}{Theorem}

\newcommand{\bb}{\mathbb}
\newcommand{\To}{\longrightarrow}

\newcommand{\isoto}{\stackrel{\simeq}{\to}}
\DeclareMathOperator{\dlog}{dlog}
\newcommand{\op}{\operatorname}
\newcommand{\sub}[1]{{\mbox{\rm \scriptsize #1}}}

\newcommand{\dotimes}{\otimes^{\bb L}}
\renewcommand{\hat}{\widehat}
\newcommand{\xto}{\xrightarrow}
\newcommand{\TR}{\mathrm{TR}}
\DeclareMathOperator{\Spec}{Spec}
\newcommand{\al}{\alpha}
\newcommand{\CycSp}{\mathrm{CycSp}}
\newcommand{\TC}{\mathrm{TC}}

\newcommand{\nuloc}{\mathrm{Ring}^{\mathrm{nu,loc}}_R}
\newcommand{\nur}{\mathrm{Ring}^{\mathrm{nu}}_R}

\newcommand{\nuh}{\mathrm{Ring}^{\mathrm{nu,h}}_R}
\newcommand{\nuzh}{\mathrm{Ring}^{\mathrm{nu,h}}}

\newcommand{\lan}{\mathrm{Lan}}

\begin{document}

\title{$K$-theory and topological cyclic homology of henselian pairs}
\date{\today}

\author{Dustin Clausen}
\address{Max-Planck-Institut f\"ur Mathematik, Bonn}
\email{dclausen@mpim-bonn.mpg.de}

\author{Akhil Mathew}
\address{Department of Mathematics, University of Chicago}
\email{amathew@math.uchicago.edu}

\author{Matthew Morrow}
\address{CNRS \& Institut de Math\'ematiques de Jussieu-Paris Rive Gauche, Sorbonne Universit\'e
}
\email{matthew.morrow@imj-prg.fr}

\maketitle 
\begin{abstract}
Given a henselian pair $(R, I)$ 
of commutative rings, 
we show that the relative $K$-theory and relative topological cyclic homology
with finite coefficients are identified via the cyclotomic trace $K \to \TC$. 
This yields a generalization of the classical Gabber--Gillet--Thomason--Suslin rigidity theorem 
(for mod $n$ coefficients, with $n$ invertible in $R$) and McCarthy's theorem on
relative $K$-theory (when
$I$ is nilpotent). 

We deduce that the cyclotomic trace is an equivalence in large
degrees between $p$-adic $K$-theory and topological cyclic homology for 
a large class of $p$-adic rings. 
In addition, we show that $K$-theory with finite coefficients satisfies continuity 
for complete noetherian rings which are $F$-finite modulo $p$. 
Our main new ingredient is a basic finiteness property of $\TC$ with finite
coefficients. 
\end{abstract}

\tableofcontents

\section{Introduction}

\subsection{Rigidity results}
The purpose of this paper is to study the (connective) algebraic $K$-theory $K(R)$ of a commutative ring $R$, by means of the cyclotomic trace \cite{BHM}
\[ K(R) \to \TC(R),   \]
from $K$-theory to {topological cyclic homology} $\TC(R)$. 
The cyclotomic trace is known to be an extremely useful tool in studying
$K$-theory. On the one hand, $\TC(R)$ is often easier to calculate directly than
$K(R)$ and has various arithmetic interpretations. For
instance, according to work of Bhatt--Morrow--Scholze \cite{BMS2}, $\TC(R)$ for
$p$-adic rings is a form of syntomic cohomology of $\spec R$. On the other
hand, the cyclotomic trace is often an effective approximation to algebraic
$K$-theory. It is known that the cyclotomic
trace is a $p$-adic equivalence in nonnegative degrees for finite algebras over the Witt vectors over a
perfect field \cite{HM97}, and the cyclotomic trace has been used in several
fundamental calculations of algebraic $K$-theory such as \cite{HMlocal}.

Our main theorem extends the known range of situations in which $K$-theory is
close to TC. To formulate our results cleanly, we introduce the following
notation. 

\begin{definition} 
For a ring $R$, we write  $\knf(R)$ for the  homotopy fiber of the cyclotomic trace $K(R) \to \TC(R)$. 
\end{definition}

Thus $\knf(R)$ measures the difference between $K$ and $\TC$.  The following fundamental result 
(preceded by the rational version due to Goodwillie \cite{Goodrel} and the $p$-adic version
proved by McCarthy \cite{mccarthy} and generalized by Dundas \cite{dundas}) shows that 
$\knf$ is nil-invariant.

\begin{theorem}[Dundas--Goodwillie--McCarthy \cite{dgm}]
\label{DGMthm}
Let $R \to R'$ be a map of rings which is a surjection with nilpotent
kernel.\footnote{Or, more generally, a map of connective $\bb{E}_1$-ring spectra such that $\pi_0R\rightarrow\pi_0R'$ is a surjection with  nilpotent kernel} Then the map $\knf(R) \to \knf(R')$ is an equivalence. \end{theorem}

The above result is equivalent to the statement that for a nilpotent ideal, the
relative $K$-theory is identified with relative topological cyclic homology. 
Our main result is an extension  of Theorem~\ref{DGMthm} with finite coefficients to a more
general class of surjections. 
We use the following classical definition in commutative algebra (see also
Definition~\ref{def:henspair} below). 
\begin{definition} 
Let $R$ be a commutative ring and $I \subset R$ an ideal. Then $(R, I)$ is said
to be a \emph{henselian pair} if given a polynomial
$f(x) \in R[x]$ and a root $\overline{\alpha} \in R/I$ of $\overline{f} \in (R/I)[x]$ with
$\overline{f}'(\alpha)$ being a unit of $R/I$, then $\overline \alpha$ lifts to a root $\alpha \in R$ of $f$.

Examples of henselian pairs include pairs $(R, I)$ where $R$ is $I$-adically
complete (by Hensel's lemma) and pairs $(R, I)$ where $I$ is locally nilpotent. 
\end{definition} 

Algebraic $K$-theory with finite coefficients prime to the characteristic 
is known to interact well with henselian pairs; one has the following result of Gabber
\cite{gabber}, preceded by work of Suslin \cite{Suslinlocal} and  Gillet--Thomason \cite{GT}. 
See also \cite[Sec.~4.6]{Arf} for a textbook reference. 

\begin{theorem}[Gabber \cite{gabber}] 
\label{gabberthm}
Let $(R, I)$ be a henselian pair. Suppose $n$ is invertible in $R$. Then the map
$K(R)/n \to K(R/I)/n$ is an equivalence of spectra. 
\end{theorem}

The main result of this paper is the following common extension of Theorem
~\ref{gabberthm} and the commutative and profinitely completed case of Theorem ~\ref{DGMthm}.
	
\begin{thm} 
\label{ourthm}
Let $(R, I)$ be a henselian pair. Then for any $n$, the map
$\knf(R)/n \to \knf(R/I)/n$ is an equivalence.  
\end{thm}

What is the significance of such rigidity results in $K$-theory? 
Computing the algebraic $K$-theory of rings and schemes is a fundamental and
generally very difficult problem. 
One of the basic tools in doing so is \emph{descent}: that is, reducing the
computation of 
the $K$-theory of certain rings to that of other (usually easier) rings built
from them.  As is well-known, algebraic $K$-theory generally does not
satisfy descent for the \'etale topology.  
On the other hand, 
a general result of Thomason--Trobaugh \cite{TT90} (see also
\cite{ThomasonICM} for a survey) states that
algebraic $K$-theory of rings and quasi-compact quasi-separated schemes satisfies
descent for the \emph{Nisnevich} topology, which is quite well-behaved for a
noetherian scheme of finite Krull dimension. 
In the Nisnevich topology, the points are given by the spectra of
\emph{henselian}
local rings. Up to a descent spectral sequence, algebraic $K$-theory can thus
be computed if it is understood for henselian local rings.
When $n$ is invertible, 
Theorem~\ref{gabberthm} enables one to reduce the calculation (with mod $n$
coefficients) to the $K$-theory
of fields.  Our main result extends this to the case where $n$ is not assumed invertible, but with the additional term coming from $\TC$.  All of this uses only the local case of Theorem \ref{ourthm}; invoking the general case also gives further information.

\subsection{Consequences}
As a consequence of Theorem~\ref{ourthm}, we deduce various global structural properties about algebraic
$K$-theory and topological cyclic homology, 
especially $p$-adic $K$-theory of $p$-adic rings. 
In many cases, we are able to extend known properties in the smooth case to
provide results on the $K$-theory of singular schemes. 

The first main consequence 
of our results is a general statement that $p$-adic algebraic $K$-theory and
$\TC$ agree in large enough degrees for reasonable $p$-torsion schemes, or
affine schemes on which $(p)$ is henselian. 

\begin{thm}[Asymptotic comparison of $K, \TC$] 
\label{KTCinlarge}
Let $R$ be a ring  henselian along $(p)$ and such that $R/p$ has finite Krull
dimension. 
Let $d = \sup_{x \in \spec (R/p) } \log_p [k(x):
k(x)^p]$ where $k(x)$ denotes the residue field at $x$ and $k(x)^p \subset
k(x)$ the subfield of $p^\sub{th}$ powers. Then 
the map $K(R)/p^i \to \TC(R)/p^i$ is a equivalence in degrees $\geq
\max(d, 1)$ for each $i \geq 1$. 
\end{thm} 
	
Theorem~\ref{KTCinlarge} specializes to a number of existing results and calculations of algebraic
$K$-theory, and enables new ones. 
\begin{enumerate}
\item  
For finitely generated algebras over a perfect field, the result was shown in
the \emph{smooth} case by Geisser--Levine \cite{GL} and Geisser--Hesselholt \cite{GH}:
in fact, both $K$-theory and TC vanish mod $p$ in sufficiently large degrees.
\item
For singular curves, the
result appears in Geisser--Hesselholt \cite{GHexc}. 
\item 
Our approach also applies to any
semiperfect or semiperfectoid ring, where it shows that $K/p$ is the connective cover of $\TC/p$.  This recovers calculations of
Nizio\l \cite{Niziol-crystalline} and Hesselholt \cite{HesselholtOC} of the $K$-theory of the ring 
$\mathcal{O}_{\mathbb{C}_p}$ of integers in the completed algebraic closure
$\mathbb{C}_p$ of $\mathbb{Q}_p$. See \cite[Sec.~7.4]{BMS2} for some recent
applications. 
\item If $R$ is any noetherian ring  henselian along $(p)$ and such that
$R/p$ is \emph{$F$-finite} (i.e., the Frobenius map on $R/p$
is finite), then the above result applies: $p$-adic $K$-theory and $\TC$ agree
in sufficiently large degrees. 
\end{enumerate}

Another application of our results is to show that $p$-adic \emph{\'etale} $K$-theory
is identified with topological cyclic homology under quite general situations;
this is shown in \cite{GH} in the smooth case. 
As a consequence, we may regard Theorem~\ref{KTCinlarge} as a type of $p$-adic
Lichtenbaum--Quillen statement. 

\begin{thm}[\'Etale $K$-theory is $\TC$ at points of characteristic $p$] 
Let $R$ be a strictly henselian local ring with residue field of characteristic
$p > 0$. Then $\knf(R)/p = 0$, i.e., the map $K(R) \to \TC(R)$ is a $p$-adic equivalence. 
\end{thm} 

In addition, we are able to obtain a general split injectivity statement about the
cyclotomic trace of local $\mathbb{F}_p$-algebras. 
\begin{thm}[Split injectivity of the cyclotomic trace]
For any local $\mathbb{F}_p$-algebra $R$ and any $i \geq 1$, the cyclotomic trace $K(R)/p^i \to
\TC(R)/p^i$ is split injective on homotopy groups. 
\end{thm}

In fact, the splitting is functorial for ring homomorphisms, and  one
can identify the complementary summand in terms of de Rham--Witt cohomology: see
Proposition~\ref{identifytwiddle}.

Next, our results also imply statements internal to $K$-theory itself, especially the
$K$-theory of $\mathbb{F}_p$-algebras. 
Let $R$ be a local $\mathbb{F}_p$-algebra. 
If $R$ is regular, there is a simple formula for the mod $p$ algebraic $K$-theory of $R$ 
given in the work of Geisser--Levine \cite{GL}. 
We let $\Omega^i_{R, \mathrm{log}}$ denote the subgroup of $\Omega^i_R$
consisting of elements which can be written as sums of products
of forms $dx/x$ for $x$ a unit. 

\begin{theorem}[Geisser--Levine \cite{GL}] 
The mod $p$ $K$-groups
$K_i(R; \mathbb{Z}/p\mathbb{Z}) = \pi_i( K(R)/p)$ are identified with the logarithmic 
forms $\Omega^{i}_{R, \mathrm{log}}$ for $i \geq 0$. 
\end{theorem} 

Algebraic $K$-theory of singular $\mathbb{F}_p$-algebras is generally much more
complicated. However, using results of \cite{Morrow-HW, Morrowpro}, we are able to prove a
pro-version of the Geisser--Levine theorem (extending results of Morrow
\cite{Morrow-HW}). 

\begin{thm}[Pro Geisser--Levine]
For any regular local $F$-finite $\mathbb{F}_p$-algebra $R$ and ideal $I \subset R$, 
there is an isomorphism of pro abelian groups
$\{K_i(R/I^n; \mathbb{Z}/p \mathbb{Z})\}_{n \geq 1} \simeq \left\{\Omega^i_{R/I^n,
\mathrm{log}}\right\}_{n \geq 1}$ for each $i \geq 0$. 
\end{thm}
Finally, we 
study the {continuity} question in algebraic $K$-theory for complete rings, considered by various
authors including \cite{GHlocal, dundas, Suslinlocal, Morrow-HW, Panin}. 
That is, when $R$ is a ring and $I$ an ideal, we study how close the map $K(R)
\to \varprojlim K(R/I^n)$ is to being an equivalence. 
Using results of
Dundas--Morrow \cite{DundasMorrow} on topological cyclic homology, we prove a general
continuity statement in $K$-theory. 
\begin{thm}[Continuity criterion for $K$-theory] 
Let $R$ be a noetherian ring and $I \subset R$ an ideal. Suppose $R$ is
$I$-adically complete and $R/p$ is $F$-finite. 
Then the map $K(R) \to \varprojlim K(R/I^n)$ is a $p$-adic equivalence. 
\end{thm} 

In fact, an argument using Popescu's approximation theorem shows that this
continuity result is essentially equivalent to our main theorem on henselian
pairs, see Remark~\ref{henspairequivtocontinuity}.

In general, our methods do not control the \emph{negative}
$K$-theory.\footnote{In fact,
Gabber rigidity fails for negative $K$-theory, cf.~the discussion in
\cite[Ex.~8.5]{Wei91}, based on \cite{Rei87}.} 
In
addition, they are essentially limited to the affine case: that is, 
we do not treat henselian pairs of general schemes. 
For example, given a complete local ring $(R, \mathfrak{m})$ and a smooth
proper scheme $X \to \spec R$, our methods 
do not let us compare the $K$-theory of $X$ to its special fiber. Note that 
questions of realizing formal $K$-theory classes as algebraic ones are expected
to be very difficult, cf.~\cite{BEK}.

Theorem~\ref{ourthm} depends essentially on a finiteness property
of topological cyclic homology with mod $p$ coefficients. 
In characteristic zero, negative cyclic homology is not a finitary invariant: that is, it
does not commute with filtered colimits.  The main technical tool we use in this paper is the observation that the situation is better modulo
$p$. 
Topological cyclic homology of a ring $R$ is built from the topological
Hochschild homology $\THH(R)$ and its natural structure as a cyclotomic spectrum. 
We use the Nikolaus--Scholze
\cite{nikolaus-scholze} description of the homotopy theory $\CycSp$ of
cyclotomic spectra to observe (Theorem~\ref{TCcommutes}) that topological cyclic homology modulo $p$ commutes with filtered colimits.

\begin{thm}[$\TC/p$ is finitary]
The construction $R \mapsto \TC(R)/p$, from rings $R$ to spectra, commutes with
filtered colimits. 
\end{thm}

Theorem~\ref{ourthm} also relies on various tools used in the classical rigidity results, reformulated in a slightly different
form as the finiteness of certain functors, as well as the theory of
non-unital henselian rings. 
In addition, it relies heavily on the calculations of Geisser--Levine
\cite{GH} and Geisser--Hesselholt \cite{GH} of $p$-adic algebraic $K$-theory and
topological cyclic homology for smooth schemes in characteristic $p$. 
We do not know if it is possible to give a proof of our result 
without using all of these techniques. 
However, in many cases (e.g., $\mathbb{F}_p$-algebras) Theorem~\ref{ourthm}
can be proved without
the full strength of the finiteness and spectral machinery; see Remark~\ref{rem:fpcaseeasier}.

\subsection*{Conventions}
In this paper, we will use the language of $\infty$-categories
and higher algebra
\cite{HTT, HA}. On the one hand, the constructions of algebraic $K$-theory and of
the cyclotomic trace are of course much older than the theory of $\infty$-categories. Many of
our arguments use standard homotopical techniques that could be carried out
in a modern model category of spectra, and we  have tried to minimize the use of
newer technology in the exposition when possible. On the other hand, we rely crucially
on the Nikolaus--Scholze \cite{nikolaus-scholze} approach to cyclotomic spectra,
which is $\infty$-categorical in nature.  

All rings will be commutative unless otherwise specified. The category of commutative algebras over a ring $R$ is denoted by $\op{CAlg}_{/R}$.

We will let $\Sp$ denote the $\infty$-category of spectra, and write $\otimes$
for the smash product of spectra. The sphere spectrum will be denoted by $S^0\in\Sp$. 
In a stable $\infty$-category $\mathcal{C}$, we will write
$\mathrm{Hom}_{\mathcal{C}}(\cdot, \cdot) \in \Sp$ for the mapping spectrum between any
two objects of $\mathcal{C}$.
We will generally write $\varprojlim$, $\varinjlim$ for inverse and direct limits
in $\Sp$; in a model categorical approach, these would be typically called
\emph{homotopy} limits and colimits and written $\mathrm{holim},
\mathrm{hocolim}$. Given an integer $r$, we let $\Sp_{\geq r}$ (resp.~$\Sp_{\leq r}$) for the
subcategory of $X \in\Sp$ such that $\pi_i(X) = 0$ for $i < r$ (resp.~$i>r$). 
We let 
$\tau_{\geq r}, \tau_{\leq r}$ denote the associated truncation functors.

We let $K$ denote connective $K$-theory and $\mathbb{K}$ denote its
nonconnective analog; $\TC$ denotes topological cyclic homology. 
Given a ring $R$ and an ideal $I \subset R$, we let $K(R, I)$
(resp.~$\mathbb{K}(R, I)$) denote the relative $K$-theory (resp.~relative nonconnective
$K$-theory), defined  as $\mathrm{fib}(K(R) \to K(R/I))$ and
$\mathrm{fib}(\mathbb{K}(R) \to \mathbb{K}(R/I))$. We define $ \knf(R, I), \TC(R, I)$
similarly. Finally, following traditional practice in algebraic $K$-theory, we
will sometimes write $K_*(R; \mathbb{Z}/p^r \mathbb{Z})$ for the mod $p^r$
homotopy of the spectrum $K(R)$ (i.e., $\pi_*( K(R)/p^r)$) and similarly for
$\TC$. 

\subsection*{Acknowledgments}
We are grateful to Benjamin Antieau, Bhargav Bhatt, Lars Hesselholt, Thomas Nikolaus, and Peter
Scholze for helpful discussions. 
We thank the referees for many helpful comments on an earlier version of the
paper. 
The second author would like to thank the Universit\'e Paris 13, the Institut de
Math\'ematiques de Jussieu-Paris Rive Gauche, and the University of Copenhagen for hospitality
during which parts of this work were done. 
The first author was supported by Lars Hesselholt's Niels Bohr Professorship.
This work was done while the second author was a Clay Research Fellow.

\section{The finiteness property of $\TC$}
\label{section_finiteness}

Given a ring $R$, we can form both its Hochschild homology $\HH(R/\mathbb{Z})$
and its topological Hochschild homology spectrum $\THH(R)$.  The main structure on Hochschild homology is the circle action,
enabling one to form the \emph{negative cyclic homology} $\HC^-(R/\mathbb{Z}) =
\HH(R/\mathbb{Z})^{hS^1}$. However,
there is additional information contained in topological Hochschild homology, encoded in the structure of a \emph{cyclotomic spectrum}.
Formally, cyclotomic spectra form a symmetric monoidal, stable $\infty$-category $(\CycSp,
\otimes, \mathbf{1})$, and $\THH(R) \in \CycSp$.  
An object of $\CycSp$ is in particular a spectrum with a circle action but also
further information. 
This structure enables one to form the topological cyclic homology $\TC(R) =
\hom_{\CycSp}(\mathbf{1}, \THH(R))$. 
Introduced originally by B\"okstedt--Hsiang--Madsen \cite{BHM}, nowadays there are
several treatments 
of $\CycSp$ by various
authors, including Blumberg--Mandell \cite{BMcyc},
Barwick--Glasman \cite{BG}, Nikolaus--Scholze \cite{nikolaus-scholze}, and  Ayala--Mazel-Gee--Rozenblyum
\cite{ayala-mg-rozenblyum}. 
In particular, the paper \cite{nikolaus-scholze} shows that the structure of $\CycSp$
 admits a dramatic simplification in the bounded-below case.

The purpose of this section is to prove a structural property of $\TC$. 
The datum of a circle action is essentially infinitary in nature: for example,
forming $S^1$-homotopy fixed points is an infinite homotopy limit, and the
construction $R \mapsto \HC^-(R/\mathbb{Z})$ does not
commute with filtered homotopy colimits.
Classical presentations of the homotopy
theory of cyclotomic spectra and $\TC$ usually involve infinitary limits, such
as 
a homotopy limit over a fixed point tower. 
However, in this section, we prove the slightly surprising but fundamental  property that $\TC/p$ commutes with filtered
colimits. Our proof is based on the simplification to the theory
$\CycSp$ in the bounded-below case demonstrated by Nikolaus--Scholze
\cite{nikolaus-scholze}.  

\subsection{Reminders on cyclotomic spectra}\label{subsection_reminders}
We start by recalling some facts about the $\infty$-category of ($p$-complete) cyclotomic spectra, following \cite{nikolaus-scholze}. 

\begin{definition}[Nikolaus--Scholze \cite{nikolaus-scholze}] 
A {\em cyclotomic spectrum} is a tuple  $$ (X ,  \, \{\varphi_{X, p}: X  \to X^{tC_p}\}_{p = 2, 3, 5, \dots
}),$$ where $X \in \fun(BS^1, \Sp)$ is a spectrum with an $S^1$-action which is moreover equipped with a map $\varphi_p:X\to X^{tC_p}$ (the {\em cyclotomic Frobenius} at $p$), for each prime number $p$, which is required to be equivariant with respect to the $S^1$-action on $X$ and the $S^1\simeq S^1/C_p$-action on the Tate construction $X^{tC_p}$. By a standard abuse of notation we will occasionally denote a cyclotomic spectrum simply by the underlying spectrum $X$.

For the precise definition of the $\infty$-category $\CycSp$ of cyclotomic
spectra (whose objects are as above) as a lax equalizer, we refer the reader to
\cite[II.1]{nikolaus-scholze}. See \cite[IV.2]{nikolaus-scholze} for a treatment
of the symmetric monoidal structure. 
\end{definition} 

The $\infty$-category $\CycSp$ naturally has the structure of a presentably
symmetric monoidal $\infty$-category\footnote{Meaning, a presentable $\infty$-category with symmetric monoidal tensor product commuting with colimits in each variable separately; or in other words a commutative algebra object in $\operatorname{Pr}^L$ with respect to Lurie's tensor product.} and the tensor product recovers the smash
product on the level of underlying spectra with $S^1$-action. Moreover, the
forgetful functor $\CycSp\to \Sp$ reflects equivalences, is exact, and preserves
all colimits. These properties may be deduced from the general formalism of lax equalizers; see \cite[II.1]{nikolaus-scholze}.

\begin{definition}\label{definition_TC_via_NS}
Given $X \in \CycSp$, its \emph{topological cyclic homology} $\TC(X) $ is
defined as the mapping spectrum $\TC(X) = \hom_{\CycSp}(\mathbf{1}, X)$, where
$\mathbf{1} \in \CycSp$ is the unit (cf.~Definition \ref{cons_trivial}). We moreover write \[ \TC^-(X) = X^{hS^1}, \quad \TP(X) = X^{tS^1}\] for its {\em negative topological cyclic homology} and \emph{periodic topological cyclic homology}, the latter having been particularly studied by Hesselholt \cite{hesselholt-tp}. 
\end{definition} 

The theory of cyclotomic spectra  studied more classically using equivariant stable homotopy theory \cite{BMcyc} (building on ideas introduced by B\"okstedt--Hsiang--Madsen \cite{BHM}) agrees with that of Nikolaus--Scholze in the bounded-below case \cite[Thm.~II.6.9]{nikolaus-scholze}. It is the bounded-below case that is of interest to us, and it will be convenient to introduce the following notation. 

\begin{definition} 
Given $n \in \mathbb{Z}$, let 
$\Sp_{\geq n}$ denote the full subcategory of $\Sp$ consisting 
of those spectra $X$ that satisfy 
$\pi_i(X) = 0$ for $i < n$. 
Let
$\CycSp_{\geq n}$ denote the full subcategory
of $\CycSp$ consisting of those cyclotomic spectra $(X,
\{\varphi_{X,p}\}_{p = 2, 3, 5, \dots} )$ such that the underlying
spectrum $X$ belongs to $\Sp_{\geq n}$.
\end{definition}

In the case in which $X\in \CycSp$ is bounded below and $p$-complete, the theory simplifies in several ways. Firstly, the Tate constructions $X^{tC_q}$ vanish for all primes $q\neq p$ \cite[Lem.~I.2.9]{nikolaus-scholze}, and therefore we do not need to specify the maps $\varphi_{X, q}$: the only required data is that of the $S^1$-action and the $S^1\simeq S^1/C_p$-equivariant Frobenius map $\varphi_X = \varphi_{X,p}: X \to X^{tC_p}$. Secondly, there is a basic simplification to the
formula for $\TC(X)$, using the two maps
\[ \can_X, \varphi_X :  \TC^-(X) \rightrightarrows  \TP(X).\]
Here $\can_X$ is the canonical map from homotopy fixed points to the Tate
construction, while $\varphi_X$ arises from taking $S^1$-homotopy-fixed points
in the cyclotomic structure map $X \to X^{tC_p}$ and using a version of the Tate orbit lemma
to identify 
$X^{tS^1} \simeq (X^{tC_p})^{h(S^1/C_p)}$ since $X$ is bounded below and $p$-complete \cite[Lem.~II.4.2]{nikolaus-scholze}. Using these two maps one has the fundamental formula
\cite[Prop II.1.9]{nikolaus-scholze} 
\begin{equation} \label{keyformula} \TC(X) = \mathrm{fib}( \TC^-(X)
\xrightarrow{\can_X - \varphi_X} \TP(X)),\end{equation}
or equivalently $\TC(X)$ is the equalizer of $\mathrm{can}_X$,
$\varphi_X$. 

A basic observation which goes into proving the main result of this section is
that the above formula makes $\TC$ have an additional finiteness property that
theories such that $\TC^-$ and $\TP$ do not enjoy. Roughly speaking the idea is
the following: if in the definition above we were to replace $\can_X-\varphi_X$
with just $\can_X$, then the fiber would be $\Sigma X_{hS^1}$, and this certainly commutes with colimits.  We will see that $\varphi_X$ is close enough to vanishing modulo $p$ to deduce the same conclusion for $\TC(X)/p$.  
As another example of such finiteness phenomena, 
compare for instance \cite{AMN}, where it was shown that the topological
cyclic homology  of a smooth and proper dg category over a finite field 
is a perfect $H \mathbb{Z}_p$-module. 

\begin{example} 
A key source of examples of cyclotomic spectra arises as follows. 
If $R$ is a ring (or, more generally, a structured ring spectrum), then one can form its topological Hochschild homology $\THH(R)$ as an object of $\CycSp$, following \cite[\S III.2]{nikolaus-scholze}. As usual, we will write $\TC(R), \TC^-(R), \TP(R)$ in place of $\TC(\THH(R)), \TC^-(\THH(R)), \TP(\THH(R))$.
\end{example}

We will also use facts from the classical approach to topological cyclic
homology (cf.~\cite[\S6]{BMcyc}) in order to verify a connectivity statement below. In particular, given $X \in \CycSp$ which is $p$-complete, one extracts genuine fixed
point spectra $\TR^{r}(X;p):=X^{C_{p^{r-1}}}$ for each $r\ge1$, which are related by Restriction ($R$), Frobenius ($F$), and Verschiebung ($V$) maps:
\[ R, F: \TR^{r+1}(X;p) \to \TR^{r}(X;p) , \quad V: \TR^{r}(X;p) \to
\TR^{r+1}(X;p).\]
Setting $\TR(X;p):=\varprojlim_R\TR^r(X;p)$, then the topological cyclic
homology of $X$ is classically defined as \begin{equation}
\label{keyformulaclassical}\TC(X)=\mathrm{fib}(\TR(X;p)\xrightarrow{\mathrm{id}-F}\TR(X;p)).\end{equation}
For a comparison between this description of $\TC$ and that given in Definition
\ref{definition_TC_via_NS}, we refer to \cite[\S II.4]{nikolaus-scholze}
(see
also \cite[Rem.~II.1.3]{nikolaus-scholze} for the relationship between
$p$-cyclotomic spectra and cyclotomic spectra).

It is the classical approach to topological cyclic homology in which important structures such as the de
Rham-Witt complex appear \cite{Hesselholt}, and in which the sequence \eqref{isotropysep} below plays a fundamental role.  Using the classical approach we also obtain the following lemma.
In Remark~\ref{rem:TCconnectivity} below, we 
will also give  a proof purely based on the Nikolaus--Scholze approach. 

\begin{lemma}
\label{connectivity}
Suppose $X \in \CycSp_{\geq n}$ is $p$-complete. Then $\TC(X) \in \Sp_{\geq n-1}$. 
\end{lemma}
\begin{proof}
Using formula (\ref{keyformulaclassical}), it suffices to show that $\TR(X;p) \in \Sp_{\geq n}$; using the Milnor sequence, it is therefore enough to prove that $\TR^r(X;p) \in \Sp_{\geq n}$ and that the
Restriction
maps $R:\pi_n\TR^{r+1}(X;p) \to \pi_n\TR^{r}(X;p)$ are surjective for all
$r\ge1$. But this inductively follows from the basic cofiber sequence
(cf.~\cite[Th.~1.2]{Hesselholt})
\begin{equation} \label{isotropysep}
X_{hC_{p^r}} \to \TR^{r+1}(X;p) \stackrel{R}{\to} \TR^r(X;p)
\end{equation}
and the observation that $X_{hC_{p^r}}\in \Sp_{\geq n}$.
\end{proof}

\begin{corollary}
\label{geometricreal}
The functor $ \TC/p:\CycSp_{\geq 0} \to \Sp$,
$X \mapsto \TC(X)/p$ 
commutes with geometric realizations.
\end{corollary}
\begin{proof} A geometric realization of connective spectra agrees with its $n^{th}$ partial geometric realization in degrees $<n$.  Since $X \mapsto \TC(X)/p=\TC(X/p)$ decreases connectivity at most by one by Lemma \ref{connectivity}, to show that $\TC/p$ commutes with geometric realizations it therefore suffices to show that it commutes with $n^{th}$ partial geometric realizations for all $n$.  But these are finite colimits, and are thus preserved by $\TC/p$ since the latter is an exact functor.\end{proof}

\subsection{Cocontinuity of $\TC/p$}\label{subsection_cocontinuity}

In this subsection we prove that $\TC/p$ commutes with all colimits:

\begin{theorem} 
\label{TCcommutes}
The functor $\TC/p: \CycSp_{\geq 0} \to \Sp$, commutes with all colimits. 
\end{theorem} 

Since $\TC/p : \CycSp \to \Sp$ is an exact functor between stable
$\infty$-categories, it automatically commutes with finite colimits; thus the
essence of the above theorem is that it commutes with filtered colimits when
restricted to $\CycSp_{\geq 0 }$.

\begin{remark} 
\label{TCinvlim}
There is also a dual assertion which is significantly easier. 
Suppose we have a  tower  $\dots \to X_3 \to
X_2 \to X_1$ in $\CycSp_{\geq 0}$. 
Then the inverse limit of the tower $\left\{X_i\right\}$ exists in $\CycSp$
(in fact, in $\CycSp_{\geq -1}$) and
is preserved by the forgetful functor $\CycSp \to \Sp$. 
Indeed, we have $(\varprojlim X_i)^{tC_p} \simeq \varprojlim X_i^{tC_p}$ because
the $X_i$'s are uniformly bounded-below, and taking $C_p$-homotopy orbits
behaves as a finite colimit in any range of degrees. 
Then we can appeal to the description of $\CycSp$ as a lax equalizer and
\cite[Prop. II.1.5]{nikolaus-scholze}. 
In particular, it follows that if $X = \varprojlim X_i$ on the level of underlying spectra, then $\TC(X) \simeq
\varprojlim \TC(X_i)$. 
\end{remark} 

The proof of Theorem \ref{TCcommutes} will proceed by reduction to the case of
cyclotomic spectra over the ``trivial'' cyclotomic spectrum $ H\bb
 F_p^\sub{triv}$. We therefore begin by recalling the definition of such
 trivial cyclotomic spectra, before studying $H \bb F_p^\sub{triv}$ and its modules in further detail.

\begin{definition}\label{cons_trivial}
We have a presentably symmetric monoidal $\infty$-category $\CycSp$, so it receives a unique symmetric monoidal,
cocontinuous functor
\[  \Sp \to \CycSp,\quad X \mapsto X^{\mathrm{triv}}  \]
whose right adjoint is the functor $\TC: \CycSp \to \Sp$.

More explicitly, this functor $X \mapsto X^{\mathrm{triv}}$ can be identified as follows: $X^\sub{triv}$ has underlying spectrum $X$ equipped with the
trivial $S^1$-action; there is therefore a resulting $S^1$-equivariant map $X \to X^{hC_p} =
F(BC_{p+}, X)$, and the cyclotomic Frobenius at $p$ for $X^\sub{triv}$ is the
composition $\varphi_{X,p}:X\to X^{hC_p} \to X^{tC_p}$. (Notation: when $X$ is a
$p$-complete spectrum for a particular prime number $p$, we will tend to drop
the $p$ from $\varphi_{X,p}$.) Compare the discussion on \cite[p.
126]{nikolaus-scholze}. 

For example, taking $X$ to be the sphere spectrum $S^0 \in
\Sp$ obtains the {\em cyclotomic sphere spectrum}, which is moreover the unit $\mathbf{1} \in \CycSp$.
\end{definition}

We will need some results about the cyclotomic spectrum
$H\mathbb{F}_p^{\mathrm{triv}}$. Firstly, recall that the Tate cohomology ring
$\pi_*(H\mathbb{F}_p^{tC_p})=\hat H^{-*}(C_p;\bb F_p)$ is 
\begin{itemize}
\item ($p>2$) the tensor product of an exterior algebra on a degree $-1$ class with a Laurent polynomial algebra on a degree $-2$ class;
\item ($p=2$) a Laurent polynomial algebra on a degree $-1$ class.
\end{itemize}
Calculating the $S^1$-homotopy fixed points and Tate construction of
$\mathbb{F}_p$, we find that
\[ \TC^-_*( H\mathbb{F}_p^{\mathrm{triv}}) = \mathbb{F}_p[x], 
\quad \quad
\TP_*( H\mathbb{F}_p^{\mathrm{triv}}) = \mathbb{F}_p[x^{\pm 1}], \quad \quad
\text{where}  \ |x|
= -2.
\]
The canonical map $\TC^-( H\mathbb{F}_p^{\mathrm{triv}}) \to \TP(
H\mathbb{F}_p^{\mathrm{triv}})$ carries $x$ to $x$ (i.e., it is the localization inverting $x$).

The cyclotomic Frobenius $\varphi_{H \bb F_p} = \varphi_{H
\mathbb{F}_p^{\mathrm{triv}}}: H\mathbb{F}_p \to
H\mathbb{F}_p^{tC_p}$ is an isomorphism on $\pi_0$. An essential point is
that the induced
Frobenius $\varphi_{H \mathbb{F}_p}^{hS^1}:\TC^-(
H\mathbb{F}_p^{\mathrm{triv}}) \to \TP(H \mathbb{F}_p^{\mathrm{triv}})$ (which we will abusively also denote by $\varphi_{H\mathbb{F}_p}$ to avoid clutter, when the source and target are clear) kills the class $x$, as we now check.

\begin{lemma} 
\label{Frobvanish}
The Frobenius map $\varphi_{H \mathbb{F}_p}:\TC^-_*(
H \mathbb{F}_p^{\mathrm{triv}}) \to \TP_*(H \mathbb{F}_p^{\mathrm{triv}})$ annihilates $x$. 
\end{lemma} 
\begin{proof} 
There is a commutative diagram
of spectra
\[ \xymatrix{
\TC^-(H \mathbb{F}_p^{\mathrm{triv}})  \ar[r]^-{\varphi_{H
\mathbb{F}_p}}  \ar[d]&  
\TP(H \mathbb{F}_p^{\mathrm{triv}}) \ar[d]  \\
H\mathbb{F}_p \ar[r]_{\varphi_{H \bb F_p}} &  H \mathbb{F}_p^{tC_p}
}\]
where the top row is obtained by taking $S^1$-homotopy fixed points on the bottom
row via the Tate orbit lemma, \cite[Lem.~II.4.2]{nikolaus-scholze}. The right vertical map
$\TP(H\bb F_p^\sub{triv})\to
H\mathbb{F}_p^{tC_p}$ is an isomorphism
in degree $-2$: indeed, the degree $-2$ generator is invertible in the source by the above calculation of $\TP(H \bb
F_p^\sub{triv})$, and therefore necessarily maps to a generator of $\pi_{-2}(H\mathbb{F}_p^{tC_p})$ by multiplicativity.  Going both ways around the diagram and noting that of course $\pi_{-2}(H\bb F_p)=0$, we find that
$\varphi_{H \mathbb{F}_p}$ must be zero on $\pi_{-2}(
\TC^-(H \mathbb{F}_p^{\mathrm{triv}}))$. 
\end{proof} 

The previous proof is modeled on the analysis of the
cyclotomic spectrum $\THH(\mathbb{F}_p)$  \cite[Sec.~IV-4]{nikolaus-scholze}.
In fact, noting that $\TC(\mathbb{F}_p)$ receives a map $H\mathbb{Z}_p \to \TC(\mathbb{F}_p)$
(the source is the connective cover of the target)
we obtain a map of cyclotomic spectra $H \mathbb{Z}_p^{\mathrm{triv}} \to \THH(\mathbb{F}_p)$, which can be used to recover Lemma \ref{Frobvanish} from the more precise assertions about the cyclotomic spectrum $\THH(\mathbb{F}_p)$.

We next need to recall some facts about circle actions on
Eilenberg--MacLane spectra. 
Cf.~\cite[Sec.~2.2]{Loday} for the classical analogs in the theory of cyclic
homology. 
\begin{lemma} 
\label{cofiberofx}
Let $X \in \md_{H \mathbb{Z}}(\Sp^{BS^1})$ be an $H \mathbb{Z}$-module spectrum
equipped with an $S^1$-action. Then: 
\begin{enumerate}
\item  
There is a functorial cofiber sequence
of $H\mathbb{Z}$-module spectra
\[ \Sigma^{-2} X^{hS^1} \stackrel{x}{\To} X^{hS^1} \To X,  \]
where the first map is given by multiplication by the generator $x \in
\pi_{-2}(H \mathbb{Z}^{hS^1}) = H^2( \mathbb{CP}^\infty; \mathbb{Z})$. 
\item
\label{hS1cofib}
There is a functorial cofiber sequence of $H \mathbb{Z}$-modules
\[ X \to  X_{hS^1} \to \Sigma^2  X_{hS^1} . \]
\item There is moreover a functorial cofiber sequence
\[ \Sigma^{-2} X^{hS^1} \xto{x \mathrm{can}_X} X^{tS^1} \To
X_{hS^1},  \]
where $\mathrm{can}_X: X^{hS^1} \to X^{tS^1}$ is the canonical map. 

\end{enumerate}

\end{lemma} 
\begin{proof} 
Assertion (1) is \cite[Lemma~IV.4.12]{nikolaus-scholze}, in view
of the cofiber sequence 
of $H\mathbb{Z}^{hS^1}$-modules $\Sigma^{-2} H\mathbb{Z}^{hS^1}
\xrightarrow{x} H\mathbb{Z}^{hS^1} \to H\mathbb{Z}$. 
For assertion (2), by \emph{loc.~cit.}, we have
functorial equivalences $X^{tS^1} \simeq X^{hS^1} \otimes_{H\mathbb{Z}^{hS^1}}
H\mathbb{Z}^{tS^1} = X^{hS^1}[1/x]$. 
The fiber sequences 
$\Sigma (H\mathbb{Z})_{hS^1} \to H\mathbb{Z}^{hS^1} \to H\mathbb{Z}^{tS^1}$
and 
$\Sigma X_{hS^1} \to X^{hS^1} \to X^{tS^1}$
make 
$(H\mathbb{Z})_{hS^1}$ into an $H\mathbb{Z}^{hS^1} $-module and yield a
natural equivalence 
$X_{hS^1} \simeq X^{hS^1} \otimes_{H\mathbb{Z}^{hS^1}} (H\mathbb{Z})_{hS^1}$. 
Using the fiber sequence of $H\mathbb{Z}^{hS^1}$-modules $H\mathbb{Z} \to
(H\mathbb{Z})_{hS^1} \xrightarrow{x} \Sigma^2 (H \mathbb{Z})_{hS^1}$, we obtain
the assertion (2). 
Assertion (3) follows from assertion (1) and  from the cofiber sequence $X^{hS^1}
\stackrel{\mathrm{can}_X}{\to} X^{tS^1} \to \Sigma^2 X_{hS^1}$, where we use 
that multiplication by $x$ is an equivalence on $X^{tS^1}$. 
\end{proof}

\begin{proposition} 
\label{functorialfibseq}
For $X \in \md_{H \mathbb{F}_p^{\mathrm{triv}}}( \CycSp)$, there exists a functorial fiber
sequence
\begin{equation} \TC(X) \To X \To X_{hS^1}.  \end{equation}
\end{proposition} 
\begin{proof} 
We show first that, for $X \in \md_{H \mathbb{F}_p^{\mathrm{triv}}}(\CycSp) $, 
the composition
\begin{equation} \label{th} \Sigma^{-2} \TC^-(X) \xTo{x} \TC^-(X)
\xrightarrow{\varphi_X} \TP(X)  \end{equation}
is functorially nullhomotopic. 
In fact, since
$X$ is an $H \mathbb{F}_p^{\mathrm{triv}}$-module, we find that
$\TC^-(X)$ is an $\TC^-( H \mathbb{F}_p^{\mathrm{triv}})$-module, $\TP(X)$ is a
$\TP(H \mathbb{F}_p^{\mathrm{triv}})$-module, and 
the map $\varphi_X$ is $\varphi_{H \bb F_p}$-linear.
We obtain a resulting functorial commutative diagram,
where the base-changes are along $\varphi_{ H \bb F_p}:
\TC^-(H \mathbb{F}_p^{\mathrm{triv}}) \to
\TP( H \mathbb{F}_p^{\mathrm{triv}})$,
\begin{equation} \label{cycdiag}
\xymatrix@C=1cm{
\Sigma^{-2} \TC^-(X) \ar[r]^x \ar[d] &  \TC^-(X) \ar[d]
\ar[rd]^{\varphi_X} \\
\Sigma^{-2} \TC^-(X) \otimes_{\TC^-(H \mathbb{F}_p^{\mathrm{triv}})}
\TP(H \mathbb{F}_p^{\mathrm{triv}})
\ar[r]_-{\varphi_{H \bb F_p}(x)} & \TC^-(X)
\otimes_{\TC^-(H\mathbb{F}_p^{\mathrm{triv}})}
\TP(H \mathbb{F}_p^{\mathrm{triv}}) \ar[r] & 
 \TP(X)  
} .\end{equation}
 However $\varphi_{H \mathbb{F}_p}(x)$ vanishes, by
Lemma~\ref{Frobvanish}, and so the composite 
\eqref{th} is functorially nullhomotopic as desired. 

It now follows that there is a functorial commutative diagram
\begin{equation} \label{cyc2} \xymatrix{
\Sigma^{-2} \TC^-(X) \ar[d]_{x} \ar[r]^{\mathrm{id}} &\Sigma^{-2}
\TC^-(X)\ar[d]^{x \mathrm{can}_X} \\
\TC^-(X) \ar[r]_{\mathrm{can}_X-\varphi_X} &  \TP(X)  
} \end{equation} 
in which the fiber of the bottom arrow is $\TC(X)$ by formula \eqref{keyformula}. Therefore $\TC(X)$ is also  the fiber of 
the natural map $X \to X_{hS^1}$ between the cofibers of the
vertical arrows in the above diagram, where we identify the cofibers using
Lemma~\ref{cofiberofx}.
\end{proof} 

See also \cite[Prop. IV.3.4]{nikolaus-scholze} for a closely related result for
$p$-cyclotomic spectra with Frobenius lifts. 

\begin{corollary} 
\label{Fpcommutecolimit}
The functor 
 $\md_{H \mathbb{F}_p^{\mathrm{triv}}}(\CycSp) \to \Sp$ given by 
$X \mapsto \TC(X)$ 
  commutes with
all colimits. 
That is, the unit of the presentably symmetric monoidal stable $\infty$-category
 $\md_{H \mathbb{F}_p^{\mathrm{triv}}}(\CycSp)$ is compact.
 \end{corollary} 
\begin{proof}
Since passage to homotopy orbits commutes with all colimits, this follows from Proposition \ref{functorialfibseq}.
\end{proof}

\begin{remark} 
\label{rem:TCconnectivity}
Let $X \in \CycSp_{ \geq 0}$ be $p$-complete. 
Using the functorial fiber sequence of Proposition~\ref{functorialfibseq}, we
can give an independent proof 
of Lemma~\ref{connectivity} (and thus of Corollary~\ref{geometricreal}) which is independent of the classical approach to
$\TC$ (i.e., based entirely on the formula \eqref{keyformula}). 
Namely, we argue that $\TC(X) \in \Sp_{\geq -1}$. 

First, suppose that $X$ is an $H \mathbb{F}_p^{\mathrm{triv}}$-module. In this
case, we find that $\TC(X) \in \Sp_{\geq -1}$ thanks to
Proposition~\ref{functorialfibseq}. More generally, if $X$ is a
$p$-complete $H \mathbb{Z}^{\mathrm{triv}}$-module, then we find that $X/p$ is an
$H\mathbb{F}_p^{\mathrm{triv}}$-module so that $\TC(X/p) \simeq \TC(X)/p$ belongs
to $\Sp_{\geq -1}$; by $p$-completeness we deduce that $\TC(X) \in \Sp_{\geq
-1}$.

In general, let $X \in \CycSp_{\geq 0}$ be an arbitrary $p$-complete object. 
Then, we consider the Adams tower for $X$ with respect to the commutative
algebra object $H\mathbb{Z}^{\mathrm{triv}} \in \CycSp$, i.e., the tower for
the totalization of the cosimplicial object $X \otimes
(H\mathbb{Z}^{\mathrm{triv}})^{\otimes \bullet + 1}$ (cf.~\cite[Sec.~2.1]{MNN17} for a reference in this generality). 
That is, we consider the tower $\left\{X_n\right\} \in \CycSp$ defined
inductively such that
$X_0 = X$ and $X_{n+1}$ is the fiber of the map $X_{n} \to X_n \otimes
H\mathbb{Z}^{\mathrm{triv}}$. 
One sees easily that $X_n \in \CycSp_{\geq n}$, and that $X_n$ is $p$-complete since $\pi_i(S^0)$ is finite for $i>0$.

Since $X_{n}/X_{n+1}$ is a $H\mathbb{Z}^{\mathrm{triv}}$-module, it follows from
the previous analysis that $\TC(X_n/X_{n+1}) \in \Sp_{\geq (n-1)}$ for each $n$. 
Inductively, 
it follows now that $\TC(X/X_{n+1}) = \TC(X_0/X_{n+1}) \in \Sp_{\geq -1}$ for
each $n$ and that the maps $\TC(X/X_{n+1}) \to \TC(X/X_{n})$ are surjective
on $\pi_{-1}$. 
Since $X \to X/X_{n+1}$ is an equivalence in degrees $\leq n$, it follows (as in \cite[Lem.~I.2.6]{nikolaus-scholze}) that
\[ X^{hS^1} \simeq \varprojlim (X/X_n )^{hS^1}, \quad 
X^{tS^1} \simeq \varprojlim (X/X_n)^{tS^1}
\]
and therefore by \eqref{keyformula}
\[ \TC(X) \simeq \varprojlim \TC(X/X_n).  \]
Using the Milnor exact sequence 
and the connectivity properties of $\TC(X/X_{n+1})$, 
one concludes that $\TC(X) \in \Sp_{\geq -1}$. 
\end{remark} 

We are now prepared to prove the main result of the subsection.

\begin{proof}[Proof of Theorem \ref{TCcommutes}]
Consider an $X \in \CycSp_{\geq 0}$ which is $p$-complete.  Since $\TC(X\otimes (-)^{\mathrm{triv}})$ is an exact functor from spectra to spectra with value $\TC(X)$ on $S^0$, we get a natural transformation
\[ \TC(X) \otimes (-) \to \TC(X \otimes (-)^{\mathrm{triv}}),\]
uniquely characterized by being an equivalence on $S^0$, hence on any finite spectrum.\footnote{This is a standard ``assembly map".  To construct it rigorously, note that left Kan extension of the restriction of the right hand side to finite spectra is uniquely characterized by its value on $S^0$, by the universal property of spectra among presentable stable $\infty$-categories, and therefore identifies with the left hand side.} By Corollary~\ref{geometricreal}, it is also an equivalence on any geometric realization of finite connective spectra, i.e.,
for any connective spectrum whose homology groups are finitely generated, e.g.,
$H \bb F_p$.
So we have shown that, for any $X \in
\CycSp_{\geq 0}$ which is $p$-complete, there is a natural equivalence
$$\TC(X) \otimes H\mathbb{F}_p \simeq \TC(X \otimes
H\mathbb{F}_p^{\mathrm{triv}}).$$

Replacing $X$ with $X/p$, the same conclusion follows for an arbitrary
$X \in \CycSp_{\geq 0}$.  Then from Corollary~\ref{Fpcommutecolimit} we deduce that the functor $X
\mapsto \TC(X) \otimes H\mathbb{F}_p$ commutes
with colimits in $\CycSp_{\geq 0}$. 
Hence, given a diagram $I\to\CycSp_{\geq
0}$, the induced map of spectra
\begin{equation} \label{TCcolim} \varinjlim_I \TC(X_i) \to \TC(X) \end{equation} 
becomes an equivalence after smashing with $H\mathbb{F}_p$. 
Since both sides are bounded below mod $p$ by Lemma \ref{connectivity}, it
follows that \eqref{TCcolim}
is an equivalence after smashing with $S^0/p$.
\end{proof}

The following consequence concerning the topological cyclic homology of rings is of particular interest. Let $\mathrm{Alg}(\Sp_{\geq 0})$ be the $\infty$-category of connective
associative ring spectra.

\begin{corollary} 
\label{TCring}
The functor $\mathrm{Alg}(\Sp_{\geq 0})\to\Sp$, $R \mapsto \TC(R)/p$ commutes with
sifted colimits. 
\end{corollary} 
\begin{proof} 
This follows from Theorem~\ref{TCcommutes} since $\TC(R) = \TC( \THH(R))$, where we note that
the topological Hochschild homology functor
$\THH: \mathrm{Alg}(\Sp_{\geq 0}) \to \CycSp_{\geq 0}$ commutes
with sifted colimits (as geometric realizations and tensor products do).
\end{proof}

Certain special cases of Corollary~\ref{TCring} appear in the literature (or
can be extracted from it). 

\begin{example} 
For commutative $\mathbb{F}_p$-algebras, Corollary~\ref{TCring} can be deduced
using the description of Hesselholt \cite{Hesselholt} of $\THH$ of smooth
$\mathbb{F}_p$-algebras in terms of the de Rham--Witt complex; we detail this in
subsection~\ref{dRWapproach} below.
\end{example}

\begin{example} 
In the case of square-zero extensions, 
the verification (called the \emph{$p$-limit axiom})  of
Corollary~\ref{TCring} plays an important role in
the proof that relative $K$-theory and $\TC$ with finite coefficients agree for nilimmersions of rings
in McCarthy \cite{mccarthy}. 
\end{example} 

\begin{example} 
Let $M$ be an $\mathbb{E}_1$-monoid in the $\infty$-category of spaces, so that
one can form the spherical monoid ring $\Sigma^\infty_+ M$. 
The formula for $\TC$ of spherical 
monoid rings \cite[Lemma IV.3.1]{nikolaus-scholze} (preceded by the formula for
spherical group rings appearing already in \cite[Eq. (0.3)]{BHM})
shows that for any such $R$, 
the cyclotomic structure map $\THH(R) \to \THH(R)^{tC_p}$ factors
$S^1$-equivariantly through
$\THH(R)^{hC_p}$, which is called a \emph{Frobenius lift} in
\cite{nikolaus-scholze}. For cyclotomic spectrum $X$ with a Frobenius lift, the construction  of
($p$-complete) $\TC$ simplifies: one has a fiber square 
\cite[Prop. IV-3.4]{nikolaus-scholze} relating $\TC(X)$ to $\Sigma
X_{hS^1}$ and two copies of $X$. It follows from this that the construction  
$M \mapsto \TC( \Sigma^\infty_+ M)/p$ commutes with filtered colimits in $M$. 
Since the free associative algebra is a spherical monoid ring, and since $\TC$
is already known to commute with geometric realizations, one can also deduce
Corollary~\ref{TCring} in this way.
\end{example} 
\subsection{Further finiteness of $\TC/p$}
Here we use the material of subsection \ref{subsection_cocontinuity} to deduce further finiteness properties of topological cyclic homology.

Firstly we show that, in any given range, $\TC/p$ can be approximated well by
functors finitely built from taking $S^1$-homotopy orbits. Let $\fun(\CycSp_{\geq 0}, \Sp)$ denote the $\infty$-category of functors $\CycSp_{\geq 0} \to \Sp$. 
For our purposes below, we will need this strengthening and not only
Theorem~\ref{TCcommutes}. 

\begin{proposition} 
\label{strong:finiteness}
For any given integer $n$, there exists a functor $F \in \fun_{}(\CycSp_{\geq 0}, \Sp)$ with the following properties:
\begin{enumerate}
\item $F$ belongs to the thick subcategory  of $\fun_{}(\CycSp_{\geq 0}, \Sp)$
generated by the functor $X \mapsto (X\otimes H\mathbb{F}_p^\sub{triv})_{hS^1}$;
\item There exists an equivalence $\tau_{\leq n} F(X) \simeq
\tau_{\leq n}(\TC(X)/p)$. \end{enumerate}
\end{proposition}
\begin{proof}
We claim that for each $n$, the functor $X \mapsto \TC(X\otimes \tau_{\leq
n+1}(S^0/p)^\sub{triv})$ has the desired properties. Indeed, it  belongs to the thick subcategory 
generated by  $X \mapsto (X\otimes H\mathbb{F}_p^\sub{triv})_{hS^1}$ 
thanks to Proposition~\ref{functorialfibseq} and
assertion~(\ref{hS1cofib}) of Lemma~\ref{cofiberofx}, while the connectivity assertion follows from Lemma~\ref{connectivity}.
\end{proof}

The second finiteness result concerns the ``pro'' structure of topological
cyclic homology and will be used later in the paper. Given $X \in \CycSp_{\geq
0}$, the classical approach to topological cyclic homology (cf.~subsection \ref{subsection_reminders}) involves the spectra
\[ \TC^r(X;p) = \mathrm{fib}(\TR^r(X;p) \xto{R-F} \TR^{r-1}(X;p)) \]
for all $r\ge1$. The system $\left\{\TC^r(X;p)\right\}_{r \geq 1}$ naturally
forms a tower of spectra under the Restriction (or, equivalently, Frobenius) maps, and $\TC(X;p) = \varprojlim_r \TC^r(X;p)$. 
Note that $\TC^r(X;p) \in \Sp_{\geq -1}$ for all $r\ge1$ (by the proof of Lemma \ref{connectivity}), and that each functor $\CycSp_{\geq 0}\to\Sp$, $X
\mapsto \TC^r(X;p)$ commutes with all colimits (by induction using (\ref{isotropysep})).

Thus the failure of $\TC$ to commute with filtered colimits arises from the infinite
tower $\left\{\TC^r(-;p)\right\}_r$. 
We will now prove a  restatement of 
Theorem~\ref{TCcommutes} to the effect that this tower is pro-constant modulo
$p$. 
First we need a couple of general lemmas on inverse systems. 

\begin{lemma} 
\label{invlimmzero}
Let $\mathcal{C}$ be a category admitting countable coproducts, and consider a tower of functors $\mathcal{C}\rightarrow \operatorname{Ab}$, say $\ldots \rightarrow F_r\rightarrow F_{r-1}\rightarrow \ldots \rightarrow F_1$.  Suppose that:
\begin{enumerate}
\item For each $X\in\mathcal{C}$ we have $\varprojlim_r F_r(X) = \varprojlim^1_r F_r(X) = 0$.
\item For each $r\in\mathbb{N}$, the functor $F_r$ commutes with countable coproducts.
\end{enumerate}
Then the tower $\{F_r\}_{r \geq 1}$ is pro-zero, i.e.\ it is zero as
an object of $\operatorname{Pro}(\operatorname{Fun}(\mathcal{C},\operatorname{Ab}))$; or, equivalently, for all $r\in\mathbb{N}$ there is an $s>r$ such that the morphism $F_s\rightarrow F_r$ is $0$.
\end{lemma} 
\begin{proof} 
First we show the weaker claim that for all $X\in\mathcal{C}$, the tower
$\{F_r(X)\}_r$ of abelian groups is pro-zero.  Set $M_r = F_r(X)$ for brevity.
From the hypothesis we find that an infinite direct sum of copies of the tower
$\{M_r\}_r$ has vanishing $\varprojlim^1$; thus \cite[Cor.~6]{Emmanouil} implies
that $\{M_r\}_r$ is Mittag--Leffler, i.e., for each $r$, the descending sequence $\{\mathrm{im}(M_s \to M_r)\}_{s \geq r}$ of submodules of $M_r$ stabilizes.  If this stable submodule were nonzero for some $r$, we would deduce the existence of a nonzero element of $\varprojlim_s M_s$, contradicting the hypothesis.  Thus the stable value is $0$ for all $r$, which exactly means that $\{M_r\}_r$ is pro-zero.

Now, suppose the claim of the lemma does not hold, i.e., that $\{F_r\}_r$ is not pro-zero.  Then for each $r\in\mathbb{N}$ and $s>r$, we can find an $Y\in\mathcal{C}$ such that $F_s(Y)\rightarrow F_r(Y)$ is nonzero.  Let $X$ denote the coproduct over all pairs $s>r$ of a choice of such a $Y$.  Then for every $s>r$ the map $F_s(X)\rightarrow F_r(X)$ is nonzero, so we deduce that the tower $\{F_r(X)\}$ is not pro-zero, in contradiction to what was established above.
\end{proof}

\begin{lemma} 
\label{invlim02}
Let $\mathcal{C}$ be an $\infty$-category admitting countable coproducts, and consider a tower of functors $\mathcal{C}\rightarrow \operatorname{Sp}$, say $\ldots \rightarrow F_r\rightarrow F_{r-1}\ldots \rightarrow F_1$.  Suppose that:
\begin{enumerate}
\item For each $X\in\mathcal{C}$ we have $\varprojlim_r F_r(X) = 0$.
\item For each $r\in\mathbb{N}$, the functor $F_r$ commutes with countable coproducts.
\item The homotopy groups of each $F_r(X)$ are zero outside some fixed interval $[a,b]$, independent of $X$ and $r$.
\end{enumerate}
Then the tower $\{F_r\}_{r \geq 1}$ is pro-zero, i.e.\ it is zero as
an object of $\operatorname{Pro}(\operatorname{Fun}(\mathcal{C},\operatorname{Sp}))$.\end{lemma} 
\begin{proof} 
The Milnor sequence and induction up the Postnikov tower show that (1) implies
$\varprojlim_r \pi_i F_r(X) = \varprojlim^1_r \pi_i F_r(X) = 0$ for all $i\in\mathbb{Z}$.  Since $\pi_i$ commutes with coproducts, we can apply the previous lemma to $\{\pi_i F_r\}_r$ and conclude that each tower $\{\pi_i F_r\}_r$ is pro-zero.  Thus each Postnikov section of $\{F_r\}_r$ is pro-zero, and therefore so is $\{F_r\}_r$ itself, by devissage up the (finite) Postnikov tower.
\end{proof} 

\begin{proposition} \label{prop_pro_constant}
Fix an integer $k$. Then the tower of objects of $\operatorname{Fun}(\CycSp_{\geq 0},\operatorname{Sp})$ given by $ \left\{ \tau_{\leq k}(\TC^r(-)/p)\right\}_{r}$ is pro-constant with value
$\tau_{\leq k}(\TC(-)/p)$. 
\end{proposition} 
\begin{proof} 
Setting $F_{r}(X) \stackrel{\mathrm{def}}{=} \tau_{\leq k+1}\mathrm{cofib}( \TC(X)/p \to
\TC^r(X)/p)$, we will show that the tower $\{F_{r}(-)\}_r$ is pro-zero.
As $k$ varies, this shows that the map 
$\left\{\TC(-)/p\right\} \to \left\{\TC^r(-)/p\right\}_r$ is a pro-isomorphism
on each homotopy group in degrees $\leq k$ (in the pro-category of functors $\CycSp_{\geq 0} \to \mathrm{Ab}$), which
is enough to conclude. 
Certainly $\varprojlim_r F_{r}(X) = 0$, and $F_r$ commutes with arbitrary
coproducts by Theorem~\ref{TCcommutes}, so the desired claim follows from
Lemma~\ref{invlim02}. 
\end{proof}

\subsection{$\TC/p$ via de Rham--Witt}

\label{dRWapproach}
In the previous subsections, we proved that $\TC/p$ commutes with filtered
colimits on the $\infty$-category $\CycSp_{\geq 0}$. In practice, one is
usually interested
in $\TC$ of rings, i.e., $\TC$ of the cyclotomic spectra $\THH(R)$ for rings
$R$. 

The purpose of this subsection is twofold. The first is to review explicitly the apparatus
of $\THH, \TC$, and the de Rham--Witt complex in the case of smooth
$\mathbb{F}_p$-algebras, which we will need in the sequel. The
second is to use this formalism to 
present a more direct approach to proving that $\TC/p$ commutes with
filtered  colimits on the category of (discrete, commutative) $\bb
F_p$-algebras. While this is much weaker than what we
have already established (and will be insufficient for some of our
applications), it can be proved without the new approach to cyclotomic spectra of
\cite{nikolaus-scholze}, and is historically the first case which was known to some experts.

We begin with a review of the relevant algebraic concepts. 
Let $R$ be an $\mathbb{F}_p$-algebra (always assumed commutative).  We let $\Omega^n_R$ denote the $n$-forms of $R$ (over $\mathbb{F}_p$) and let $\Omega_R^{\ast}$ denote the algebraic de Rham complex of $R$. 

\begin{definition} 
\label{cartier}
The \emph{inverse Cartier operator} $C^{-1}: \Omega^n_R \to
H^n(\Omega^{\ast}_R) \subset \Omega^n_R/d
\Omega^{n-1}_R$ is the multiplicative operator uniquely characterized by the formulas
\( C^{-1}(a) = a^p, \quad  C^{-1}(db) = b^{p-1} db,   \)
for $a, b \in R$. See \cite{Cartier1957}, also \cite[Prop.~3.3.4]{BLM}, for a construction of this map. 
\end{definition} 
Note that the construction $C^{-1}$ is not well-defined as an operator $\Omega^n_R \to
\Omega^n_R$;
it is only well-defined modulo boundaries. It has image in the
cohomology of the de Rham complex; if we further assume that $R$ is ind-smooth,
i.e., can be written as a filtered colimit of smooth $\mathbb{F}_p$-algebras\footnote{Since smooth algebras are finitely presented, it follows that the category of ind-smooth algebras indeed identifies with $\operatorname{Ind}$ of the category of smooth algebras.}, then $C^{-1}$ provides a natural (Frobenius
semi-linear) isomorphism 
between $\Omega^n_R$ and $H^n( \Omega_R^{\ast})$ by the classical
\emph{Cartier isomorphism}, cf.~\cite[Th.~7.2]{Katz70}. 

\begin{definition}[The de Rham--Witt complex] 
\label{dRW1}
Let $R$ be an $\bb F_p$-algebra. For $r\ge 1$, we denote as usual by
\[W_r(R)\xto{d}W_r\Omega^1_R\xto{d}W_r\Omega^2_R\xto{d}\cdots\] the classical de
Rham--Witt complex (more precisely, differential graded algebras) of
Bloch--Deligne--Illusie \cite{illusie-derham-witt}; this is the usual de Rham complex $\Omega^\bullet_R$ in the case $r=1$. We recall that the individual de Rham--Witt groups are equipped with Restriction and Frobenius maps $R,F:W_r\Omega^n_R\to W_{r-1}\Omega^n_R$.
We let $W \Omega_R^{\ast}$ denote the inverse limit of the tower
$\left\{W_r \Omega^{\ast}_R\right\}$ under the Restriction maps. 
\end{definition}

The operator $F$ acts as a type of divided Frobenius. 
For example, one has a commutative diagram
\begin{equation} 
\xymatrix{
W\Omega_R^{\ast} \ar[d] \ar[r]^F & W \Omega_R^{\ast} \ar[d] \\
\Omega_R^{\ast} \ar[r]^-{C^{-1}} &  \Omega_R^{\ast}/d \Omega^{\ast-1}_R.
}
\end{equation} 
Here the vertical maps are induced by the natural projections to the first term of the inverse limit defining $W\Omega_R^{\ast}$.

In the approach to $\THH$ via equivariant stable homotopy theory, the 
de Rham--Witt complex plays a fundamental role thanks to the following result of
Hesselholt, a version of the classical Hochschild--Kostant--Rosenberg (HKR)
theorem for Hochschild homology. 

\begin{theorem}[Hesselholt \cite{Hesselholt}] 
Let $R$ be a smooth $\mathbb{F}_p$-algebra. Then, for each $s\ge1$, we have an 
isomorphism of graded rings
\begin{equation} \label{HHKR1}  \TR^s_*(R; p) \simeq W_s \Omega^{\ast}_R \otimes_{\mathbb{Z}/p^s}
\mathbb{Z}/p^s [ \sigma_s], \quad |\sigma_s| = 2.  \end{equation}
The Restriction maps
are determined by the Restriction maps on the de Rham--Witt complex and send 
$\sigma_s \mapsto p  \sigma_{s-1}$. 
Therefore, we have an isomorphism
\begin{equation}  \label{HHKR2} \TR_*(R; p) \simeq W \Omega^{\ast}_R,  \end{equation}
and the $F$ map on $\TR(R; p)$ induces the $F$ map on $W \Omega_R^{\ast}$. 
\end{theorem}

We can now prove the finiteness property of $\TC/p$ for $\mathbb{F}_p$-algebras. 
While the result is not logically necessary now, the diagram \eqref{thisdiag}
will play a role in the sequel.

\begin{proposition}
The functor $\op{CAlg}_{/\bb F_p}\to\Sp$, $R\mapsto \TC(R)/p$ commutes with filtered colimits.
\end{proposition}

\begin{proof}
Let $R=\op{colim}_iR_i$ be a filtered colimit of $\bb F_p$-algebras. By functorially picking simplicial resolutions of all the terms by ind-smooth $\bb F_p$-algebras, and recalling that $\TC/p$ commutes with geometric realizations of rings, we reduce to the case in which all the $R_i$ (hence also $R$ itself) are ind-smooth over $\bb F_p$.

Next we use Hesselholt's HKR theorem. Since each $\TR^s(-;p)$
commutes with filtered colimits (induction on (\ref{keyformulaclassical})), as
does $W_s\Omega^n_{-}$, the formula \eqref{HHKR1} remains valid for ind-smooth $\bb
F_p$-algebras. By taking the inverse limit over $s$,  we deduce that
$\pi_*(\TR(R;p)) \simeq W\Omega_R^*$ for any ind-smooth $\bb F_p$-algebra $R$,
i.e., \eqref{HHKR2} is valid in the ind-smooth case as well.
Note here that each $W\Omega_R^n$ is $p$-torsion-free, by a basic property of
the de Rham--Witt complex for smooth $\mathbb{F}_p$-algebras \cite[Cor.~I.3.6]{illusie-derham-witt} which can easily be extended to the ind-smooth case as it
derives from the stronger assertion that each Restriction map $W_s \Omega_R^{\ast} \to
W_{s-1} \Omega_R^{\ast}$ annihilates every element $x \in W_s \Omega_R^{\ast}$
with $px  = 0$.
As a result, we also conclude that 
$\pi_* ( \TR(R; p)/p) \simeq W \Omega_R^{\ast}/p$.  

Taking fixed points for the Frobenius, 
we find that $\pi_n(\TC(R)/p)$ fits into an exact sequence
\[W\Omega^{n+1}_R/p\xto{F-1}W\Omega^{n+1}_R/p\To \pi_n(\TC(R)/p)\To
W\Omega^{n}_R/p\xto{F-1}W\Omega^{n}_R/p,\] and similarly for each $R_i$. To complete the proof it is enough to show, for each $n\ge0$, that the kernel and cokernel of $F-1:W\Omega^{n}_-/p\to W\Omega^{n}_-/p$ commutes with our filtered colimit $R=\op{colim}_iR_i$; we stress that $W\Omega^{n}_-/p$ does not commute with filtered colimits.

For any ind-smooth $\bb F_p$-algebra $R$, consider the commutative diagram
 \begin{equation} \label{thisdiag}\xymatrix{
W\Omega^{n}_R/p\ar[r]^{F-1}\ar[d] & W\Omega^{n}_R/p\ar[d]\\
\Omega^n_R\ar[r]_{C^{-1}-1}&\Omega^n_R/d\Omega^{n-1}_R
}. \end{equation} 
Here $C^{-1}$ is the inverse Cartier operator. 
To complete the proof we will show that the two horizontal arrows have isomorphic kernels (resp.\ cokernels), i.e., the square is both cartesian and cocartesian.

To see this, it suffices to show that the map induced between the kernels of
the vertical maps (which are surjective) is an isomorphism. 
The kernel of the map $W \Omega^n_R \to \Omega^n_R$ is generated by the images
of $V, dV$ by \cite[Proposition
I.3.18]{illusie-derham-witt}.  Note that the citation is for smooth
$\mathbb{F}_p$-algebras, but the statement clearly passes to an ind-smooth
algebra as the terms $W_s \Omega^n_{-}$ commute with filtered colimits and $W
\Omega^n_{-} = \varprojlim W_s \Omega^n_{-}$. 
It follows that 
the kernel of the first vertical map in \eqref{thisdiag} is spanned by the
images of $V, dV$ while the kernel of the second vertical map is spanned by the
images of $V, d$. 

Thus, we need to show that the map
\begin{equation}
\label{mapFminusone}
F - 1: \mathrm{im}\left(W \Omega^n_{R}/p\oplus W \Omega^{n-1}_{R}/p\xto{V+dV} W \Omega^n_{R}/p\right) \to 
\mathrm{im}\left(W \Omega^n_{R}/p\oplus W \Omega^{n-1}_{R}/p \xto{V+d} W \Omega^n_{R}/p\right)
\end{equation} 
 is an isomorphism. 

First, we show that $F-1$ is surjective. 
We record the following identities in $W \Omega_R^n/p$,
\[(F-1)V=-V,\qquad (F-1)d\sum_{i>0}V^i=d;\]
using it, we see immediately that $F - 1$ is surjective in \eqref{mapFminusone}. 

To see injectivity, 
we will use an important property of the de Rham--Witt complex: the image of $F$ on
$W \Omega_R^n$ consists precisely of those elements $w$ with $dw$ divisible by
$p$, thanks to \cite[(I.3.21.1.5)]{illusie-derham-witt}. Again, the reference is stated in
the smooth case, but it also passes to the ind-smooth case;\footnote{Since the
passage to the ind-smooth case is not completely immediate, we provide the
argument. Firstly, the quasi-isomorphisms $W_s \Omega_R^\bullet/p\simeq
\Omega_R^{\bullet}$ for $s\ge1$ in the smooth case
\cite[Corol.~I.3.15]{illusie-derham-witt} extend at once to the ind-smooth case,
and taking the limit then shows $W \Omega_R^\bullet/p\simeq \Omega_R^\bullet$
(to take the limit recall again that the $p$-torsion in $W_s\Omega^n_R$ is
annihilated by $R:W_s\Omega^n_R\to W_{s-1}\Omega^n_R$). Secondly, the identity
$\ker(d \colon \Omega^n_R\to\Omega^{n+1}_R)=\mathrm{Im}( F \colon
W_2\Omega^n_R\to \Omega^n_R)$ also extends at once from the smooth
\cite[Prop.~I.3.21]{illusie-derham-witt} to the ind-smooth case. Therefore,
given $w\in W\Omega^n_R$ satisfying $dw\in pW\Omega^n_R$, we may write $w=Fx+y$
for some $x,y\in W\Omega^n_R$ such that $y$ vanishes in $\Omega^n_R$. But then
$dy=dw-pFdx\in pW\Omega^n_R$, so $y$ defines a class in $H^n(W\Omega^\bullet/p)$
which vanishes in $H^n(\Omega^\bullet_R)$; by the first claim, we deduce that
the initial class was zero, in other words that $y=dz+pz'$ for some $z\in
W\Omega^{n-1}_R$, $z'\in W\Omega^n_R$. So $w= Fx+FdVz+FVz'$, as required.} for more on this
point cf.~\cite[Sec.~2]{BLM}. 

Suppose $x \in W \Omega^n_R$
can be written as $x = Vy + dVz$ and $(F -1 )x $ is divisible by $p$; we show
that $x $ is divisible by $p$ in $ W \Omega^n_R$. This will hold if $y, z \in
\mathrm{im}(F)$. 
First, $-dx \equiv d((F-1) x)$ modulo $p$, so $dx$ is divisible by $p$. 
It follows that 
$d Vy$, and hence $dy = FdVy$ is divisible by $p$, so that $y$ belongs to the
image of $F$. In particular, modulo $p$, we have $x \equiv dV z$ and 
$(F-1) d Vz  = d (z - Vz)$ is divisible by $p$. 
Therefore, $z - Vz$ belongs to the image of $F$. Since the operator $V$ is
topologically nilpotent on $W \Omega^n_R$, we find that $(1 - V)$ is invertible
and $z \in \mathrm{im}(F)$. 
Therefore, $x$ is divisible by $p$ as desired. 
\end{proof}

An alternate approach to the 
de Rham--Witt complex is developed in \cite{BLM} based on the theory of
\emph{strict Dieudonn\'e complexes}, which are essentially a linear version
of Witt complexes (at least in the case of algebras over $\mathbb{F}_p$). 

A \emph{saturated} Dieudonn\'e complex 
$(X^{\ast}, d, F)$ is a  $p$-torsion-free cochain complex
equipped with an operator $F: X^{\ast} \to X^{\ast}$ such that $dF = pFd$, the
opeartor $F: X^n \to X^n$ is injective, and the image of $F$ consists
precisely of those $x \in X^n$ such that $p $ divides $dx$; this implies that one
can define uniquely an operator $V$ such that $FV = VF = p$.  
A saturated Dieudonn\'e complex $(X^{\ast}, d, F)$ is called \emph{strict} if 
$X^{\ast}$ is in addition complete for the filtration defined by
$\{\mathrm{im}(V^n, d V^n)\}$. The de Rham--Witt complex of an ind-smooth
algebra is a strict Dieudonn\'e complex, and the functor $A \mapsto W
\Omega_A^{\ast}$ commutes with filtered colimits from ind-smooth
$\mathbb{F}_p$-algebras to the category of strict Dieudonn\'e complexes (where
colimits in the latter involve a completion process). 
Thus, the result can
be deduced once one knows: 

\begin{proposition} 
\label{strictDieudonne}
The functors
from the category of strict Dieudonn\'e complexes to graded abelian groups
given by 
$$(X^{\ast}, d, F)\mapsto \mathrm{ker}(F-1: X^{\ast}/p \to X^{\ast}/p), \quad 
\mathrm{coker}( F-1 : X^{\ast}/p \to X^{\ast}/p)
$$
commute with filtered colimits. \end{proposition} 
Proposition~\ref{strictDieudonne} can be proved in an entirely analogous manner
as above.

\section{Henselian rings and pairs}

\subsection{Nonunital rings}
In this section we will study various categories of non-unital rings. 

\begin{definition} 
Given a commutative base ring $R$, a \emph{nonunital $R$-algebra}
is an $R$-module $I$ equipped with a multiplication map $I \otimes_R I \to I$
which is commutative and associative. We denote the category of nonunital $R$-algebras by $\nur$.
By contrast, when we say
``$R$-algebra,'' we assume the existence of a unit.  
\end{definition} 

Given an $R$-algebra $S$, any ideal $I \subset S$ is a nonunital $R$-algebra. 
Conversely, given a nonunital $R$-algebra $J$, we can adjoin to it a unit and thus form the $R$-algebra $S = R
\ltimes J$, in which $J$ is embedded as an ideal. This latter construction establishes an equivalence between the category of nonunital $R$-algebras and the category of augmented $R$-algebras. 

Note that the category $\nur$ has all limits and colimits, and that the forgetful
functor to sets preserves all limits and sifted colimits. We next describe the free objects in
$\nur$.
\begin{example} 
Let $R[x_1, \dots, x_n]^+$ denote the ideal $(x_1, \dots, x_n) \subset
R[x_1, \dots, x_n]$ in the polynomial ring $R[x_1, \dots, x_n]$.  
Then $R[x_1, \dots, x_n]^+$ is the free object of $\nur$ on $n$ generators.
\end{example} 

\begin{example} 
We say that a sequence $I' \to I \to \overline{I}$ in $\nur$ is a \emph{short
exact sequence}
if it is a short exact sequence of underlying abelian groups. 
That is, $I \to \overline{I}$ is surjective, and we can recover $I'$ as the
pullback $0 \times_{\overline{I}} I$ in the category $\nur$. 
\end{example}

Next we define the notion of a local nonunital $R$-algebra.

\begin{definition} 
A non-unital $R$-algebra $I$ 
is \emph{local} if the following equivalent conditions
hold: 
\begin{enumerate}
\item  For any $x \in I$, there exists $y \in I$ (necessarily unique) such that $x + y + xy = 0$.
\item Whenever $I$ embeds as an ideal in an $R$-algebra $S$, then $I$ is contained
in the Jacobson radical of $S$. 
\item $I \subset R \ltimes I$ is contained in the Jacobson radical of $R \ltimes
I$. 
\end{enumerate}
The equivalence of the above conditions follows because, given an ideal $I
\subset S$ in a commutative ring $S$, then $I$ belongs to the Jacobson radical of $S$ if and only if $1 + I$
consists of units of $S$. 

We let $\nuloc \subset \nur$ denote the full subcategory of nonunital local 
$R$-algebras. 
\end{definition} 

Since the element $y$ in condition (1) is unique, the subcategory $\nuloc$
is closed under limits and sifted colimits in $\nur$, both of which are computed at the level of underlying sets.

To describe all colimits, we need to localize further by observing that the inclusion $\nuloc \subset \nur$ admits a left adjoint. Namely, given a nonunital $R$-algebra $I$, we can build a local nonunital $R$-algebra $I[(1+I)^{-1}]$ as follows: the elements formally written $``1+x"$ for $x\in I$ form a commutative monoid under multiplication, and this monoid acts on $I$ by multiplication.  This makes $I$ into a non-unital algebra over the monoid ring $R[1+I]$, and we can set
\begin{equation} \label{loccons} I[(1+I)^{-1}] := I\otimes_{R[1+I]}
R[(1+I)^{gp}], \end{equation}
where $(1+I)^{gp}$ is the group completion.

Equivalently, if $I$ is embedded as an ideal in an $R$-algebra $S$, then we can form the localization $S[(1+I)^{-1}]$ in the usual sense of commutative algebra, and then realize $I[(1+I)^{-1}]$ as the kernel of the augmentation $S[(1+I)^{-1}]\rightarrow S/I$.  The equivalence of this description with the previous one follows from the exactness of localizations.

In any case, this construction $\nur\to \nuloc$, $I \mapsto I[(1 + I)^{-1}]$ is the desired left adjoint.
To compute a colimit in $\nuloc$, one computes the colimit in $\nur$ and
then applies this left adjoint. 
\begin{example} 
Let $R[x_1, \dots, x_n]^+_{(x_1, \dots, x_n)}\in\nuloc$ be the image of $R[x_1, \dots, x_n]^+\in\nur$ under the left adjoint explained above. In other words, $R[x_1, \dots, x_n]^+_{(x_1, \dots, x_n)}$ is the ideal $(x_1,
\dots, x_n)$ of the localization $R[x_1, \dots, x_n]_{1 + (x_1, \dots, x_n)}$ of
the polynomial ring $R[x_1, \dots, x_n]$ at its 
multiplicative subset $1 + (x_1, \dots, x_n)$.

It follows that 
$R[x_1, \dots, x_n]^+_{(x_1, \dots, x_n)}$  is a local nonunital $R$-algebra,
and it is the free object on $n$ generators. 
\end{example} 

\begin{remark} 
Stated more formally, the preceding discussion shows that the inclusion $\nuloc \subset \nur$ is the right adjoint of a
\emph{localization} functor on the category $\nur$. Specifically, $\nuloc$
consists of those objects in $I \in \nur$ which are orthogonal \cite[Sec.
5.4]{Borceux1} to the map
$R[x]^+ \to R[x]^+_{1 + (x)}$, in the sense that such that any map $R[x]^+ \to I$ extends
uniquely over $R[x]^+_{1 + (x)}$.  Moreover, one sees easily from the
construction of the  localization \eqref{loccons} that it is independent of the ground ring $R$.

Analogous statements will be true when we restrict further to the subcategory of henselian nonunital rings; see the proof of Lemma \ref{lemma_nuh}.
\end{remark} 

We now make the following definition following the discussion in \cite{gabber}.
This is surely known to experts, but for the convenience of the reader we spell
out some details. 
\begin{definition} \label{def_hens_nu}
A nonunital $R$-algebra $I$ is \emph{henselian} if 
for every  $n \geq 1$ and every polynomial $g(x) \in I[x]$, the equation
\begin{equation}\label{henseq} x(1 +
x)^{n-1} + g(x)  = 0 \end{equation} has a solution in $I$. 
We let $\nuh \subset \nur$ denote the full subcategory of henselian
nonunital $R$-algebras. 
\end{definition} 

\begin{remark} 

We note that if $I$ is a henselian nonunital $R$-algebra then:
\begin{enumerate}
\item Considering the equation $x + xy + y = 0$ for $y \in I$, we find that $I$
is local, i.e., $\nuh \subset \nuloc$.
\item The root of the equation \eqref{henseq} is necessarily unique: denoting
the equation by $f(x)=0$ for simplicity, if $\alpha,
\alpha' \in I$ are both roots then we have $f'(\alpha) \in 1 + I$ and  
so the Taylor expansion shows that 
$$0 = f(\alpha) - f(\alpha')   \in  \alpha - \alpha' + I(\alpha - \alpha'),$$ whence 
$\alpha = \alpha'$ since $I$ is local.
\end{enumerate}

\end{remark} 

Since the solution of \eqref{henseq} is unique if it exists, the category of henselian nonunital rings has all limits and sifted colimits, both of which are computed at the level of underlying sets. 
Moreover, if $I \twoheadrightarrow I'$ is a surjection in $\nur$
and $I \in \nuh$, then $I' \in \nuh$ too. 
Finally, the condition that a nonunital $R$-algebra be henselian does
not depend on the base ring $R$, i.e., we might as well take $R = \mathbb{Z}$ in
the definition.

Now we prove the existence of a left adjoint $\nur \to \nuh$, which will be
called \emph{henselization}.\footnote{It follows from this that the category of henselian
nonunital rings is the category of algebras over a \emph{Lawvere theory}, cf.
Remark~\ref{remark_Lawvere}. The free algebras in this theory are the
henselizations of polynomial rings, Example~\ref{freeobjs}.}
We will see below in Corollary \ref{corollary_hens_ind_of_base} that the henselization does not depend on the base ring $R$.

\begin{lemma} \label{lemma_nuh}
The category $\nuh$ is presentable and 
the inclusion $\nuh \subset \nur$ admits a left adjoint.
\end{lemma} 
\begin{proof} 
It follows from the above discussion that $\nuh$ is the orthogonal \cite[Sec.~5.4]{Borceux1} of $\nur$
with respect to the maps
\[ f_{n,t}: R[x_0, \dots, x_t]^+  \to 
R[x_0, \dots, x_t, y]^+/( y(1 + y)^{n-1} + x_0 + x_1 y + \dots + x_t y^t ),\]
for $n, t \geq 1$. 
That is, a given object $X \in \nur$ belongs to $\nuh$ if and only
\begin{equation} \label{hommap} 
\hom_{\nur}(
R[x_0, \dots, x_t, y]^+/( y(1 + y)^{n-1} + x_0 + x_1 y + \dots + x_t y^t ), X)
\xto{\cdot\circ f_{n,t}}
\hom_{\nur}(R[x_0, \dots, x_t]^+ , X)
\end{equation}
is an isomorphism for all $n,t\ge1$. It now follows formally that $\nuh$ is presentable and the
desired left adjoint exists \cite[Cor. 5.4.8]{Borceux1}. 

Alternatively, one can appeal to the theory of henselian pairs and define the
left adjoint directly by taking the henselization of the pair $(R \ltimes I, I)$; see the forthcoming Construction \ref{construction_henselization}.
\end{proof}

\begin{example} 
\label{freeobjs}
Let $R\left\{x_1, \dots, x_n\right\}^+\in \nuh$ denote the 
henselization of $R[x_1, \dots, x_n]^+ \in \nur$. 
By construction, $R\left\{x_1, \dots, x_n\right\}^+ \in \nuh$ is the free object on $n$
generators.\end{example} 

\begin{remark}\label{remark_Lawvere}
The categories $\nur$, $\nuloc$ and $\nuh$ are all examples of models for a {\em
Lawvere} or {\em algebraic theory} \cite[Ch.~3]{ARloc} \cite[Ch.~3]{Borceux}.

Let $\mathcal{C}$ be a category satisfying the following conditions:

\begin{itemize}
\item[(Law)] $\mathcal C$ has all limits and colimits, and is equipped with a
functor $U \colon \mathcal C\to \mathrm{Sets}$ which is conservative, preserves
sifted colimits, and admits a left adjoint $F \colon \mathrm{Sets}\to\mathcal C$.
\end{itemize}
For example, $\mathcal C$ might be the category of rings, nonunital rings,
groups, etc.  In this case, one takes for $U$ the forgetful functor taking the
underlying set, and its right adjoint $F$ is the free ring, nonunital ring,
group, etc.~on the given set. Denoting by $\{1\}$ a one-point set, the element
$F(\{1\})$ is therefore both compact projective (i.e.,
$\hom_{\mathcal{C}}(F(\{1\}), \cdot)$ commutes with sifted colimits) and a
strong generator (i.e., $\hom_{\mathcal{C}}(F(\{1\}), \cdot)$ is faithful and
conservative). The {\em free object of $\mathcal C$ on $n$ generators} is by
definition $\bigsqcup_nF(\{1\})\simeq F(\{1,\dots,n\})$. By the monadicity
theorem, $\mathcal{C}$ is monadic over $\mathrm{Sets}$, via a monad that
preserves sifted colimits. 

Let $\mathcal C'\subset \mathcal C$ be a full subcategory, and assume that
$\mathcal C'$ is closed under limits and sifted colimits (in $\mathcal C$) and that the inclusion $\mathcal C'\subset\mathcal C$ is a right adjoint. Then clearly $\mathcal C'$ also satisfies conditions (Law), with $U'$ given by the restriction of $U$ to $\mathcal C'$. Moreover, the free objects of $\mathcal C'$ are given by applying the left adjoint $\mathcal C\to \mathcal C'$ to the free objects of $\mathcal C$. In practice, this is 
how the ``free objects'' in the above categories are constructed. 

According to general results of Lawvere theory \cite[Thm.~3.9.1]{Borceux} (in fact, we do not use the
results in this paragraph, but the point of view may be helpful), the full
subcategory  $\mathcal{C}_{\Sigma} \subset \mathcal{C}$ of compact projective objects of
$\mathcal{C}$ is the idempotent completion of the full subcategory consisting of
the free objects $\left\{\bigsqcup_{n} F(\{1\}) : n \geq 0\right\}$. Furthermore, $\mathcal{C}$ can be identified with the category of presheaves on $\mathcal{C}_\Sigma$ which commute with finite products. In particular, objects of $\mathcal{C}$ can be identified with sets equipped with various ``operations'' arising from maps between the free objects satisfying various relations. 
\end{remark}

\subsection{Henselian pairs}
Following Gabber \cite{gabber}, we now discuss the connection between henselian nonunital rings and the more familiar notion of henselian pairs \cite[Tag
09XD]{stacks-project} or \cite[Ch.~XI]{Raynaud-henselian}. 
We will thus deduce that the constructions of the previous
subsection do not depend on the base ring $R$.

\begin{definition} 
\label{def:henspair}
A \emph{pair} is the data $(S, I)$ where $S$ is a commutative ring and $I \subset S$ is
an ideal. The collection of pairs forms a category in the obvious
manner. 

The pair $(S, I)$ is said to be \emph{henselian} if the following equivalent
(cf.~\cite[Tag 09XD]{stacks-project} for the equivalence) conditions hold: 
\begin{enumerate}
\item  
Given a polynomial $f(x) \in S[x]$ and a root $\overline{\alpha} \in S/I$ of $\overline{f} \in (S/I)[x]$ with
$\overline{f}'(\alpha)$ being a unit of $S/I$, then $\overline \alpha$ lifts to a root $\alpha \in S$ of $f$.
Note that the lifted root $\alpha \in S$ is necessarily unique by the same argument as we gave after Definition \ref{def_hens_nu}.
\item The ideal $I$ is contained in the Jacobson radical of $S$, and the same
condition as (1) holds for \emph{monic} polynomials  $f(x) \in S[x]$. 
\item Given any commutative diagram
\begin{equation} \label{lifthens}  \xymatrix{
A \ar[d] \ar[r] &  S \ar[d]  \\
B \ar[r] \ar@{-->}[ru] &  S/I\
}\end{equation}
with $A \to B$ \'etale, there exists a lift as in the dotted arrow. 

\end{enumerate}
We may also say that the surjective map $S \to S/I$ is a henselian pair, if there is no risk
of confusion. 

If $S$ is local with maximal ideal $\frak m$, then $S$ is said
to be a \emph{henselian local ring} if $(S, \frak m)$ is a henselian pair. 
\end{definition}

\begin{remark}[Uniqueness in the lifting property]
\label{uniquenessinliftingproperty}
Let $(S, I)$ be a henselian pair and consider a diagram as in 
\eqref{lifthens} with $A \to B$ \'etale. Then the lifting is unique. 
Indeed, given two liftings $f_1, f_2: B \to S$ in \eqref{lifthens}, we have a commutative diagram
\[ \xymatrix{
B \otimes_A B \ar[d]^m  \ar[r]^{f_1 \otimes f_2} &  S \ar[d]  \\
B  \ar[r] &  S/I,
}\]
for $m \colon B \otimes_A B \to B$ the multiplication map. 
Since $m$ is also \'etale (in fact, the projection on a direct factor), 
the lifting property again shows that $f_1 \otimes f_2$ factors through $m$, so
$f_1 = f_2$. 
\end{remark} 

\begin{remark}[{Cf.~\cite[Tag 09XD]{stacks-project}}] 
\label{subideal:henspair}
Let $(R, I)$ be a henselian pair and let $J \subset I$ be a subideal. Then $(R,
J)$ remains a henselian pair. 
\end{remark} 

  We also record for future reference the following property of henselian pairs
with respect to smooth morphisms. 

\begin{theorem}[{Elkik \cite[Sec.~II]{elkik}}] 
\label{elkikthm}
Let $(S, I)$ be a henselian pair. Then $S \to S/I$ has the right lifting
property with respect to smooth maps. That is, any diagram as in
\eqref{lifthens} with $A \to B$ smooth (rather than \'etale) admits a lift as in the dotted arrow.
\end{theorem} 

We crucially need the following observation of Gabber that the condition that $(S,
I)$ be henselian depends only on $I$ as a nonunital ring, cf.~also
\cite[Prop.~XI.1]{Raynaud-henselian}. 

\begin{proposition}[{\cite[Prop.~1]{gabber}}] 
\label{checkhensnonunital}
Let $(S, I)$ be  a pair. Then $(S, I)$ is a henselian pair 
if and only if $I$ is henselian as a nonunital ring. 
\end{proposition} 

\begin{corollary} 
Let $(S, I)$ be a pair.  Then $(S, I)$ is a
henselian pair if and only if $(\mathbb{Z} \ltimes I, I)$ is a
henselian pair. 
\end{corollary}

We recall the following basic construction (see \cite[Tag
0EM7]{stacks-project}) in the theory of henselian pairs. 
\begin{cons}\label{construction_henselization}
Given a pair $(S, I)$, there is a pair $(S^h, I^h)$ and a map $(S, I) \to (S^h,
I^h)$, called the \emph{henselization} of the original pair, with the following properties.
\begin{enumerate}
\item $(S^h, I^h)$ is a henselian pair and is the initial henselian pair
receiving a map from $(S, I)$. That is, the construction $(S, I) \mapsto
(S^h, I^h)$ is the left adjoint 
to the forgetful functor from henselian pairs to pairs. 
\item $S^h$ is a filtered colimit of \'etale $S$-algebras and $I^h = I S^h =
I \otimes_S S^h$.  
\item The map $S/I \to S^h/I^h$ is an isomorphism. 
\end{enumerate}
 
\end{cons}

\begin{remark} 
\label{hensfact}
Let $(S, I)$ be a pair. Suppose given a factorization $S \to \widetilde{S} \to
S/I$ such that
$\widetilde{S}$ is a filtered colimit of \'etale $S$-algebras and such that
$\widetilde{S} \to S/I$ has the right lifting property with respect to \'etale
morphisms (i.e., the surjection $\widetilde{S} \to S/I$ is a henselian pair). Then $(\widetilde{S}, I \widetilde{S})$ is the henselization of $(S,
I)$. To see this, 
consider a henselian pair $(A, J)$ and a map $(S, I) \to (A, J)$. Considering
the commutative diagram
\[ \xymatrix{
S \ar[d] \ar[r] & \widetilde{S} \ar@{-->}[ld] \ar[r] &  S/I \ar[d]  \\
A \ar[rr] & & A/J
},\]
the ind-\'etaleness of $S \to \widetilde{S}$ implies the existence of a
\emph{unique}
dotted arrow $\widetilde{S} \to A$ making the diagram commute
(cf.~Remark~\ref{uniquenessinliftingproperty}). This verifies the
universal property of the henselization. 

In particular, one can construct $(S^h, I^h)$ by appealing to Quillen's small object argument 
\cite[Th. 2.1.14]{Hovey} to factor $S \to S/I$ as the composite of a filtered colimit of
\'etale morphisms and a morphism that has the right lifting property with
respect to \'etale morphisms; the universal property 
of the henselization then constructs it as a functor. Note that in \eqref{lifthens}, it suffices to take
$A, B$ finitely generated over $\mathbb{Z}$, by the structure theory for \'etale
morphisms.
\end{remark}

We now review the relation between henselizations of pairs and
henselizations of nonunital rings. Most of this is implicit in \cite{gabber}, but we spell out the details.

\begin{definition} 
\label{milnorsquare}
A \emph{Milnor square} of commutative rings is a diagram
\begin{equation}  \xymatrix{
S \ar[d] \ar[r]&  T \ar[d]  \\
S'  \ar[r] &  T'
} \label{milnsquare} \end{equation} 
such that the vertical arrows are surjective and such that the diagram is both
cartesian and cocartesian in the category of commutative rings. 
It follows in particular that 
if $I \subset S$, $J \subset T$
are the respective kernels of the vertical maps $S \to S'$, $T \to T'$, then $f|_I
: I \isoto J$ establishes an isomorphism of the ideals $I$, $J$. 
\end{definition} 

\begin{lemma} 
\label{milnsquarehens} 
Consider a Milnor square as in \eqref{milnsquare} with respect to the ideals $I
\subset S$, $J \subset T$. 
Then the henselizations $S^h$, $T^h$ of $S$, $T$ along $I$, $J$ fit into a Milnor square 
\[ \xymatrix{
S^h \ar[d]  \ar[r] &  T^h \ar[d]  \\
S' \ar[r] &  T'
}.\]
Furthermore, $T^h \simeq S^h \otimes_S T$. 
\end{lemma} 
\begin{proof} 
Since $S^h$ is flat over $S$, we can base-change \eqref{milnsquare} along $S \to
S^h$ to obtain a
new Milnor square
\[  \xymatrix{
S^h \ar[d]  \ar[r] & T \otimes_S S^h    \ar[d]  \\
S' \otimes_S S^h \ar[r] &  T' \otimes_S S^h
}.  \]
Note that $S' \otimes_S S^h  \simeq S'$ by the properties of the henselization, and similarly for $T$. Therefore the bottom arrow can be rewritten as $S' \to T'$. 
Since this is a Milnor square and the left vertical arrow is a henselian pair,
so is the right vertical arrow thanks to Proposition~\ref{checkhensnonunital}.
We have a factorization  $T \to T \otimes_S S^h \to  T'$ as the composite of an
ind-\'etale map and a map having the right lifting property with respect to
\'etale maps. Thus, the result follows in view of Remark~\ref{hensfact}. 
\end{proof} 

\begin{corollary}\label{corollary_hens_ind_of_base}
 Let $(S, I)$ be a pair and let $(S^h, I^h)$ be its henselization. Let $R$
be any ring mapping to $S$. Then $I^h$
is the henselization of $I$ as an object of $\nur$. 
In particular, the henselization of $I$ as an object of $\nur$ does not
depend on the base ring $R$. 
\end{corollary} 
\begin{proof} 
To distinguish the possible henselizations which a priori might not coincide, we
temporarily write $I^{nuh}$ to denote the henselization of $I$ as an object of
$\nur$. The equivalence between $\nur$ and augmented $R$-algebras easily implies
that $(R\ltimes I^{nuh},I^{nuh})$ is the henselization of the pair $(R \ltimes I, I)$. Applying Proposition~\ref{milnsquarehens} to the Milnor square
\[ \xymatrix{
R\ltimes I \ar[d]  \ar[r] &  S \ar[d]  \\
R \ar[r] &  S/I
}\]
then reveals that the square
\[ \xymatrix{
R\ltimes I^{nuh} \ar[d]  \ar[r] &  S^h \ar[d]  \\
R \ar[r] &  S^h/I^h
}\]
is also Milnor, i.e., that $I^{nuh}\isoto I^h$, as desired.
\end{proof}

Finally we recall Gabber's result that free henselian nonunital $\mathbb Q$-algebras are colimits of such $\mathbb Z$-algebras. For each $N>0$, denote by \[[N]: \mathbb{Z}\left\{x_1,
\dots, x_n\right\}^+ \to \mathbb{Z}\left\{x_1,
\dots, x_n\right\}^+\] the endomorphism in $\nuzh_{\mathbb Z}$ which sends $x_i \mapsto Nx_i$.

\begin{corollary}[{\cite[Prop.~3]{gabber}}] 
\label{rationalhens}
Let $n > 0$. The filtered colimit of the endomorphisms $[N]$ on $\mathbb{Z}\left\{x_1,
\dots, x_n\right\}^+$, indexed over natural numbers $N$ ordered by divisibility, is
$\mathbb{Q}\left\{x_1, \dots, x_n\right\}^+$. 
\end{corollary} 
\begin{proof} 
Henselization commutes with filtered colimits, and the filtered colimit of
the analogous maps $[N]: \mathbb{Z}[x_1, \dots, x_n]^+$ is given by
$\mathbb{Q}[x_1, \dots, x_n]^+$. We now use Corollary \ref{corollary_hens_ind_of_base} that henselization does
not depend on the base ring. 
\end{proof} 

\label{sec:kanext}

\section{The main rigidity result}
In this section, we prove the main result of the paper (Theorem~\ref{mainthm}
below), which states that the relative $K$-theory and relative topological cyclic homology of a henselian pair agree after profinite completion.

Our proof will rely on several steps: a direct verification for smooth algebras in equal characteristic
(where it is a corollary of the deep calculations of Geisser--Levine \cite{GL} of
$p$-adic $K$-theory and of Geisser--Hesselholt \cite{GH} of topological cyclic
homology), a finiteness property of $K$-theory and topological cyclic homology (expressed in the language of
``pseudocoherent functors'' below), and an imitation of the main steps of
the proof of 
Gabber rigidity \cite{gabber}.

\subsection{Generalities on pseudocoherence}
In this subsection, we 
describe a basic finiteness property (called \emph{pseudocoherence}) for spectrum-valued functors
that will be necessary in the proof of the main theorem. Later, we will show
that $K$-theory and $\TC$ satisfy this property. 

\begin{definition} 
Let $\mathcal{C}$ be a small category and $F: \mathcal{C} \to \mathrm{Ab}$ a functor.
We say that $F$ is \emph{finitely generated} if there exist finitely many
objects
$X_1, \dots, X_n \in \mathcal{C}$ and a surjection 
of functors
\begin{equation} \label{fingenfun} \bigoplus_{i=1}^t \mathbb{Z}[\hom_{\mathcal{C}}(X_i, \cdot)]
\twoheadrightarrow F. \end{equation}
\end{definition} 

We now want to introduce analogous concepts when $\mathrm{Ab}$ is replaced by
the $\infty$-category $\Sp$ of spectra. 
We consider the $\infty$-category $\fun(\mathcal{C}, \Sp)$ of
functors from $\mathcal{C}$ to spectra, and recall that
$\fun(\mathcal{C}, \Sp)$ is a 
presentable, stable $\infty$-category in which limits and colimits are computed
targetwise. A family of compact generators of $\fun(\mathcal{C}, \Sp)$ is given by the
corepresentable functors $\Sigma^\infty_+ \mathrm{Hom}_{\mathcal{C}}(C, \cdot)$, for $C \in
\mathcal{C}$.

\begin{definition} 
Let $\mathcal{C}$ be a small category and $F: \mathcal{C} \to \Sp$ a functor.
\begin{enumerate}
\item  
We say that $F$ is \emph{perfect} if $F$ belongs to the thick subcategory of $\fun(\mathcal C,\Sp)$ generated by the functors $\Sigma^\infty_+ \hom_{\mathcal{C}}(C, \cdot)$ for $C
\in \mathcal{C}$ (or equivalently is  a compact object of $\fun(\mathcal{C},
\Sp)$, by a thick subcategory argument).
\item
We say that $F$ is \emph{pseudocoherent} if, for each $n \in \mathbb{Z}$, there exists a
perfect functor $F'$ and  a map $F' \to F$ such that $\tau_{\le n}F'(C)\to\tau_{\le n}F(C)$ is an equivalence for all $C\in \mathcal C$.
\end{enumerate}
\end{definition}

In the setting of structured ring spectra, the analog of a perfect functor is a
perfect module and the analog of a pseudocoherent functor is an almost perfect
module. We refer to \cite[7.2.4]{HA} for a detailed account of the theory in
that setting; therefore, in our setting of functors, we will only sketch the
proofs of the basic properties.
Note in the next result that pseudocoherent functors are automatically bounded
below, so the initial hypothesis is no loss of generality. 
\begin{lemma} 
\label{pscohcompact}
Let $d \in \mathbb{Z}$. The following conditions on a functor $F \in \fun(\mathcal{C}, \Sp_{\geq d})
\subset \fun(\mathcal{C}, \Sp)$ are equivalent:
\begin{enumerate}
\item $F$ is pseudocoherent.
\item For each $n \in \mathbb{Z}$, the functor $\tau_{\leq n} F$ is a compact object of the
$\infty$-category $\fun(\mathcal{C}, \Sp_{\leq n})$. 
\item There exists a sequence $G_{d-1} = 0 \to G_d \to G_{d+1} \to G_{d+2} \to \dots $ such that 
$G_i/G_{i-1}$ is a finite direct sum of functors of the form $\Sigma^i
\Sigma^\infty_+ \hom_{\mathcal{C}}(C, \cdot)$ for $C \in \mathcal{C}$ and such that
$\varinjlim_i G_i \simeq F$. 
\end{enumerate}
\end{lemma} 

\begin{proof} 
Without loss of generality, we can assume $d =0 $ by shifting. 
For (1) $\Rightarrow$ (2), suppose that $F\in \fun(\mathcal{C}, \Sp)$ (we do not
need $F$ connective) is pseudocoherent and let $n\in\mathbb Z$. Then by definition there is a perfect functor $F'$ and a map $F' \to
F$ inducing an equivalence $\tau_{\leq n} F' \simeq \tau_{\leq n} F$. Since
$\tau_{\leq n}: \fun(\mathcal{C}, \Sp) \to \fun(\mathcal{C}, 
\Sp_{\leq n})$ preserves compact objects (as its right adjoint preserves filtered
colimits), it follows that $\tau_{\leq n} F  \in \fun(\mathcal{C}, 
\Sp_{\leq n})$ is compact. 

For (2) $\Rightarrow$ (3), suppose that $F \in \fun(\mathcal{C},
\Sp_{\geq 0})$ and $\tau_{\leq n} F  \in \fun(\mathcal{C}, 
\Sp_{\leq n})$ is compact for each $n$.  Then, by assumption $\tau_{\leq 0} F  \in \fun(\mathcal{C}, \mathrm{Ab})
\subset \fun(\mathcal{C}, \Sp_{\leq 0})$ is
a compact object and thus $\pi_0 F$ is a finitely generated functor. 
We can thus find a functor $G \in \fun(\mathcal{C}, \Sp_{\geq 0})$ which is a
direct sum of $\Sigma^\infty_+
\hom_{\mathcal{C}}(X, \cdot)$, for finitely many
$X\in\mathcal C$, and a map $G \to F$ which is a surjection on
$\pi_0$. Let $F_1 $ be the resulting cofiber and observe that $F_1$
has the same property as $F$ (namely $\tau_{\leq n} F_1 \in \fun(\mathcal{C}, 
\Sp_{\leq n})$ is compact for each $n$, since compact objects are closed under pushouts) and that $F_1 $ is connected.  
Continuing in this way, we find a sequence of functors $F \to F_1 \to F_2 \to
\dots $ such that 
$F_n$ is concentrated in degrees $\geq n$ and such that the cofiber of each $F_n \to F_{n+1}$ is
a finite direct sum of shifts of representables.  We obtain the desired sequence $G_n$ by setting $G_n=\operatorname{Fib}(F\to F_{n+1})$. 

For (3) $\Rightarrow$ (1), note that each $G_n$ is perfect by induction on $n$,
and that $G_{n+1}\rightarrow F$ is an equivalence on $\tau_{\leq n}$.
\end{proof} 
Before stating some more properties, we introduce a further generalization (for a category relative
to a subcategory) that will also be necessary in the sequel.

\begin{definition} 
Let $\mathcal{D} \subset \mathcal{C}$ be a small full subcategory of a (possibly
large, but locally small) category $\mathcal{C}$. Let $F: \mathcal{C} \to \Sp$ be a functor. 
\begin{enumerate}
\item  
We say that $F$ is \emph{$\mathcal{D}$-perfect} if $F$ belongs to the thick subcategory of $\fun(\mathcal C,\Sp)$ generated by the functors $\Sigma^\infty_+ \hom_{\mathcal{C}}(D, \cdot)$ for $D
\in \mathcal{D}$.

\item
We say that 
a functor $F: \mathcal{C} \to \Sp$
is called \emph{$\mathcal{D}$-pseudocoherent} if for each $n \in \mathbb{Z}$,
there exists a 
$\mathcal{D}$-perfect functor $F'$ and a map $F' \to F$ such that $\tau_{\leq n}
F'(C) \to \tau_{\leq n} F(C)$ is an equivalence for all $C \in \mathcal{C}$. 
\item A functor $F_0: \mathcal{C} \to \mathrm{Ab}$ 
is called \emph{$\mathcal{D}$-finitely generated} if 
there is a surjection as in \eqref{fingenfun} with $X_i \in \mathcal{D}$ for
each $i$. 
\end{enumerate}
\end{definition} 

Next, we recall a basic construction that lets us reduce
$\mathcal{D}$-pseudocoherence to (unrelative) pseudocoherence. 
\begin{cons}
Let $\mathcal{D} \subset \mathcal{C}$ be a full subcategory. 
Given a functor $G: \mathcal{D} \to \Sp$, we can form the \emph{left Kan
extension} $\lan(G):\mathcal{C} \to \Sp$ (cf. \cite[Sec.~4.3]{HTT}). 
By definition, for $C \in \mathcal{C}$, $\lan(G)(C)$ is the colimit
$$ \lan(G)(C) = \varinjlim_{D \to C, D \in \mathcal{D}} G(D)  \in \Sp.$$
The construction $F \mapsto \lan(F)$ is the left adjoint of the forgetful
functor $\fun(\mathcal{C}, \Sp) \to \fun(\mathcal{D}, \Sp)$.  A functor 
$F \in \fun(\mathcal{C}, \Sp)$ is said to be 
\emph{left Kan extended} from $\mathcal{D}$ if the natural map $\lan(
F|_{\mathcal{D}}) \to F$ is an equivalence in $\fun(\mathcal{C}, \Sp)$. 
The objects of $\fun(\mathcal{C}, \Sp)$ which are left Kan extended from
$\mathcal{D}$ form a subcategory equivalent to $\fun(\mathcal{D}, \Sp)$ (via
restriction). 
\end{cons}

\begin{lemma} 
\label{lkanpscoh}
A functor $F: \mathcal{C} \to \Sp$ is $\mathcal{D}$-pseudocoherent if and only if
it is left Kan extended from $\mathcal{D}$ and $F|_{\mathcal{D}}: \mathcal{D}
\to \Sp$ is
pseudocoherent.  
\end{lemma} 
\begin{proof} 
The functor $\Sigma^\infty_+ \hom_{\mathcal{C}}(D, \cdot) \in
\fun(\mathcal{C}, \Sp)$ for $D \in \mathcal{D}$ is left Kan extended from
$\mathcal{D}$, so a thick subcategory argument shows that any
$\mathcal{D}$-perfect functor is left Kan extended from $\mathcal{D}$.  To get
the same for any $\mathcal{D}$-pseudocoherent functor, note that the
truncations $\Sp\rightarrow \Sp_{\leq n}$ are a conservative family of
colimit-preserving functors, so a functor to $\Sp$ is left Kan extended if and
only if its image in each $\Sp_{\leq n}$ is.

Conversely, suppose that $F$ is left Kan extended from $\mathcal{D}$ and $F|_{\mathcal{D}}: \mathcal{D}
\to \Sp$ is pseudocoherent.  Then $\tau_{\leq n}F\mid_\mathcal{D}\simeq \tau_{\leq n}G\mid_\mathcal{D}$ for some $G:\mathcal{C}\rightarrow\Sp$ in the thick subcategory generated by the $\Sigma^\infty_+ \hom_{\mathcal{C}}(D, \cdot)$, and we deduce $\tau_{\leq n}F\simeq \tau_{\leq n}G$ by left Kan extension.
\end{proof}

\begin{proposition} 
\label{pscohfilt2}
Let $\mathcal{D} \subset \mathcal{C}$ be a small full subcategory. 
Let $F: \mathcal{C} \to \Sp_{\geq d}$ be a functor for some $d \in \mathbb{Z}$. 
Then the following are equivalent: 
\begin{enumerate}
\item $F$ is $\mathcal{D}$-pseudocoherent.  
\item 
There exists a sequence $G_{d-1} = 0 \to G_d \to G_{d+1} \to G_{d+2} \to \dots $ such that 
$G_i/G_{i-1}$ is a finite direct sum of functors of the form $\Sigma^i
\Sigma^\infty_+ \hom_{\mathcal{C}}(D, \cdot)$ for $D \in \mathcal{D}$ and such that
$\varinjlim_i G_i \simeq F$. 
\end{enumerate}
\end{proposition}
\begin{proof} 
Combine Lemmas~\ref{lkanpscoh} and \ref{pscohcompact}. 
\end{proof}

\begin{proposition} 
\label{general:pscoh}
Let $\mathcal C$ be a category and $\mathcal{D} \subset \mathcal{C}$ a
small full subcategory.
\begin{enumerate} 
\item  
The subcategory of $\fun(\mathcal{C}, \Sp)$ spanned by the $\mathcal{D}$-pseudocoherent
functors is thick.
\item Let $d \in \mathbb{Z}$,  let $K$ be a simplicial set, and let $f: K \to
\fun(\mathcal{C}, \Sp_{\geq d})$ be a
$K$-indexed diagram of functors $\mathcal{C}
\to \Sp_{\geq d}$. Suppose that
the $n$-skeleton 
$\mathrm{sk}_n K$ is a finite simplicial set for every $n\in\mathbb Z$ and
that for each vertex $k_0 \in K_0$, the functor $f(k_0) \in \fun(\mathcal{C}, \Sp)$ is $\mathcal{D}$-pseudocoherent. 
Then the functor $\varinjlim_K f \in \fun(\mathcal{C}, \Sp)$ is
$\mathcal{D}$-pseudocoherent. 
\item 
Let $d \in \mathbb{Z}$, and let $F_\bullet:\Delta^\sub{op}\to \fun(\mathcal{C}, \Sp_{\geq d})$ be a simplicial object in the category of functors $\mathcal{C}
\to \Sp_{\geq d}$. Suppose that $F_i$ is $\mathcal{D}$-pseudocoherent for every $i \geq 0$. Then the
geometric realization $|F_\bullet|: \mathcal{C} \to \Sp$ is $\mathcal{D}$-pseudocoherent. 
\item Let $d\in \bb Z$, and let $F\in \fun(\mathcal{C}, \Sp)$ be a
$\mathcal{D}$-pseudocoherent functor such that $\pi_nF=0$ for all $n<d$. Then $\pi_d F: \mathcal{C} \to \mathrm{Ab}$ is a
 $\mathcal{D}$-finitely generated functor. 
 \item If $\mathcal{C}$ has finite coproducts and $\mathcal{D}$ is closed
 under them, then the subcategory of
 $\mathcal{D}$-pseudocoherent functors in $\fun(\mathcal{C}, \Sp)$ is closed under smash
 products. 
\end{enumerate}
\end{proposition} 
\begin{proof} 
Thanks to Lemma~\ref{lkanpscoh}, 
 we may assume without loss of generality that $\mathcal{C} =
\mathcal{D}$, since the property of being left Kan extended is preserved under
all colimits. We may also take $d = 0$. 
 Claim (1) follows from 
Lemma~\ref{pscohcompact} because
compact objects (of any $\infty$-category) are closed under finite colimits and
retracts. 
For (2), 
it suffices by Lemma~\ref{pscohcompact} to show that
$\tau_{\leq n} ( \varinjlim_K f )$ is compact as an object of $\fun(\mathcal{C},
\sp_{\leq n})$ for each $n$. 
But $\tau_{\leq n} f = \tau_{\leq n}  ( \varinjlim_{\mathrm{sk}^{n+1} K } f)$
and
$\varinjlim_{\mathrm{sk}^{n+1} K } f$ is a finite colimit of pseudocoherent
functors, hence pseudocoherent itself by (1); this implies the desired compactness
assertion. 
Claim (3) is handled similarly, because
$\tau_{\leq n} ( |F_\bullet|) \simeq \tau_{\leq n} ( \varinjlim_{\Delta_{\leq
n+1} }F)$ can be computed as a truncated geometric realization, and this is a
finite colimit, cf.~\cite[Lemma~1.2.4.17]{HA}. 
Claim (4) follows from the filtration Lemma~\ref{pscohcompact}. 
For (5), it suffices to observe that the smash products of functors of the form
$\{\Sigma^\infty_+ \hom_{\mathcal{C}}(X, \cdot)\}_{X \in \mathcal{D}}$ are still of this form under the
assumption that $\mathcal{C}$ has finite coproducts and $\mathcal{D} \subset
\mathcal{C}$ is closed under them. It then follows
by a thick subcategory argument that the subcategory of $\mathcal{D}$-perfect functors is
closed under smash products, which implies the analogous assertion for
$\mathcal{D}$-pseudocoherent functors. 
\end{proof} 

In order to analyze algebraic $K$-theory below, it will be useful to use 
both the $K$-theory space and spectrum simultaneously. Therefore we first prove
a useful tool
(Proposition~\ref{suspensionpscoh}) that for functors $F$
taking values in connected spectra, pseudocoherence of $F$ is equivalent to
that of $\Sigma^\infty \Omega^\infty F$. 
To prove it, we need the
following general result from Goodwillie calculus. Compare \cite[Cor.
1.3]{AhKu}, for instance. 
The tower $\left\{P_n(\Sigma^\infty \Omega^\infty)\right\}$ is a special case
of the Goodwillie tower of an arbitrary functor introduced in \cite{GooIII}. 

\begin{proposition} 
There is a natural tower $\left\{P_n=P_n ( \Sigma^\infty
\Omega^\infty)\right\}_{n \geq 0}$ of functors $\Sp \to \Sp$
receiving a map from $\Sigma^\infty \Omega^\infty$, 
\[ 
\Sigma^\infty \Omega^\infty X \to \{ \dots \to P_n
(X) \to P_{n-1}(X) \to \dots \to 
P_1(X) \}
\]
such that: 
\begin{enumerate}
\item $P_1(X) \simeq X$ and the fiber of  
$P_n (X) \to P_{n-1} (X) $ is naturally equivalent to $(X^{\otimes n})_{h \Sigma_n}$. 
\item If $X \in \Sp_{\geq 1}$, the map 
$\Sigma^\infty \Omega^\infty X  \to \varprojlim_n P_n (X)$ is an equivalence and
the connectivity of 
$\Sigma^\infty \Omega^\infty X \to P_n(X)$ tends to $\infty$ with $n$.
\end{enumerate}
\label{goodwilliefilt}
\end{proposition}

We can now prove the following tool for pseudocoherence. 

\begin{proposition} 
\label{suspensionpscoh}
Let $\mathcal{C}$ be a category with finite coproducts and 
let $\mathcal{D} \subset \mathcal{C}$ be a small subcategory closed under
finite coproducts. 
Let
$F: \mathcal{C}
\to \Sp_{\geq 1}$ a functor. Then the following are equivalent: 
\begin{enumerate}
\item  
$F$ is $\mathcal{D}$-pseudocoherent. 
\item
$\Sigma^\infty \Omega^\infty F$ is $\mathcal{D}$-pseudocoherent. 
\item 
$\Sigma^\infty_+ \Omega^\infty F$ is $\mathcal{D}$-pseudocoherent.
\end{enumerate}
\end{proposition} 
\begin{proof} 
The equivalence of (2) and (3) follows because $\Sigma^\infty, \Sigma^\infty_+$
differ by the constant functor at $S^0$, which is corepresentable since
$\mathcal{C}$ has an initial object in $\mathcal{D}$.

Suppose $F$ is $\mathcal{D}$-pseudocoherent. 
Using Proposition~\ref{goodwilliefilt}, 
we conclude that $\Sigma^\infty \Omega^\infty F$ is the homotopy limit of a
tower whose associated graded is given by $F, (F^{\otimes 2})_{h
\Sigma_2}, (F^{\otimes 3})_{h \Sigma_3 }, \dots $. 
Since $F$ takes values in $\Sp_{\geq 1}$, the connectivity of the associated graded terms
tends to $\infty$, so this tower stabilizes in any given finite range. As $F$
is $\mathcal{D}$-pseudocoherent, so are all the graded terms in
view of Proposition~\ref{general:pscoh}. It follows that $\Sigma^\infty \Omega^\infty F$ is
$\mathcal{D}$-pseudocoherent. 

Conversely, suppose $\Sigma^\infty \Omega^\infty F$ is
$\mathcal{D}$-pseudocoherent. 
First, recall that the adjunction $(\Sigma^\infty, \Omega^\infty)$
between the $\infty$-categories of pointed spaces and connective spectra is monadic since
$\Omega^\infty|_{\Sp_{\geq 0}}$ commutes with sifted colimits \cite[Prop.~1.4.3.9]{HA} and in view of the $\infty$-categorical monadicity theorem
\cite[Sec.~4.7.3]{HA}; 
alternatively, this follows explicitly from delooping
machinery going back to \cite{Maygeom}.  
As a result of monadicity, for any $X \in \Sp_{\geq 0}$, one has a natural simplicial spectrum
(the bar resolution) $(\Sigma^\infty \Omega^\infty)^{\bullet + 1} X$ whose
geometric realization is equivalent to $X$.  
Therefore, in our case, we can
resolve $F$ via the iterates $(\Sigma^\infty \Omega^\infty)^k F$ for $k \geq 1$.
By the previous direction and our assumption, each of these is
$\mathcal{D}$-pseudocoherent.
Thus we have written $F$ as a geometric realization of
$\mathcal{D}$-pseudocoherent
functors $\mathcal{C} \to \Sp_{\geq 1}$,
so $F$ itself is $\mathcal{D}$-pseudocoherent. 
\end{proof}

Finally, we observe that for functors into connective spectra, pseudocoherence
can be tested after smashing with $H\mathbb{Z}$. 
\begin{proposition} 
\label{HZpscoh}
Let $\mathcal{D} \subset \mathcal{C}$ be a small subcategory.
Let $F: \mathcal{C} \to \Sp$ be a functor. If $F$ is
$\mathcal{D}$-pseudocoherent, then $H\mathbb{Z}
\otimes F$ is $\mathcal{D}$-pseudocoherent. Conversely, if $F: \mathcal{C} \to \Sp$ is
uniformly bounded-below and $H \mathbb{Z} \otimes F$ is
$\mathcal{D}$-pseudocoherent, so is
$F$. 
\end{proposition} 
\begin{proof} 
The first direction follows because we can approximate $H \mathbb{Z}$ in any
range by a finite spectrum. The second direction follows because a thick subcategory
argument now implies that if $H \mathbb{Z} \otimes F$ is
$\mathcal{D}$-pseudocoherent, so is
$\tau_{\leq n} S^0 \otimes F$ for each $n$, and we can approximate $F$ in any
given range by $\tau_{\leq n} S^0 \otimes F$. 
\end{proof}

\subsection{An axiomatic rigidity argument}\label{subsection_axiomatic}
In this subsection, we present an axiomatic form, in the language of finitely
generated functors, of the argument used by Gabber
\cite{gabber} to deduce rigidity for henselian pairs from the case of
henselizations of smooth points on varieties. 
We will use the notion of pseudocoherence in the case 
where $\mathcal{C}  = \nuh$ and $\mathcal{D} = (\nuh)_{\Sigma}$ is the category of compact projective objects 
obtained by idempotent completing the subcategory
$\{R\left\{x_1, \dots, x_n\right\}^+: n \geq 0\}$ of $ \nuh$.\begin{definition} 
\begin{enumerate}
\item  
A functor $\nuh \to \mathrm{Ab}$ is \emph{projectively finitely generated} 
if it is $\mathcal{D}$-finitely generated for $\mathcal{D} = (\nuh)_{\Sigma}$.
\item
A functor $\nuh \to \Sp$ is \emph{projectively pseudocoherent} if it
is $\mathcal{D}$-pseudocoherent for $\mathcal{D} = (\nuh)_\Sigma$. 
\end{enumerate}

\end{definition} 

\begin{example} 
\label{ex:pseudocoh}
For each $n$, 
the functor $I \mapsto \Sigma^\infty_+ I^n$ (where $I^n$ here denotes the
cartesian product of $n$ copies of $I$) is projectively pseudocoherent. Indeed, it is the
suspension spectrum of the corepresentable associated to $
R\{x_1,\dots,x_n\}^+ \in (\nuh)_\Sigma$. 
\end{example} 

\begin{example} 
A projectively pseudocoherent functor $\nuh \to \Sp$ commutes with filtered
colimits. In fact, the corepresentable functors $I \mapsto \Sigma^\infty_+ I^n$
clearly do, and the result for arbitrary projectively pseudocoherent functors
follows 
because the class of functors $\nuh \to \Sp$ which commute with filtered
colimits is stable under all colimits. 
See also Corollary~\ref{kanextcrit} below. 
\end{example} 

\label{Moved this around and rephrased language slightly}
We now give the following technical result for 
functors into abelian groups, which is simply a slight reformulation of the approach taken by Gabber
\cite{gabber}. \begin{lemma} \label{gabberlem}
Let $F_0: \nuzh \to \mathrm{Ab}$ be a projectively finitely generated functor. Suppose that: 
\begin{enumerate}
\item  
$F_0(I) = 0$ if $I \in \nuzh$ is annihilated by an integer $N > 0$. 
\item
$F_0(I) =
0$ if $I \in \nuzh$ is a nonunital $\mathbb{Q}$-algebra.
\item 
Given a short exact sequence $I' \to I \twoheadrightarrow \overline{I}$ in
$\nuzh$, the sequence $F_0(I') \to F_0(I) \to F_0(\overline{I}) \to 0$ of
abelian groups is exact.
\item $F_0$ commutes with filtered colimits. 
\end{enumerate}
 Then $F_0 = 0$. 
\end{lemma} 
\begin{proof}
We will prove the following claim: there exists $N\ge1$ such that, for any $I \in \nuzh$, the inclusion map $i_N:NI\to I$ induces the zero map $0=(i_N)_*: F_0(NI) \to F_0(I)$.
Once this is done, the statement of the lemma will follow by applying (1) and (3) to the short exact sequence $NI\to I\to I/NI$.
 
As in Corollary~\ref{rationalhens}, let $[N]: \mathbb{Z}\left\{x_1, \dots, x_n\right\}^+ \to
\mathbb{Z}\left\{x_1, \dots, x_n\right\}^+$ be the operator that multiplies each
$x_i$ by $N$, so that the colimit of the $[N]$, indexed over $N\ge1$ ordered by divisibility, is $\mathbb{Q}\left\{x_1, \dots, x_n\right\}^+$. 
By assumption $F_0$ commutes with filtered colimits
and annihilates any
nonunital henselian 
$\mathbb{Q}$-algebra; so, given any element 
$u \in F_0(\mathbb{Z}\left\{x_1, \dots, x_n\right\}^+)$, it follows that there
exists $N > 0$ such that $[N]_*(u) = 0$. 

Now we use that $F_0$ is projectively finitely generated. This means that there exists a surjection 
\[ \bigoplus_{i=1}^t \mathbb{Z}[\mathrm{Hom}_{\nuzh}(\mathbb{Z}\left\{x_1,
\dots, x_{n_i}\right\}^+, \cdot)] \to F_0\] of functors for some $n_1,\dots,n_t\ge1$. For each $i=1,\dots,t$, let $g_i\in F_0(\mathbb{Z}\left\{x_1, \dots, x_{n_i}\right\}^+)$ be the image of $\mathrm{id}\in \mathrm{Hom}_{\nuzh}(\mathbb{Z}\left\{x_1,
\dots, x_{n_i}\right\}^+, \mathbb{Z}\left\{x_1,
\dots, x_{n_i}\right\}^+)$. Then the above surjection concretely means the following: given any $I\in \nuzh$ and any element $y\in F_0(I)$, then $y$ is a finite sum of elements of $F_0(I)$ obtained by pushing forward the $g_i$ along various maps $\phi:\mathbb{Z}\left\{x_1,
\dots, x_{n_i}\right\}^+ \to I$. 

By the second paragraph, there exists an integer $N > 0$ such that $[N]_*(g_i)=0$ for all the generators $g_i$. From this we can easily complete the proof of the claim, as follows. Given any $I \in \nuzh$, every map $\phi: \mathbb{Z}\{ x_1, \dots, x_{n_i}\}^+ \to  N I 
$ has the property that the composite
$i_N \circ \phi: \mathbb{Z}\left\{x_1, \dots, x_{n_i}\right\}^+ \to
I$ factors through $[N]$, i.e., one has
a commutative diagram
in $\nuzh$
\[ \xymatrix{
\mathbb{Z}\left\{x_1, \dots, x_n\right\}^+ \ar[r]_-{\phi}\ar[d]^{[N]} &  NI
\ar[d]^{i_N}  \\
\mathbb{Z}\left\{x_1, \dots, x_n\right\}^+ \ar[r] &  I
}.\]
By choice of $N$ it follows that the composite
\[ \mathbb{Z}[\mathrm{Hom}_{\nuzh}( \mathbb{Z}\left\{x_1, \dots,
x_{n_i}\right\}^+, NI)] \to 
F_0(NI) \to F_0(I)
\]
is zero.  But we have shown that every element of $F_0(NI)$ is a finite sum of 
elements of the form $\phi_*(g_i)$ for various maps $\phi$. Since $g_i
\in F_0( \mathbb{Z}\left\{x_1, \dots, x_{n_i}\right\}^+)$ is annihilated by $[N]$,
it follows that $(i_N)_*: F_0(NI) \to F_0(I)$ is zero, as claimed.
\end{proof}

The following is the main technical step used in the proof of our rigidity
result. 

\begin{proposition}[Axiomatic rigidity argument] 
Let $F: \nuzh \to \Sp$ be a projectively pseudocoherent functor. Suppose that: 
\begin{enumerate}
\item For each prime field $R$ and each $n$, we have $F( R\left\{x_1, \dots, x_n\right\}^+) =
0$. 
\item Given a short exact sequence $I' \to I \twoheadrightarrow \overline{I}$ in
$\nuzh$, the sequence $F(I') \to F(I) \to F(\overline{I})$ is a fiber sequence
of spectra. 
\item If $I \in \nuzh$ is nilpotent, then $F(I) = 0$. 
\end{enumerate}
Then $F = 0$.\label{pseudoargument}
\end{proposition} 
\begin{proof} Let $R$ be a prime field  and
restrict $F$ to $\nuh$ to obtain a functor $F_R: \nuh \to \Sp$. 
We observe first that $F_R$ 
is projectively pseudocoherent  (since the building blocks $I
\mapsto \Sigma^\infty_+ I^n$ for
projectively pseudocoherent functors $\nuh \to \Sp$ are independent of the base ring $R$). 
In particular, $F_R$ is left Kan extended from the subcategory
$(\nuh)_{\Sigma}$. By assumption, $F_R$ annihilates $(\nuh)_\Sigma$, so that
$F_R = 0$ on $\nuh$. 

It follows that $F$ annihilates 
any nonunital henselian algebra over a field. 
Moreover, if $I \in \nuzh$ is such that there exists $N \in \mathbb{Z}_{>0}$
with $NI = 0$, it follows that $F(I) = 0$: in fact, since $F$ preserves finite
products we reduce to the case where
$N  = p^r$ for some $r$, and then use the short exact sequence
$pI \to I \to I/p I$ where $pI$ is nilpotent. 
The resulting fiber sequence then shows that $F(I) \simeq F(I/pI) = 0$.

Suppose $F$ is not the zero functor. 
Let $d$ be minimal such that $\pi_d F \neq 0$. Then $\pi_d F$ is a
projectively finitely
generated functor $\nuzh \to \mathrm{Ab}$ which commutes with filtered
colimits; furthermore, $\pi_d F$ annihilates any $I \in \nuzh$ which is either
of bounded torsion or a $\mathbb{Q}$-vector space. Using 
Lemma~\ref{gabberlem} below, we find that $\pi_d F = 0$, a contradiction. 
\end{proof}

This completes the proofs of the main results from the present subsection. 
For the convenience of the reader, we recall that  the process
of Kan extension which appeared in our notion of projective pseudocoherence has an alternative description via simplicial resolutions
(which will not be used in the sequel). 
\begin{cons}
\label{kanextsifted}
Let $\mathcal{E}$ be an $\infty$-category that has sifted colimits, fix a base ring $R$ and a functor
\[ F: \nuh \to \mathcal{E},  \] and
consider the following condition on $F$:

\begin{itemize}
\item[(Kan)] $F$ is left Kan extended from the subcatgory $(\nuh)_\Sigma \subset \nuh$. 
\end{itemize}

Suppose $F$ satisfies (Kan). Then one can compute $F$ from its values on $(\nuh)_\Sigma$ as
follows. First, $F$ commutes with filtered colimits, and hence its value is
determined on all free objects in $\nuh$.  Next, if $I$ is an arbitrary
nonunital henselian $R$-algebra, then there exists  an augmented simplicial
object $I_\bullet:\Delta^\sub{op}_+ \to \nuh$ such that $I_{-1} = I$,
each $I_i$ is a free algebra in 
$\nuh$, and 
$|I_\bullet| \simeq I_{-1} = I$, i.e., $I_\bullet$ is a simplicial resolution
of $I$ by free objects; then the value of $F$ on $I$ is given by $F(I) = |F(I_\bullet)|$. 

The construction of such simplicial 
resolutions (and the independence of choices) is a general homotopical technique
going back to Quillen \cite{QuillenHA}, originally developed to build the cotangent
complex via simplicial commutative rings. Using the results of
\cite{QuillenHA} (which work in particular for any category satisfying (Law)
from Remark \ref{remark_Lawvere}, by \cite[Rmk.~1, pg.~II.4.2]{QuillenHA}), we
can make the category $\fun( \Delta^\sub{op}, \nur)$ of simplicial  nonunital henselian $R$-algebras into a model category where the weak equivalences and fibrations are those of underlying simplicial sets.  The resolution $I_\bullet$  of $I$ is then obtained as a cofibrant replacement in this model category of the constant simplicial object $I$.
This can also be 
phrased
using the language of nonabelian derived $\infty$-categories
\cite[Sec.~5.5.8]{HTT}. 

We record these observations for the future in the following corollary 
which does not mention simplicial non-unital henselian rings. 

\begin{corollary} 
\label{kanextcrit}
Let $\mathcal{E}$ be an $\infty$-category with sifted colimits, and let $F: \nuh \to \mathcal{E}$ be a functor. Then the following are equivalent: 
\begin{enumerate}
\item $F$ is left Kan extended from $(\nuh)_{\Sigma} \subset \nuh$. 
\item Whenever $I_\bullet$ is a simplicial object in $\nuh$ whose geometric
realization is equivalent to $I_{-1} \in \nuh$ (in particular, the homotopy type
of the underlying simplicial set of $I_\bullet$ is discrete), then the map 
$|F(I_\bullet)| \to F(I_{-1})$ is an equivalence.\footnote{Since the category
of nonunital henselian rings is the category of algebras over a monad on the
category of sets, there is always a canonical resolution of any object by free
objects; it would suffice to consider these resolutions.} 
Furthermore, $F$ commutes with filtered colimits. 
\end{enumerate}
\end{corollary} 
\end{cons}

\subsection{Pseudocoherence of $K, \TC$}
In this subsection, we show that $K$-theory and $\TC$ satisfy the
projective pseudocoherence property studied in subsection \ref{subsection_axiomatic}.
Given a space $X$, we write $C_*(X; \mathbb{Z}) = H\mathbb{Z}\otimes \Sigma^\infty_+ X$ for the singular chains on $X$
with $\mathbb{Z}$-coefficients, viewed as a (generalized) Eilenberg--MacLane spectrum.

\begin{proposition} 
\label{glpscoh}
The functor 
$ \nuzh
\to \Sp$
given by 
$I \mapsto C_*( BGL_\infty(\mathbb{Z} \ltimes I); \mathbb{Z})$  is
projectively
pseudocoherent. 
\end{proposition}
\begin{proof} 
First, by the homological stability results of Maazen and van der Kallen
\cite{vdK}, we can approximate $C_*(BGL_\infty(\mathbb{Z} \ltimes I);
\mathbb{Z})$ by 
$C_*(BGL_n(\mathbb{Z} \ltimes I);
\mathbb{Z})$ in any given finite range. It is crucial for us that this stability range is independent of
the choice of $I$: namely, the stability range depends on the Krull dimension of
the maximal spectrum of $\mathbb{Z} \ltimes I$, but since $I$ is contained
in the Jacobson radical this is the maximal spectrum of $\mathbb{Z}$. 

Thus, it suffices to show that for any $n$, the functor $I \mapsto
C_*(BGL_n(\mathbb{Z} \ltimes I); \mathbb{Z})$ is projectively pseudocoherent. We use the short
exact sequence of groups
\[ 1 \to GL_n(1 + I) \to GL_n( \mathbb{Z} \ltimes I) \to GL_n(\mathbb{Z}) \to 1,  \]
where $GL_n(1 + I)$ is, by definition, the kernel of the second map in the above
sequence. It follows that 
there is a $GL_n(\mathbb{Z})$-action on 
$C_*(BGL_n(1 + I);
\mathbb{Z})$ and we have
\begin{equation} \label{glnZorbit} C_*(BGL_n(\mathbb{Z} \ltimes I); \mathbb{Z})
\simeq C_*(BGL_n(1 + I);
\mathbb{Z})_{h GL_n(\mathbb{Z})}.  \end{equation}

We next argue that the functor 
$I \mapsto C_*(BGL_n(1 + I); \mathbb{Z})$ is projectively pseudocoherent. 
Note that as a set, $GL_n(1 + I)$ is naturally isomorphic to the
cartesian product  $I^{n^2}$ since $I$ is contained in
the radical of $\mathbb{Z} \ltimes I$. 
Now we observe that using the classical bar construction and the
isomorphism (of sets) 
$GL_n(1 + I) \simeq I^{n^2}$, that 
$C_*(BGL_n(1 + I); \mathbb{Z})$ is a geometric realization of
functors of the form $I \mapsto \mathbb{Z}[I^{n^2i}]$ for $i \geq 0$. Since 
the functor
$I \mapsto \mathbb{Z}[I^{n^2i}] 
 \simeq H \mathbb{Z} \otimes \Sigma^\infty_+
 \hom_{\nuzh}(\mathbb{Z}\left\{x_1, \dots, x_{n^2i}\right\}^+, I)$ is clearly 
projectively pseudocoherent, it follows that 
$I \mapsto C_*(BGL_n(1 +I); \mathbb{Z})$ is projectively pseudocoherent as desired. 

Finally, the group $GL_n(\mathbb{Z})$ admits a finite index normal subgroup $N \leq
GL_n(\mathbb{Z})$ such that $BN$ has the homotopy type of a finite CW complex,
by the existence of the Borel-Serre compactification (see \cite{Serresurvey}
for a survey). Since $(-)_{hGL_n(\mathbb{Z})} \simeq ((-)_{hN})_{h(G/N)}$, we conclude by two applications of Proposition \ref{general:pscoh} that taking $GL_n(\mathbb{Z})$-homotopy orbits preserves projective
pseudocoherence. Therefore, 
we conclude in view of the previous paragraph and \eqref{glnZorbit}. 
\end{proof} 

The following lemma is well-known:

\begin{lemma} 
\label{K0invhenspair}
Let $(R, I)$ be a henselian pair. Then $K_0(R) \simeq K_0(R/I)$. 
\end{lemma}
\begin{proof} 
In fact, we claim that isomorphism classes of finitely generated projective $R$
modules and isomorphism classes of finitely generated projective $R/I$-modules
agree. For one direction, if $M, N$ are finitely generated projective $R$-modules, then any
isomorphism $M/IM \simeq N/IM$ can be lifted to a map $M \to N$, which is
necessarily an isomorphism by Nakayama's lemma applied to the kernel; we use
here that  $I$ is contained in the
Jacobson radical of $R$. For the other direction, by lifting idempotents any
projective $R/I$-module lifts to $R$, cf. \cite[Tag 0D49]{stacks-project}. 
\end{proof} 

\begin{proposition} 
\label{Kpscoh}
The functor $\nuzh \to \Sp$ given by $I \mapsto K( \mathbb{Z} \ltimes I)$
is projectively pseudocoherent.
\end{proposition} 
\begin{proof} 
Since $I$ is henselian, we have $K_0( \mathbb{Z} \ltimes I) = \mathbb{Z}$ by Lemma \ref{K0invhenspair}; since the constant functor $\mathbb{Z}$ is projectively pseudocoherent, it remains to see 
that the functor $I \mapsto \tau_{\geq 1} K(\mathbb{Z} \ltimes I)$ is
projectively 
pseudocoherent. 
By Propositions  \ref{suspensionpscoh} and \ref{HZpscoh} and the
plus-construction description of $K$-theory (cf.~\cite[IV.1]{KBook} for an
account), it suffices to check that
the functor \[I \mapsto H \mathbb{Z} \otimes \Sigma^\infty_+ \Omega^\infty
\tau_{\geq 1} K(\mathbb{Z} \ltimes I) \simeq C_*(BGL_\infty(\mathbb{Z} \ltimes I);
\mathbb{Z})\] is projectively pseudocoherent, which follows from 
Proposition~\ref{glpscoh} above. 
\end{proof}

\begin{remark} 
In particular, Proposition \ref{Kpscoh} shows that the functor $\nuzh \to \Sp$,
$I \mapsto K( \mathbb{Z} \ltimes I)$ commutes with simplicial resolutions in
view of Construction~\ref{kanextsifted}.  The same argument shows
that this also holds in the larger category of nonunital local rings, even
without the henselian condition.  This observation, in various forms, plays an important role in the study of the local structure of $K$-theory. Compare, for instance \cite[Lemma
I.2.2]{Goodrel} for the nilpotent case and \cite[Ch.~III, Prop.~1.4.2]{dgm}
for a more general assertion  for radical pairs. 
\end{remark} 

We next carry out analogous arguments for topological cyclic homology. 
This is considerably simpler and does not rely on tools such as homological
stability; we will instead use the finiteness properties of cyclotomic spectra from Section \ref{section_finiteness}.
 
\begin{lemma} 
The Eilenberg--MacLane functor $\nuzh \to \Sp$ given by $I \mapsto HI$ is projectively pseudocoherent. 
\label{EMpscoh}
\end{lemma} 

Lemma~\ref{EMpscoh} is a direct consequence of the following result, in the
context of abelian groups. 

\begin{lemma} 
\label{Abpscoh}
Let $\mathrm{Ab}$ be the category of abelian groups. 
The functor $\mathrm{Ab} \to \Sp$ given by $A \mapsto HA$ is
$\mathrm{Latt}$-pseudocoherent for $\mathrm{Latt} \subset \mathrm{Ab}$ the
subcategory of finitely generated free abelian groups. 
\end{lemma} 
\begin{proof}

We can write functorially $HA = \varinjlim_n \Sigma^{-n}\Sigma^\infty  K(A, n)$
for $K(A, n)$ the $n$th Eilenberg--MacLane space for $A$. 
This colimit stabilizes in any ranges of degrees by the Freudenthal suspension
theorem, so it suffices to show that the functor 
$A \mapsto \Sigma^\infty K(A, n)$ is projectively pseudocoherent. But this
follows from the iterated bar construction which gives a functorial model for
$K(A, n)$ as the colimit of an $n$-fold simplicial space each of whose terms is
a product of copies of $A$. 
\end{proof}

 In the setting of $H\mathbb{Z}$-modules rather than
spectra, 
Lemma~\ref{Abpscoh}
is an unpublished result of Deligne, which states that any abelian group $A$ has a
functorial resolution by free abelian groups all of whose terms are finite direct sums of the
form $\mathbb{Z}[A^n]$. One can deduce Deligne's result from the
above using the finiteness of the stable homotopy groups of spheres, and vice
versa; compare Proposition \ref{HZpscoh}. 
For more discussion and a presentation of essentially the same argument in more classical terms of homological algebra, cf.~\cite[Appendix to
Lec.~IV]{Condensed}. 

\begin{lemma} 
\label{THHps}
The functor $\nuzh \to \Sp$ given by $I \mapsto \THH( \mathbb{Z} \ltimes I)$ is
projectively pseudocoherent.
\end{lemma} 
\begin{proof} 
Using the cyclic bar construction for $\THH$, this follows because all the terms
$(H(\mathbb{Z} \ltimes I))^{\otimes k}$ (of which $\THH$ is a geometric
realization) are projectively pseudocoherent in view of Lemma~\ref{EMpscoh}
and Proposition~\ref{general:pscoh}.
\end{proof}

\begin{proposition} 
\label{TCpscoh}
The functor $\nuzh \to \Sp$ given by $I \mapsto \TC( \mathbb{Z} \ltimes I)/p$ 
is projectively pseudocoherent. 
\end{proposition} 
\begin{proof} 
This now follows 
from Lemma~\ref{THHps} and Proposition~\ref{strong:finiteness}. Indeed, for
any of the functors $F$ in the statement of the latter, we have that $I \mapsto F(
\THH(\mathbb{Z} \ltimes I))$ is projectively pseudocoherent (note that taking
$S^1$-homotopy orbits preserves pseudocoherence thanks to
Lemma~\ref{general:pscoh}, as the skeleta of $BS^1$ are finite complexes), and then we can approximate
$\TC/p$ in any range. 
\end{proof} 

\subsection{Equal characteristic case}
\label{subsec:equalchar}

We next prove a special case of the rigidity result in equal characteristic. 
We begin by reviewing results of Geisser--Levine \cite{GL} and
Geisser--Hesselholt \cite{GH} in a formulation that will be convenient for us.
Compare also \cite{Geissersurvey} for a survey treatment. 
Our main result here (Proposition~\ref{rigidity:henspaircharp}) is a special
case of the rigidity statement in the case of a smooth henselian pair over a
perfect field of characteristic $p$. 
We keep the notation and terminology of the introduction. In particular, we will
use the cyclotomic trace $K(R) \to \TC(R)$ for a ring $R$, and denote by
$\knf(R)$ the homotopy fiber of this map. 
Recall also 
(Definition~\ref{cartier}) the inverse Cartier operator on differential forms. 

\begin{definition} 
For an $\mathbb{F}_p$-algebra $R$, we let 
$\nu^n(R) = \ker( 1 - C^{-1}: \Omega^n_R \to \Omega^n_R / d \Omega^{n-1}_R)$. We
will also write this as $\Omega^n_{R, \mathrm{log}}$. 
We also let $\widetilde{\nu^n}(R)$ be the cokernel of $1 - C^{-1}$. 
\end{definition} 

The construction $\nu^n$ thus defines a sheaf for the \'etale topology.
To see this, we observe that it is a kernel of a map between two objects, both
of which are quasi-coherent over the Frobenius twist by the next lemma. By contrast,
$\widetilde{\nu^n}$ vanishes locally in the \'etale topology since $1 -
C^{-1}$ is surjective locally in the \'etale topology.  
Moreover, $\nu^0$ is the constant sheaf $\mathbb{F}_p$ and $\nu^n(R) \subset
\Omega^n_R$ is the subgroup generated \'etale locally by differential forms $d
\log x_1 \wedge \dots \wedge d \log x_n$ for $x_1, \dots, x_n$ units (cf.
\cite[Thm.~2.4.2]{illusie-derham-witt} for the smooth case, and \cite[Cor.~4.2]{Morrow-HW} in general).

\begin{lemma} 
Let $R \to S$ be an \'etale map of $\mathbb{F}_p$-algebras. 
Let $R^{(1)} \rightarrow R, S^{(1)} \rightarrow S$ be the Frobenius twists of $R$ and $S$.\footnote{Thus $R^{(1)}$ consists of formal expressions of the form ``$a^p$", and the map to $R$ is ``$a^p$" $\mapsto a^p$.  More straightforwardly, we can set $R^{(1)}=R$ and take the map $R^{(1)}=R\rightarrow R$ to be the Frobenius.}
Then: 
\begin{enumerate}
\item The map $\Omega^n_R \otimes_R S \to \Omega^n_S$ is an isomorphism.  
\item The 
de Rham differentials $d: \Omega^{n-1}_R \to \Omega^n_R, \Omega^{n-1}_S \to
\Omega^n_S$ are respectively $R^{(1)}, S^{(1)}$-linear so that the quotients 
$\Omega^n_R/d \Omega_R^{n-1} , \Omega^n_S/d \Omega_S^{n-1} $ inherit the
structure of $R^{(1)}, S^{(1)}$-modules respectively.
The 
map $\Omega^n_R/d \Omega_R^{n-1} \otimes_{R^{(1)}} S^{(1)} \to
\Omega^n_S/d\Omega^{n-1}_S$ is an isomorphism. 
\end{enumerate}
\label{qcohomega}
\end{lemma} 
\begin{proof} 
Part (1) is standard. Since $R \to S$ is \'etale, the natural square linking
$R^{(1)}, R, S^{(1)}$ and $S$ is cocartesian (cf.~\cite[Tag
0EBS]{stacks-project}), which now implies
(2). 
\end{proof}

We will need this definition in 
light of the following fundamental results about the structure of $K$-theory
and $\TC$ for ind-smooth $\mathbb{F}_p$-algebras. 
\begin{theorem}
Let $R$ be an ind-smooth $\mathbb{F}_p$-algebra.  Then, for each $n \geq 0$, one has:
\begin{enumerate}
\item (Geisser--Levine \cite{GL}) There is a natural map
$\pi_n ( K(R)/p) \rightarrow \nu^n(R)$, which is an isomorphism if $R$ is local. 
\item (Geisser--Hesselholt \cite{GH}) There is a functorial exact sequence
$0 \to \widetilde{\nu}^{n+1}(R) \to \pi_n ( \TC(R)/p) \to  {\nu}^{n}(R) \to 0$.
Furthermore, under these identifications, the composite 
of $\pi_n  ( K(R)/p) \to \pi_n (\TC(R)/p) \to \nu^n(R) $ (where the first map
arises from the cyclotomic trace) is the map of (1).
\item If $R$ is local, we have a functorial identification $\pi_n ( \knf(R)/p) \simeq \widetilde{\nu}^{n+2}(R)$.
\end{enumerate}
\label{GHGL}
\end{theorem} 
\begin{proof} 
First we note that it suffices to prove this theorem in the case where $R$ is essentially of finite type over $\mathbb{F}_p$, by extending using filtered colimits.
Then the first assertion follows from \cite[Thm.~8.3]{GL}, 
which shows that on smooth varieties over $\mathbb{F}_p$, the Zariski
sheafification of  
the presheaf
$\pi_n(K(\cdot)/p)$
is given by $\nu^n$.

The exact sequence describing $\pi_n( \TC(R)/p)$ actually follows from
\eqref{thisdiag}, though we will give a slightly different proof of the
stronger second claim. 
Namely, we use the results of Geisser--Hesselholt \cite[Thm.~4.2.6]{GH}
that on smooth quasi-compact quasi-separated schemes over $\mathbb{F}_p$,
the homotopy group sheaves in the \emph{\'etale} topology of $\TC/p$ identify with those of $K/p$ via the cyclotomic trace, and hence identify with the $\nu^n$; and moreover one has an \'etale descent spectral
sequence starting from the \'etale cohomology of $\nu^n$ and converging to
$\pi_*( \TC/p)$. 
In the \'etale topology, there is a short exact sequence of sheaves $0 \to \nu^n \to
\Omega^n \xrightarrow{1 - C^{-1}} \Omega^n/d \Omega^{n-1} \to 0$, and the
second and third terms are quasi-coherent (either over the structure sheaf or
its Frobenius twist by Lemma~\ref{qcohomega}), so have no higher cohomology on affines.  
It follows that $H^0( (\spec R)_{et}, \nu^n) = \nu^n(R), H^1( (\spec R)_{et}, \nu^n) =
\widetilde{\nu}^n(R)$, and the \'etale descent spectral sequence thus implies
the second claim.  The third claim follows from the first two. 
\end{proof}

\begin{lemma} 
\label{solveeq:hens}
Let $(R, I)$ be a henselian pair. Let $s \in I$ and
let $n \geq  1$. 
Then the equation $x - s x^n = 1$  can be solved in $R$. 
\end{lemma}
\begin{proof} 
The equation has a simple root in $R/I$ (namely, $x = 1$), which
therefore admits a lift to $R$ by definition of henselian.
\end{proof}

\begin{proposition} 
\label{toyFprigidityTC}
Let $(R, I)$ be a henselian pair of $\mathbb{F}_p$-algebras.
Then for each $n$, the map $\nu^n(R) \to \nu^n(R/I)$ is surjective and the map
$\widetilde{\nu}^{n}(R) \to \widetilde{\nu}^n(R/I)$ is an isomorphism. 
\end{proposition} 
\begin{proof}
We use the commutative diagram
\[ \xymatrix{
V_0 \ar[d] \ar[r] &  V_1 \ar[d]  \\
\Omega^n_R \ar[r]^-{1 - C^{-1}} \ar[d]  &  \Omega^n_R/ d \Omega^{n-1}_R  \ar[d]  \\
\Omega^n_{R/I} \ar[r]^-{1 - C^{-1}} &  \Omega^n_{R/I}/ d \Omega^{n-1}_{R/I}
}\]
where we define $V_0, V_1$ to be the kernels of the surjective vertical maps
$\Omega^n_R \to 
\Omega^n_{R/I}$ and 
$\Omega^n_R/ d \Omega^{n-1}_R  \to \Omega^n_{R/I}/ d \Omega^{n-1}_{R/I}$. 
We will show that the map $V_0 \to V_1$ is surjective. This easily implies 
the desired conclusions about the maps $\nu^n(R) \to \nu^{n}(R/I)$ and
$\widetilde{\nu}^n(R) \to \widetilde{\nu}^n(R/I)$ thanks to the snake lemma. 

Consider a differential form 
$\omega = a dx_1 dx_2 \dots dx_n \in \Omega^n_R$ such that one of $\left\{a, x_1, \dots,
x_n\right\}$ belongs 
to $I$ (here $n = 0$ is allowed). The image of such a class in 
$\Omega^n_R/d \Omega_R^{n-1}$ belongs to $V_1$, and $V_1$ is generated by such
classes. 
For $u \in R$, 
we have $$(1 - C^{-1}) ( u  \omega)  = (u - u^p a^{p-1} x_1^{p-1} \dots x_n^{p-1})
\omega.$$ Since $(R, I)$ is a henselian pair and $a^{p-1} x_1^{p-1} \dots
x_n^{p-1}\in I$, we can
choose 
$u \in R$ such that $u - u^p a^{p-1} x_1^{p-1} \dots x_n^{p-1} = 1$,
using Lemma~\ref{solveeq:hens}.
The class $u \omega \in \Omega^n_R$ belongs to $V_0$ and has image given by
$\omega$; 
so $\omega$ is in the image of $V_0 \to V_1$, as desired.
\end{proof} 

\begin{proposition} 
\label{rigidity:henspaircharp}
Let $(R, \mathfrak{m})$ be an ind-smooth henselian local $\mathbb{F}_p$-algebra with residue field
$k$. 
Then the map $\knf(R) \to \knf(k)$ becomes an
equivalence modulo $p$. 
\end{proposition} 
\begin{proof} 
By Theorem~\ref{GHGL}, it suffices to show that the map 
\( \widetilde{\nu}^n(R) \to \widetilde{\nu}^n(k)   \)
is an isomorphism. 
This follows from Proposition~\ref{toyFprigidityTC}. 
\end{proof}

\subsection{Proof of the main result}

In this subsection, we prove the main result of this paper, Theorem~\ref{mainthm}.
Our goal is to show that the construction $R \mapsto \knf(R)$ with mod $p$
coefficients is invariant under taking the quotient by a henselian ideal. 

A key ingredient in the proof will be to use results of
Geisser--Hesselholt \cite{GHexc} about excision in $K$-theory and topological
cyclic homology to relate the setup of a general henselian pair to the 
case $(\mathbb{Z} \ltimes I, I)$, where we can appeal to the finiteness results of the
previous subsection.

We start by discussing these excision results.  For this, we need to invoke the
non-connective variant $\mathbb{K}^{\operatorname{inv}}$ of
$K^{\operatorname{inv}}$, defined as the fiber of the cyclotomic trace
$\mathbb{K}\rightarrow \TC$ from non-connective $K$-theory.  Recall also the standard notation $K(R,I) = \operatorname{fib}(K(R)\rightarrow K(R/I))$, and similarly for any other functor on rings.

The following was proved in the rational case in \cite{cortinas}, with finite
coefficients in \cite{GHexc}, integrally under some assumptions
in \cite{DKexcint}, and in full generality in \cite{land-tamme}.
Here commutativity is not necessary. 

\begin{theorem}[Corti\~nas; Geisser--Hesselholt; Dundas--Kittang; Land--Tamme] 
\label{DKGHC}
Suppose $R$ is a unital associative ring, $I\subset R$ a two-sided ideal, and
$f:R\rightarrow S$ a homomorphism such that $f$ restricts to an isomorphism
from $I$ to a two-sided ideal $J$ of $S$ (so one has a Milnor square, cf.~ 
Definition~\ref{milnorsquare}).  Then the induced map
$$\mathbb{K}^{\operatorname{inv}}(R,I)\rightarrow \mathbb{K}^{\operatorname{inv}}(S,J)$$
is an equivalence.
\end{theorem} 

This can be read as saying that $\mathbb{K}^{\operatorname{inv}}(R,I)$ ``only
depends on $I$."  To make this precise, one can define
$\mathbb{K}^{\operatorname{inv}}(I) =
\mathbb{K}^{\operatorname{inv}}(\mathbb{Z}\ltimes I,I)$ for any non-unital
ring $I$.  Clearly the new $\mathbb{K}^{\operatorname{inv}}$
restricts to the old one on unital rings; moreover, Theorem  \ref{DKGHC} implies 
that $\mathbb{K}^{\operatorname{inv}}$ sends short exact sequences
$I'\rightarrow I\twoheadrightarrow I''$ of non-unital rings to fiber sequences
of spectra.  In particular, $\mathbb{K}^{\operatorname{inv}}(I) \simeq
\mathbb{K}^{\operatorname{inv}}(R,I)$ whenever $I$ embeds as a two-sided ideal
in $R$, in view of Theorem~\ref{DKGHC} applied to the map $(\mathbb{Z} \oplus I, I) \to (R, I)$ of pairs of
rings with an ideal.

For our purposes, 
we will need the analog of Theorem~\ref{DKGHC} for connective $K$-theory, in
the context of non-unital henselian (commutative) rings.   We can make the
switch thanks to the following proposition, whose proof uses another excision
theorem due to Bass-Milnor-Swan:

\begin{proposition} 
Suppose $(R,I)$ is a henselian pair, and $f:R\rightarrow S$ is a ring homomorphism such that $f$ restricts to an isomorphism from $I$ to an ideal $J$ of $S$.  Then the square of spectra
\begin{equation} \label{Ksquare} \xymatrix{
K(R, I) \ar[d]  \ar[r] &  K(S, J) \ar[d]  \\
\mathbb{K}(R, I) \ar[r] &  \mathbb{K}(S, J)
}\end{equation}
is cartesian. 
\end{proposition} 
\begin{proof} 
Let $F$ be the fiber of the top horizontal map and let $\mathbb{F}$ be the fiber
of the bottom horizontal map. Note that the maps $K(R) \to \mathbb{K}(R), K(R/I)
\to \mathbb{K}(R/I)$, etc.\ are equivalences in degrees $ \geq 0$; thus, taking
fibers, the maps $K(R, I) \to \mathbb{K}(R, I)$ and $K(S, J) \to \mathbb{K}(S,
J)$ are equivalences in degrees $\geq 0$. Taking fibers again, we find that the
map $F \to \mathbb{F}$ is an equivalence in degrees $\geq 0$.

On the other hand, $\mathbb{F}$ is concentrated in degrees $\geq 0$ (even
$\geq 1$) by the
excision theorem of Bass and Bass--Heller--Swan, \cite[Thm.~XII.8.3]{Bass}.  Thus it suffices to show that $F$ is also concentrated in degrees $\geq 0$.

But since $(R, I)$ and $(S, J)$ are henselian pairs, it follows by
Lemma~\ref{K0invhenspair} that the maps
$K_0(R) \to K_0(R/I), K_0(S) \to K_0(S/J)$ are isomorphisms and the maps
$K_1(R) \to K_1(R/I), K_1(S) \to K_1(S/J)$ (which are the abelianizations of
$GL_\infty(R) \to GL_\infty(R/I), GL_\infty(S) \to GL_\infty(S/J)$)  are surjections; thus $K(R, I)$ and
$K(S, J)$ are concentrated in degrees $\geq 1$, so that $F$ is concentrated in
degrees $\geq 0$, as desired.
\end{proof}

\begin{corollary} 
\label{replaceexcision}
Suppose $(R,I)$ is a henselian pair, and $f:R\rightarrow S$ is a ring homomorphism which restricts to an isomorphism from $I$ to an ideal $J\subset S$. Then the map $\knf(R, I) \to \knf(S, J)$ is  an equivalence. 
\end{corollary} 
\begin{proof} 
Combining with the homotopy cartesian square \eqref{Ksquare} between relative
connective and nonconnective $K$-theory, we find that the result now follows from
Theorem~\ref{DKGHC}.  
\end{proof} 

Again, this can be interpreted as saying that $K^{\operatorname{inv}}$ makes sense for non-unital henselian rings.  We can now state and prove the main result of this paper.

\begin{theorem} 
\label{mainthm}
Let $(R, I)$ be a henselian pair. Then for any prime number $p$,  the map
$\knf(R) \to \knf(R/I)$ becomes an equivalence modulo $p$. 
 Equivalently,
the map 
$K(R, I) \to \TC(R, I)$ becomes an equivalence modulo $p$.
\end{theorem} 

\begin{proof} 
By Corollary~\ref{replaceexcision}, it suffices to consider the case where
$(R, I) = (\mathbb{Z} \ltimes I, I)$ where $I$ is a nonunital henselian ring. 
We now consider the functor 
\[ F: \nuzh \to \Sp, \quad 
F(I) = \knf(\mathbb{Z} \ltimes I, I)/p.
\]
By Propositions~\ref{Kpscoh} and \ref{TCpscoh}, $F$ is a projectively
pseudocoherent functor.  Implicit here is the statement
(due to Quillen) that the
$K$-groups of $\mathbb{Z}$ are finitely generated, as are the mod $p$ homotopy
groups of $\TC(\mathbb{Z})$ (both of which follow from evaluating the projectively
pseudocoherent functors $I \mapsto K(\mathbb{Z} \ltimes I), \TC(\mathbb{Z}
\ltimes I)/p$ at $I = 0$). 

The assertion of the theorem is that $F=0$.  We will show this by checking that $F$ satisfies the hypotheses of
Proposition~\ref{pseudoargument}.

First observe that $F$ sends a short exact sequence $I' \to I
\twoheadrightarrow \overline{I}$ of nonunital henselian rings to a fiber sequence of spectra. 
To see this, 
we consider the diagram
\[ \xymatrix{ F(I') \ar[d]^{\mathrm{id}}   \ar[r] & F(I) \ar[d]  \ar[r] &  F(\overline{I})
\ar[d]   \\
F(I') \ar[r] &  \knf( \mathbb{Z} \ltimes I)/p \ar[r] &  \knf( \mathbb{Z} \ltimes 
\overline{I})/p . }\]
Note that the right-hand square is homotopy cartesian, so the top row is a
fiber sequence if and only if the bottom row is a fiber sequence. But the bottom
row is a fiber sequence by Corollary~\ref{replaceexcision} applied to the
homomorphism $(\mathbb{Z}\ltimes I',I')\rightarrow (\mathbb{Z}\ltimes I, I')$.

Next, we claim that if $I'  \in
\nuzh$ is nilpotent, then $F(I') = 0$. This follows from  
Theorem~\ref{DGMthm}
which shows that $\knf( \mathbb{Z} \ltimes I', I') = 0$. 
To complete the proof, i.e., to verify the conditions of
Proposition~\ref{pseudoargument}, we need to check that 
\[ F( k\left\{x_1, \dots, x_n\right\}^+) = 0,  \]
whenever $k$ is a prime field. 
We write $k\left\{x_1, \dots, x_n\right\}$ for the henselization of the
polynomial ring $k[x_1, \dots, x_n]$ at $(x_1, \dots, x_n)$, so
$k\left\{x_1, \dots, x_n\right\}^+ \subset k\left\{x_1, \dots, x_n\right\}$ is the maximal ideal. 
Using Corollary~\ref{replaceexcision}, it suffices to show that
the map $\knf(k\left\{x_1, \dots, x_n\right\}) \to \knf(k)$ is an equivalence
modulo $p$. 
There are two cases for this. 
If $\mathrm{char}(k) \neq p$, this follows from Gabber rigidity. 
If $\mathrm{char}(k) = p$, 
this follows from Proposition~\ref{rigidity:henspaircharp}. 
\end{proof} 

\begin{remark}
Recall that Gabber's proof \cite{gabber} of rigidity is cleanly separated into two halves: the
first half reduces the general case to the case of henselizations of smooth
algebras over a field at a rational point, and the second half proves that case.
In our proof above, we only need to invoke the second half of Gabber's work,
which is also covered by Gillet--Thomason \cite{GT}.

However, our entire line of reasoning is in some sense modeled on the first half
of Gabber's proof. (The exceptions are Proposition \ref{rigidity:henspaircharp},
which is the new characteristic $p$ ingredient, and the commutation of $\TC/p$
with filtered colimits, which is a necessary technical statement.)  So we are
not really avoiding the first half of Gabber's proof, just explicating it in a
modified context.  Let us note in particular that the projective
pseudocoherence of $K$-theory is our replacement for Suslin's ``method of
universal homotopies" from \cite{Suslinlocal}, which has been crucial to many of the known
rigidity results in $K$-theory.
\end{remark}

\begin{remark} 
\label{rem:fpcaseeasier}
While the 
the full strength of pseudocoherence appears to be necessary to obtain results
for $\mathbb{Z}$-algebras, the result for henselian pairs of $\mathbb{F}_p$-algebras 
follows from Proposition~\ref{rigidity:henspaircharp} and the observation that
$K, \TC/p$ commutes with filtered colimits of $\mathbb{F}_p$-algebras. 
One can then extend the result to $\mathbb{Z}_p$-algebras using the $p$-adic
continuity statement of \cite{GHlocal} (and a similar filtered colimit
argument). Thus, for $\mathbb{Z}_p$-algebras at least, one can prove the main
result purely using the classical approach to $\TC$ and ingredients predating
\cite{nikolaus-scholze}.  
\end{remark} 

\begin{remark} 
Theorem~\ref{mainthm} is false integrally; that is, the integral
Dundas--Goodwillie--McCarthy theorem does not hold for henselian pairs. 
As an example, 
let $(R, \mathfrak{m})$ be a henselian local $\mathbb{F}_p$-algebra with
residue field $\mathbb{F}_p$. 
Then $\TC(R), \TC(\mathbb{F}_p)$ are $p$-complete spectra, so 
$\TC_1(R, \mathfrak{m})$ is a derived $p$-adically complete abelian group. 
However, $K_1(R, \mathfrak{m}) =\mathrm{ker}(R^{\times} \to
\mathbb{F}_p^{\times})$ which will essentially never be derived $p$-complete. 
For instance, consider the power series ring $R_0 =  \mathbb{F}_p[[x]]$ and let
$R = (R_0)_{\mathrm{perf}}$ be the perfection of $R_0$. 
We have a surjection $R \to \mathbb{F}_p$, whose kernel is the henselian ideal
$I = \bigcup (x^{1/p^n})$. 
In this case, $\mathrm{ker}(R^{\times} \to \mathbb{F}_p^{\times})$ is a
nonzero
$\mathbb{Z}[1/p]$-module, which in particular is not derived $p$-complete. 
\end{remark} 

\section{Continuity and pro statements in algebraic $K$-theory}
In this section we consider various applications to the continuity problem in algebraic $K$-theory and to the related problem of describing the pro $K$-theory of formal schemes.
In particular we show that, under mild hypotheses, algebraic $K$-theory with finite coefficients is continuous for complete noetherian rings (Theorem~\ref{Kcont}). 
We also show that algebraic $K$-theory satisfies a ``derived'' form of $p$-adic
continuity for rings that are henselian along $p$ (Theorem~\ref{derivedpadiccont}), extending
results of Geisser--Hesselholt \cite{GHlocal}. Finally, we prove a pro version of
the Geisser--Levine \cite{GL} theorem on the $p$-adic $K$-theory of regular local
$\mathbb{F}_p$-algebras to describe (under mild hypotheses) the pro abelian groups 
$\{K_*(A/I^s; \mathbb{Z}/p^r\mathbb{Z})\}_s$ when $A$ is a regular local $\mathbb{F}_p$-algebra and $I \subset A$ is any ideal, extending results of Morrow \cite{Morrow-HW}.

\subsection{Continuity and $K$-theory}

In this and the next subsection, we consider the following classical continuity
question in $K$-theory.

\begin{question} 
Let $R$ be a ring and $I$ be an ideal. 
How close is the map \begin{equation} \label{contmap}K(R) \To \varprojlim_s
K(R/I^s) \end{equation} to being an equivalence?
\end{question} 

In order for this question to be reasonable, we should assume that $R$ is
$I$-adically complete, or at least that $(R, I)$ forms a henselian pair. This question has been considered by various authors, and the above map has notably been shown to be an equivalence modulo $p$ in the following cases, which we order historically:
\begin{enumerate}
\item $R$ a complete discrete valuation ring of mixed characteristic $(0,q)$, $q\neq p$, with $I$ being the maximal ideal (Suslin \cite{Suslinlocal}).
\item $R$ a complete discrete valuation of mixed characteristic $(0,p)$, with $I$ being the maximal ideal (Panin \cite{Panin});
\item $(R,I)$ a henselian pair and $p$ invertible in $R$ (Gabber rigidity \cite{gabber}; this subsumes case (1)).
\item $R$ a complete discrete valuation ring of equal characteristic $p$ with
perfect residue field, with $I$ being the maximal ideal, in degrees $\leq 4$
(Dundas \cite{Dundascont}).
\item $R=A[[x_1,\dots,x_n]]$ and $I=(x_1,\dots,x_n)$, where $A$ is any $F$-finite, regular, local $\bb F_p$-algebra (Geisser--Hesselholt \cite{GHcomplete}; this subsumes case (4)).
\item $R$ any ring of finite stable rank in which $p$ is a non-zero-divisor and which is henselian along $I=pR$ (Geisser--Hesselholt \cite{GHlocal})
\item $R$ any $F$-finite, regular, local $\bb F_p$-algebra which is complete with respect to an ideal $I\subseteq R$ such that $R/I$ is ``generalised normal crossings'' (Morrow \cite{Morrow-HW}; this subsumes case (5)).
\end{enumerate}

In order to apply our earlier results to this question, we observe that if $(R, I)$ is a henselian pair  then Theorem~\ref{mainthm} implies that 
the map \eqref{contmap} becomes an equivalence modulo $p$ if and only if the
corresponding map $\TC(R) \to \varprojlim_s \TC(R/I^s)$ becomes an equivalence
modulo $p$. 
This latter question is often more tractable and has been carefully
studied in the recent work of Dundas--Morrow \cite{DundasMorrow}, who show that such
continuity for topological cyclic homology holds quite generally under the assumption of
$F$-finiteness. 

\begin{definition} 
\label{Ffinite}
An $\mathbb{F}_p$-algebra $R$ is said to be \emph{$F$-finite} if the absolute Frobenius
map $R \to R$ is finite, in other words if $R$ is a finitely generated module over its subring of $p^\sub{th}$-powers.
\end{definition}

Under $F$-finiteness, many additional finiteness properties follow; we refer to \cite{DundasMorrow} for more details. For
example, if $R$ is an $F$-finite, noetherian $\bb F_p$-algebra, then the homotopy groups  $\pi_nL_{R/\mathbb{F}_p}$ of the cotangent complex are finitely generated $R$-modules for all $n$ \cite[Cor. 3.8]{DundasMorrow}; in particular, the (algebraic) module of K\"ahler differentials $\Omega^1_{\bb F_p[[t]]/\bb F_p}$ is a free $\bb F_p[[t]]$-module of rank one, whereas the analogous construction for a characteristic zero field is much harder to control and is not $t$-adically separated.

In this subsection, we simply combine the Dundas--Morrow results on topological cyclic homology with
Theorem~\ref{mainthm} to give a general answer to the above continuity question in $K$-theory. Since the results of \cite{DundasMorrow} are stated only for $\mathbb{Z}_{(p)}$-algebras, we begin with
a brief detour.

\begin{lemma} \label{modplemma0}
Let $R \to R'$ be a map of commutative rings such that the map $R
\lotimes_{\mathbb{Z}} \mathbb{F}_p \to R' \lotimes_{\mathbb{Z}} \mathbb{F}_p$ is
an equivalence. Then $\THH(R) \to \THH(R')$ is a mod $p$ equivalence.
\end{lemma} 
\begin{proof} 
The hypothesis is equivalent to saying that $HR\rightarrow HR'$ is a mod $p$
equivalence of spectra.  It follows that $(HR)^{\otimes n}\rightarrow
(HR')^{\otimes n}$ is a mod $p$ equivalence for all $n\geq 0$, and hence that
$\THH(R)\rightarrow \THH(R')$ is a mod $p$ equivalence, since mod $p$ equivalences are preserved under colimits.
\end{proof} 

\begin{lemma} 
\label{modplemma}
Let $R$ be a noetherian ring and let $\widehat{R}_p$ be its $p$-adic completion
(which is also the derived $p$-adic completion, as $R$ has bounded $p$-power
torsion). Then the map 
$\THH(R) \to  \THH(\widehat{R}_p)$ is a mod $p$ equivalence. 
\end{lemma} 
\begin{proof} 
A noetherian ring has bounded $p$-torsion, which implies that the map $R \to \widehat{R}_p$ induces an equivalence $R
\dotimes_\bb Z \mathbb{F}_p \simeq \widehat{R}_p \dotimes_\bb Z
\mathbb{F}_p$.  Thus the statement follows from Lemma \ref{modplemma0}.
\end{proof} 

The following is a slight extension of the continuity results of Dundas--Morrow
\cite{DundasMorrow}, who treated the case in which $R$ is a $\mathbb{Z}_{(p)}$-algebra.

\begin{proposition} 
\label{THHcont}
Let $R$ be a noetherian ring which is complete along an ideal $I\subseteq R$, and suppose that 
$R/pR$ is $F$-finite.
Then the map $\THH(R) \to \varprojlim_s \THH(R/I^s)$ is an equivalence modulo $p$.
\end{proposition} 
\begin{proof} 
First we note that $\hat R_p = \varprojlim_s \widehat{R/I^s}_p$ and $\widehat{R/I^s}_p = \hat R_p /I^s\hat R_p$.  Indeed, $R$ being noetherian, we have that on finitely generated $R$-modules the $p$-adic completion identifies with the derived $p$-completion. It therefore preserves short exact sequences and sequential inverse limits along surjective transition maps, as these are examples of derived limits.

In particular, $\hat R_p$ is $I\hat
R_p$-adically complete. Therefore, the canonical map
\[ \THH(\widehat{R}_p)/p \To \varprojlim_s \THH( \widehat{R}_p/I^s
\widehat{R}_p)/p\] is an equivalence by \cite[Thm.~4.5]{DundasMorrow}. But by the above, each quotient $\hat R_p/I^s\hat
R_p$ coincides with the $p$-adic completion of $R/I^s$, so that applying
Lemma~\ref{modplemma} to $R$ and to each $R/I^s $ completes the proof.
\end{proof}

We can now state and prove our main result of the subsection, which resolves the continuity question in algebraic $K$-theory for all complete noetherian rings satisfying an $F$-finiteness hypothesis.

\begin{theorem} 
\label{Kcont}
Let $R$ be a noetherian ring which is complete along an ideal $I$, and suppose that 
$R/pR$ is $F$-finite. Then the map $K(R) \to \varprojlim_s K(R/I^s)$ is an
equivalence modulo $p$. 
\end{theorem} 
\begin{proof} 
Using Theorem~\ref{mainthm}, the result reduces to the
statement that $\TC(R)/p \to \varprojlim_s \TC(R/I^s)/p$ is an equivalence.  By Remark~\ref{TCinvlim}, this in turn follows from the analogous statement for $\THH$, which is given by Proposition~\ref{THHcont}. 
\end{proof} 

\begin{remark}
\label{henspairequivtocontinuity}
One can see that Theorem~\ref{Kcont} is in fact equivalent to Theorem~\ref{mainthm} (our main theorem on henselian pairs).  Indeed, suppose Theorem~\ref{Kcont} is known, and let $(R,I)$ be a henselian pair.  We want to show that $\knf(R)/p\rightarrow \knf(R/I)/p$ is an equivalence.  Since $K$ and $\TC/p$ commute with filtered colimits, so does $\knf/p$, hence we can assume $(R,I)$ is the henselization of a finite type $\mathbb{Z}$-algebra at an ideal.  

Let $\hat{R}$ denote the $I$-adic completion of $R$. 
By N\'eron--Popescu desingularization \cite{Popescu1, Popescu2} (see also
\cite[Tag 07BW]{stacks-project}), applicable in view of the geometric regularity of $R \to
\hat{R}$ (see \cite[7.8.3(v)]{EGA_IV_2}), the map $R\rightarrow \hat{R}$ is
ind-smooth, i.e., we can write $\hat{R}$ as a filtered colimit of smooth
$R$-algebras.
Given a smooth $R$-algebra $A$ and a map $A \to \hat{R}$ of
$R$-algebras, it follows
that the map $R \to A$ admits a section by Elkik's theorem
(Theorem~\ref{elkikthm}). 
Therefore,  we deduce that $R\rightarrow \hat{R}$ is a filtered colimit
of split injections.  Thus the claim for $(R,I)$ (i.e., that $\knf(R)/p \to
\knf(R/I)/p$ is an equivalence) will follow from the claim for
$(\hat{R}, I \hat{R})$.  (This is a standard Artin approximation argument.)

However, by the above continuity in $K$ and $\TC$, we have
$$\knf(\hat{R})/p\overset{\simeq}{\rightarrow}\varprojlim
\knf(R/I^n)/p.$$  On the other hand, the right hand side is a constant limit
with value $\knf(R/I)/p$, by Theorem~\ref{DGMthm}.  This gives the claim for
$(\hat{R}, I\hat{R})$, and therefore in general, by the above argument.
\end{remark}

We finish the subsection by checking that the $F$-finiteness hypothesis in the
above theorem appears to be necessary. This arises from a well-known problem in
Milnor $K$-theory, namely that the complexity of symbols modulo powers of the
ideal in question can increase without bound; to make this precise we closely
follow the exposition of \cite[App.~B]{BlochEsnaultKerz2013}, where Bloch--Esnault--Kerz proved an analogous discontinuity result in characteristic zero.

\begin{theorem}
Let $k$ be an field of characteristic $p$ which is not $F$-finite, i.e., any
$p$-basis\footnote{We refer to \cite[\S26]{Matsumura1989} for a reminder on the notion of a $p$-basis, including the fact that a collection of elements $\{b_i\}$ of $k$ forms part of a $p$-basis if and only if their differentials $\{db_i\}$ are linearly independent in $\Omega^1_k$.} of $k$ has infinite cardinality. Then the map $K(k[[t]])/p \to
\varprojlim_s K(k[t]/(t^s))/p$ is not an equivalence; more precisely, the map on $\pi_2$ is not surjective.
\end{theorem}
\begin{proof}
We begin with several straightforward reductions to Milnor $K$-theory. Firstly,
since $k[[t]]^\times$ is $p$-torsion-free, the canonical map $K_2(k[[t]])/p\to
\pi_2 (K(k[[t]])/p)$ is an isomorphism. The $p$-torsion in the
pro abelian group $\{k[t]/(t^s)^\times\}_s$ is also zero, since the transition
map $k[t]/(t^{ps})^\times\to k[t]/(t^s)^\times$ clearly kills all $p$-torsion in
the domain, and therefore $\{K_2(k[t]/(t^s))/p\}_s\to
\{\pi_2(K(k[t]/(t^s))/p) \}_s$ is an isomorphism of pro abelian groups. Secondly, $k[[t]]$ and
$k[t]/(t^s)$ are local rings with infinite residue field, whence a classical
result in algebraic $K$-theory \cite[\S8]{vanderKallen1977} states that  $K_2^M(k[[t]])\isoto K_2(k[[t]])$
and $K_2^M(k[t]/(t^s))\isoto K_2(k[t]/(t^s))$. Finally, the
canonical map $\pi_2 (\varprojlim_s K(k[t]/(t^s))/p)\to \varprojlim_s\pi_2(
K(k[t]/(t^s))/p)$ is surjective, by the Milnor sequence. In conclusion, to prove the theorem it is sufficient to show that the canonical map $K_2^M(k[[t]])/p\to\varprojlim_sK_2^M(k[t]/(t^s))/p$ is not surjective.

To do this, we will detect symbols using the $\dlog$ maps from $K_2^M$ to absolute K\"ahler differentials $\Omega^2_{-}:=\Omega^2_{-/\bb F_p}$:
\[\xymatrix{
\Omega^2_{k[[t]]}\ar[r]&\varprojlim_s\Omega^2_{k[t]/(t^s)}\\
K_2^M(k[[t]])/p\ar[u]_{\op{dlog}}\ar[r] &\varprojlim_s K_2^M(k[t]/(t^s))/p\ar[u]_{\op{dlog}}
}\]
It remains to construct an element in the bottom right of the diagram whose image in the top right does not come from the top left.

We now closely follow Bloch--Esnault--Kerz, with the necessary modifications to deal with the fact that we will eventually need to restrict to dlog forms. Given any map of rings $R\to S$ and differential form $\tau\in\Omega^2_{S/R}$, define its {\em weight} $w(\tau)$ to be the smallest integer $n$ for which it is possible to write $\tau=\sum_{i=1}^na_idb_i\wedge dc_i$ for some $a_i,b_i,c_i\in S$. 
If $\ell$ is a subfield of $k$ and $b_1,\dots,b_n,c_1,\dots,c_n\in k$ form part
of a $p$-basis for $k$ relative to $\ell$, then the elements $d b_1,\dots,d
b_n,d c_1,\dots,d c_n$ are linearly independent in the $k$-vector space
$\Omega^1_{k/\ell}$ and so this element $\sum_{i=1}^ndb_i\wedge dc_i$ has weight
$\ge n$ in $\Omega^2_{k/\ell}$ (Lemma~\ref{linalg} below).

Next consider the derivation
$$k[t]/(t^s) \rightarrow \Omega^1_k \otimes_k k[t]/(t^s)$$
which ``holds $t$ constant", so $\sum_n c_n t^n \mapsto \sum_n (dc_n)t^n$.  This extends by multiplicativity to a map
$$\Omega^\ast_{k[t]/(t^s)}\rightarrow \Omega^\ast_k \otimes_k k[t]/(t^s)$$
which splits the canonical map $\Omega^\ast_k\otimes_k k[t]/(t^s)\rightarrow
\Omega^\ast_{k[t]/(t^s)}$.   Restricting to degree $\ast=2$ and passing to the
limit as $s\to\infty$ defines \[e:\varprojlim_s\Omega^2_{k[t]/(t^s)}\To
\varprojlim_s\Omega^2_k\otimes_kk[t]/(t^s),\] where each element on the right
may be expressed as $\sum_{j\ge 0}\tau_i t^i$ for some unique
$\tau_0,\tau_1,\dots\in \Omega^2_k$ (the {\em $t$-adic coefficients} of the
element). Bloch--Esnault--Kerz \cite[Lem.~B2]{BlochEsnaultKerz2013} show
(assuming $k = \mathbb{C}$, but the argument works in general) that the image of the composition \[\Omega^2_{k[[t]]}\To \varprojlim_s\Omega^2_{k[t]/(t^s)}\To \varprojlim_s\Omega^2_k\otimes_kk[t]/(t^s)\] lands inside the set of those elements $\sum_{j\ge 0}\tau_j t^j$ whose $t$-adic coefficients satisfy the following: there exists $N\ge0$ such that $w(\tau_j)\le N\binom{j+2}j$ for all $j\ge0$.

We are now prepared to complete the proof by constructing a bad element of $\varprojlim_s K_2^M(k[t]/(t^s))/p$. First pick a sequence $0<w_1<w_2<\cdots$ of integers growing sufficiently fast such that no value of $N$ satisfies $w_j\le N\binom{2j}{2j-2}$ for all $j\ge1$. Then pick a sequence of subfields $\bb F_p=k_0\subset k_1\subset k_2\subset\cdots$ of $k$ such that any $p$-basis for $k_j$ relative to $k_{j-1}$ has $\ge2w_j$ elements; let $b_1^{(j)},\cdots,b_{w_j}^{(j)},c_1^{(j)},\cdots,c_{w_j}^{(j)}\in k_i$ be part of such a relative $p$-basis. Set \[f_s:=\sum_{j=1}^{s-1}\sum_{i=1}^{w_j}\{1+t^jb_i^{(j)},1+t^jc_i^{(j)}\}\in K_2^M(k[t]/(t^{s}))/p\] and note that the transition map $K_2^M(k[t]/(t^{s}))/p\to K_2^M(k[t]/(t^{s-1}))/p$ exactly kills the $j=s-1$ part of the sum and therefore sends $f_s$ to $f_{s-1}$; we therefore may define $f:=\varprojlim_s f_s\in\varprojlim_sK_2^M(k[t]/(t^{s}))/p$.

Let $\sum_{i\ge 1}\tau_it^i$ be the expansion of $e(\dlog
f)\in\varprojlim_s\Omega^2_k\otimes_kk[t]/(t^s)$. Noting that
\[\dlog(1+t^jb)\wedge\dlog(1+t^jc)\equiv t^{2j}da\wedge db\text{ mod
}dt,\,t^{3j}\] with all the higher order terms given by various expressions in
$b,c$, we see that \[\tau_{2s-2}\equiv\sum_{i=1}^{w_{s}}b_i^{(s)}\wedge
c_i^{(s)} \text{ mod }{\Omega^2_{k_{s-1}}}.\] By the second paragraph of the proof we deduce that the image of $\tau_{2s-2}$ in $\Omega^2_{k/k_{s-1}}$ has weight $\ge w_s$, whence a fortiori $w(\tau_{2s-2})\ge w_s$. By choice of the sequence $0< w_1<\cdots$, it follows that $\dlog f$ cannot be lifted to $\Omega^2_{k[[t]]}$, which completes the proof.
\end{proof}

Let $V$ be a vector space over a field $k$. 
Given a 2-form $\omega \in \bigwedge^2 V$, we define the \emph{weight}
of $\omega$ to be the minimal $n$ such that there exist elements
$\left\{x_i, y_i|1 \leq i \leq n\right\} \subset V$ such that
$\omega = \sum_{i=1}^n x_i \wedge y_i$. 
The following linear algebra lemma was used in the above proof.
\begin{lemma} 
\label{linalg} 
Let $\left\{u_1, \dots, u_n, v_1, \dots, v_n\right\}$ be linearly independent in
$V$. Then the form $\sum_{i=1}^n u_i \wedge v_i$ has weight $n$. 
\end{lemma} 
\begin{proof} 
Recall that we have an \emph{interior product} $V^{\vee} \otimes \bigwedge^2 V
\to V$ given  by $v \otimes (x \wedge y) \mapsto \left \langle v,
x\right\rangle y - \left \langle v, y\right\rangle x$. 
Given $\omega \in \bigwedge^2 V$, we obtain a map $f_\omega: V^{\vee} \to V$. 
If $\omega$ has weight $w$, then the rank of $f_{\omega}$ is at most $2w$. 
Completing $\left\{u_1,\dots, u_n, v_1, \dots, v_n\right\}$ to a basis of $V$
and forming the dual basis, 
one now sees that the rank of $f_{\omega_0}$ for $\omega_0 = \sum_{i=1}^n u_i
\wedge v_i$ is $2n$; together, this implies the lemma. 
\end{proof}

\subsection{$p$-adic continuity and $K$-theory}
We now specialize the continuity question to the case where the ideal is $(p)$.
In this case, we can often obtain stronger pro isomorphisms rather than simply isomorphisms on inverse limits and, secondly, hypotheses such as noetherianness and $F$-finiteness are no longer necessary. 

\begin{definition}
Let $R$ be a commutative ring. We say that $K$-theory is \emph{$p$-adically continuous} for $R$ if the map of spectra
$K(R) \to \varprojlim_i K(R/p^iR)$ is an equivalence modulo $p$.
\end{definition}

One has the following general result.

\begin{theorem}[Geisser--Hesselholt \cite{GHlocal}]
\label{GHcontinuity}
If $R$ is a local ring which is $p$-torsion-free and henselian along $(p)$, then $K$-theory is $p$-adically continuous at $R$. Moreover, the
map $K(R)/p \to \left\{ K(R/p^iR)/p\right\}_{i \geq 1}$ 
induces an isomorphism of pro abelian groups 
upon applying $\pi_j$ for any $j$. 
\end{theorem}

Although the previous theorem concerns $p$-adic rings, its proof uses Gabber rigidity for $\mathbb{Z}[1/p]$-algebras, via a clever trick due to Suslin.  Here we will obtain a generalization of Theorem \ref{GHcontinuity} using our version of rigidity, which directly applies to $p$-adic rings.  We observe also that the proof naturally yields something slightly
stronger than an equivalence of pro abelian groups. 

In this section, it will be convenient to work not only with commutative
rings, but also with more general ring spectra.  Hence, we will use the following variant. 

\begin{var}
The functors $K, \TC$ are defined not only for ordinary rings $R$, but more
generally for arbitrary $\mathbb{E}_1$-ring spectra $R$. For such $R$, we define
$\knf(R)$ similarly, as the fiber of the cyclotomic trace $K(R) \to \TC(R)$. 
\end{var}

Part of the theorem of Dundas--Goodwillie--McCarthy \cite{dgm} states that if $R$
is a connective $\mathbb{E}_1$-algebra, then the map $\knf(R) \to \knf(\pi_0 R)$
is an equivalence. 
In this sense, there is no extra generality afforded by the above variant. On
the other hand, we will find that it is often easier to control $\TC$ for appropriately
``derived'' constructions than underived constructions.

We next review some facts about nilpotent towers; compare \cite{Mthick}. 

\begin{definition} 
Let $\left\{A_i\right\}_{i \geq 1}$ be a tower of abelian groups. We say that
the tower is \emph{nilpotent} if there exists $N> 0$ such that all the maps
$A_{i+N} \to A_i$ are zero. We say that
the tower is \emph{quickly converging} if there exists $r \in \mathbb{Z}_{>0}$ such
that the tower $\{\mathrm{im}(A_{i+r} \to A_i)\}_{i\ge1}$ is eventually constant. 
This is stronger than the Mittag--Leffler condition. 
\label{quickconvab}
\end{definition}

By \cite[Lemma 3.10]{Mthick}, it follows that the collection of towers of
abelian groups which are quickly converging forms an abelian subcategory of the
category of towers which is closed under extensions. 
It thus follows that if $\left\{A_i\right\}$ is a quickly converging tower and
$A$ is the inverse limit, then the kernel and cokernel of the map of towers
$\left\{A\right\} \to \left\{A_i\right\}$ are both nilpotent. 
In particular, the category of quickly converging towers is the smallest abelian 
subcategory containing the nilpotent towers and the constant towers which is
closed under extensions. 

\begin{definition}
Let $\left\{X_i\right\}_{i \geq 1}$ be a tower in $\Sp$, i.e., $\{X_i\}_{i\ge1}\in\op{Tow}(\Sp):=\op{Fun}(\bb N^\sub{op},\Sp)$. 
We say that the tower is \emph{nilpotent} if there exists $N> 0$ such that all the
maps $X_{i+N} \to X_i$ are nullhomotopic. The collection of nilpotent towers
forms a thick subcategory of the stable $\infty$-category $\op{Tow}(\Sp)$. 
We say that a tower $\left\{X_i\right\}$ is \emph{quickly converging}
if the cofiber of the map of towers $\left\{X\right\}
\to \left\{X_i\right\}$ is nilpotent, where $X:=\varprojlim_iX_i$, or equivalently if $\left\{X_i\right\}$ lies in the smallest thick subcategory of towers containing the constant towers and the nilpotent towers.
\end{definition}

\begin{lemma}
Let $\left\{X_i\right\}_{i\ge1}$ be a tower in $\Sp$.
 If $\left\{X_i\right\}$ is quickly converging,  the
tower $\left\{\pi_j X_i\right\}$ of abelian groups is quickly converging for each $j$. 
The converse holds 
if we suppose that each $X_i$ has homotopy groups concentrated in the fixed range
$[a,b]$.
\end{lemma} 
\begin{proof} 
It follows from \cite[Lemma 3.11]{Mthick} that if $\left\{X_i\right\}$ is
quickly converging, then the towers of homotopy groups are 
quickly converging. 
Indeed, if $\left\{X_i\right\}$ is nilpotent or constant, then the tower of
homotopy groups is clearly quickly converging, and the referenced result shows
that the condition on homotopy groups is stable in cofiber sequences of towers. 
For the converse direction, a d\'evissage reduces to the
case where $a= b$. In this case, it follows from the paragraph following 
Definition~\ref{quickconvab}: that is, quickly convergent towers of abelian
groups are built up from constant and nilpotent towers. 
\end{proof} 

\begin{definition} 
We say that a tower $\left\{X_i\right\}$ of spectra  is \emph{almost nilpotent}
if for each $n \in \mathbb{Z}$, the truncated tower $\left\{\tau_{\leq n}
X_i\right\}$ is nilpotent, or equivalently if each tower $\left\{\pi_n
(X_i)\right\}$ of abelian groups is nilpotent. We say that a tower is \emph{almost quickly
converging} if for each $n$, the tower $\left\{\tau_{\leq n} X_i\right\}$ is
quickly converging. If the tower is uniformly bounded below, then it is almost
quickly converging if 
and only  if each tower $\left\{\pi_n
(X_i)\right\}$ of abelian groups is quickly convergent.  
\end{definition}

The following then is a direct consequence of \cite[Lemma 3.11]{Mthick}. 
\begin{lemma} 
The almost nilpotent and almost quickly converging towers 
form thick subcategories of $\mathrm{Tow}( \Sp)$ including the nilpotent and
quickly converging towers respectively.
\end{lemma} 

From this we get:

\begin{lemma} 
\label{simplicialalmostquickconv}
Let $X_{\bullet, \ast}$ be a simplicial object in the $\infty$-category of
towers of connective spectra. Suppose each tower $\dots \to X_{j,2} \to X_{j,
1}$ is almost quickly convergent, with limit $X_j$. 
Then the geometric realization $\dots \to |X_{\bullet, 2}| \to |X_{\bullet, 1}|$
is almost quickly converging and has limit $\vert X_\bullet \vert$ via the natural comparison map.
\end{lemma} 
\begin{proof} 
This follows from approximating the geometric realization with an $n$-skeletal
geometric realization in any range of degrees. 
\end{proof}

We now include a basic observation that given a tower of cyclotomic spectra
whose underlying tower of spectra is almost quickly converging, the tower of
$\TC/p$ is also almost quickly converging.  
\begin{lemma} 
\label{TCalmostquickconv}
Let $\left\{X_i\right\}$ be a tower in $\CycSp_{ \geq 0}$.
Suppose that the underlying tower of spectra $\left\{X_i/p\right\}$ is
almost quickly
converging. 
Then the tower of spectra $\left\{\TC(X_i)/p\right\}$ is almost quickly converging.
\end{lemma} 
\begin{proof}  
This follows in a similar fashion as Proposition~\ref{strong:finiteness}. See
also Remark \ref{TCinvlim} for the identification of the inverse limit of the
above tower in $\CycSp$.
\end{proof}

\begin{lemma} 
\label{transfernilp}
Let $\left\{X_i\right\}$ and $\left\{Y_i\right\}$ be two towers of connective spectra.  If $\left\{X_i\right\}$ and $\left\{Y_i\right\}$ are almost quickly converging with inverse limits $X$ and $Y$, then $\left\{X_i\otimes Y_i\right\}$ is almost quickly converging with inverse limit $X\otimes Y$.
\end{lemma} 
\begin{proof} 
Note that $\tau_{\leq m}( X_i \otimes Y_i) \simeq \tau_{\leq m} ( \tau_{\leq m}
X_i \otimes Y_i)$. So if $\left\{\tau_{\leq m} X_i\right\}$ is nilpotent, then so is 
$\left\{\tau_{\leq m}( X_i \otimes Y_i)\right\}$, and the limit is zero.  Symmetrically, we get the same if $\left\{\tau_{\leq m}Y_i\right\}$ is nilpotent.  On the other hand, if $\left\{\tau_{\leq m}X_i\right\}$ and $\left\{\tau_{\leq m} Y_i\right\}$ are both constant, we see that $\tau_{\leq m} (X_i\otimes Y_i)$ is constant, and the map from $X\otimes Y$ to the limit is an isomorphism degrees $\leq m-1$.  A thick subcategory argument lets us conclude.\end{proof}

\begin{theorem} 
\label{THHmodp}
Suppose that $\{R_i\}_{i\ge1}$ is a tower of connective $\mathbb{E}_1$-ring
spectra, and $R$ is another connective $\mathbb{E}_1$-ring spectrum with a
comparison map $R\rightarrow \varprojlim_i R_i$.  If this comparison map is an equivalence modulo $p$ and the tower of spectra $\{R_i/p\}$ is almost quickly converging, then the comparison map
$$\TC(R)\rightarrow\varprojlim_i \TC(R_i)$$
is an equivalence modulo $p$ and the tower $\{\TC(R_i)/p\}$ is almost quickly converging.
\end{theorem} 
\begin{proof} 
Note that a tower is almost quickly converging modulo $p$ if and only if it is almost quickly converging after smashing with $X$, for any choice of spectrum $X$ generating the same thick subcategory as $S^0/p$.  For example one can take $X=S^0/p\otimes S^0/p$.\footnote{This is the cone of multiplcation by $p$ on $S^0/p$, hence lies in the thick subcategory generated by $S^0/p$.  For the converse, note that $p^2$ kills $S^0/p$, so $S^0/p$ is a retract of $(S^0/p \otimes S^0/p)/p$.}  Keeping this in mind, an inductive application of Lemma \ref{transfernilp} shows that $\left\{(HR_i)^{\otimes n}\right\}_i$ is almost quickly converging modulo $p$ for any $n\geq 0$, and has mod~$p$ limit $(HR)^{\otimes n}/p$.  Then by Lemma \ref{simplicialalmostquickconv} we deduce that $\left\{\THH(R_i)/p\right\}$ is almost quickly converging with limit $\THH(R)/p$.  From this, Lemma \ref{TCalmostquickconv} lets us conclude.
\end{proof} 
In the $\mathbb{Z}$-linear case, this immediately implies the following
general $p$-adic continuity result for $\TC$. We denote by $\mod_{H\mathbb{Z}}$ the symmetric monoidal $\infty$-category of $H\mathbb{Z}$-module spectra, or equivalently the derived $\infty$-category of $\mathbb{Z}$. 

\begin{theorem}
\label{TCderpcontinuous}
Let $R$ be a connective
$\mathbb{E}_1$-algebra in $\mod_{H\mathbb{Z}}$.
Then the map 
\[   \TC(R) \to \varprojlim_i \TC( R \otimes_{H\mathbb{Z}} H\mathbb{Z}/p^i) \]
is a $p$-adic equivalence. 
Moreover the tower on the right-hand-side is almost quickly
converging modulo $p$.
\end{theorem}
\begin{proof}
It suffices to remark that the transition maps in the tower of fibers
$\{\operatorname{fib}(R\rightarrow R\otimes_{H\mathbb{Z}} H\mathbb{Z}/p^i\mathbb{Z})\}_n$ identify with multiplication by $p$ maps, hence are zero modulo $p$, so that Proposition \ref{THHmodp} applies.
\end{proof}

We now obtain a general result on derived $p$-adic continuity of $K$-theory,
and as a corollary a generalization of Theorem~\ref{GHcontinuity}. 

\begin{theorem}
\label{derivedpadiccont}
Let $R$ be a connective $\mathbb{E}_1$-algebra in $\md_{H\mathbb{Z}}$ such that $\pi_0(R)$
is commutative and henselian along $(p)$. Then the tower $\{K(R  \otimes_{H\mathbb{Z}}
H\mathbb{Z}/p^i\mathbb{Z}) / p\}_{i \geq 1}$ is almost
quickly converging, and
$$K(R)\rightarrow \varprojlim_i K(R\otimes_{H\mathbb{Z}}H\mathbb{Z}/p^i\mathbb{Z})$$
is an equivalence modulo $p$.
\end{theorem}

\begin{proof} 
By Theorem~\ref{mainthm} we see $\knf(R) = \knf(R \otimes_{H\mathbb{Z}}
H\mathbb{Z}/p^i \mathbb{Z})$ for
each $i$. This reduces us to the analogous result for $\TC$, which is Theorem~\ref{TCderpcontinuous}. 
\end{proof} 

There is also an underived version under a very mild hypothesis, which
generalizes Geisser--Hesselholt's Theorem~\ref{GHcontinuity}.
\begin{theorem} 
Let $R$ be a commutative ring which 
is henselian along $(p)$, and suppose the $p$-power-torsion in $R$ is bounded. Then 
$K$-theory is $p$-adically continuous at $R$. Moreover, the tower $\{K(R/p^iR)/p\}$ is almost quickly converging.
\end{theorem}
\begin{proof} 
The map of towers $\{HR\otimes_\mathbb{Z}\mathbb{Z}/p^i\mathbb{Z}\}\rightarrow
\{H(R/p^iR)\}$ has fiber $\{\Sigma H(R[p^i])\}$ with transition maps multiplication by
$p$.  Because of the bounded $p$-power torsion hypothesis, this tower
$\{\Sigma H(R[p^i])\}$ is almost nilpotent modulo $p$.  So we deduce $p$-adic
continuity for $\TC$ as in Proposition \ref{TCderpcontinuous}, and then $p$-adic
continuity for $K$-theory as in Proposition \ref{derivedpadiccont}.
\end{proof}

\begin{remark}
Without assuming $\mathbb{Z}$-linear structure, one can get similar results by
replacing $-\otimes_\mathbb{Z}\mathbb{Z}/p^i\mathbb{Z}$ with $-\otimes S_i$ for
any reasonable tower of $E_\infty$-algebras $\left\{S_i\right\}$ of which $p$-adically approximates the sphere spectrum, for example the tower coming from the usual cosimplicial Adams resolution associated to $S^0\rightarrow H\mathbb{F}_p$.
\end{remark}

\subsection{Pro Geisser--Levine theorems}
We begin by recalling from subsection \ref{subsec:equalchar} that one
classically associates to any $\bb F_p$-algebra $R$ the abelian group
$\nu^n(R)=\Omega^n_{R,\sub{log}}$, defined either as $\ker( 1 - C^{-1}:
\Omega^n_R \to \Omega^n_R / d \Omega^{n-1}_R)$ or as the subgroup of
$\Omega^n_R$ which is generated \'etale locally by dlog forms. As we recalled in Theorem~\ref{GHGL}, when $R$ is moreover local and ind-smooth then Geisser and Levine established isomorphisms $K_n(R;\bb Z/p\bb Z)\cong\Omega^n_{R,\sub{log}}$ for all $n\ge0$. The goal of this section is to establish an entirely analogous description of the pro abelian groups $\{K_n(R/I^s;\bb Z/p\bb Z)\}_s$ whenever $I\subseteq R$ is an ideal.

It is convenient to work not only modulo $p$ but more generally modulo $p^r$ with $r>1$. This necessitates introducing some standard notation surrounding de Rham--Witt groups, in particular the logarithmic subgroup which serves as a mod $p^r$ lift of $\Omega^n_{R,\sub{log}}$.

\begin{definition}[Logarithmic de Rham--Witt groups]\label{definition_dRW}
Let $R$ be an $\bb F_p$-algebra. 
We recall the de Rham-Witt complex $\{W_r \Omega_R^{\ast}\}$ as in
\cite{illusie-derham-witt} (see also 
Definition~\ref{dRW1}). 
Letting $[\cdot]:R\to W_r(R)$ denote the Teichm\"uller lift, there is a resulting group homomorphism $\dlog[\cdot]:R^\times\to W_r\Omega^1_R$, $\al\mapsto\dlog[\al]=\tfrac{d[\al]}{[\al]}$, and more generally \[\dlog[\cdot]:R^{\times\otimes n}\to W_r\Omega_{R}^n,\quad \al_1\otimes\cdots\otimes\al_n\mapsto\dlog[\al_1]\wedge\cdots\wedge\dlog[\al_n].\] When $R$ is local, the image of this map will be denoted by $W_r\Omega^n_{R,\sub{log}}$.
\end{definition}

\begin{definition}[Globalization]
Given an $\bb F_p$-scheme $X$, one defines $W_r\Omega^n_X$ to be the Zariski (or \'etale, depending on the context) sheaf obtained by sheafifying $U\mapsto W_r\Omega^n_{\mathcal O_X(U)}$. When $X=\Spec R$ is affine, this sheaf (in either topology) has no higher cohomology and has global sections $W_r\Omega_R^n$: this follows from flat descent and the facts that if $R\to S$ is \'etale then so is $W_r(R)\to W_r(S)$ \cite[Thm.~2.4]{vanderKallen1986} and moreover the canonical map $W_r\Omega_R^n\otimes_{W_r(R)}W_r(S)\to W_r\Omega_S^n$ is an isomorphism \cite[Prop.~1.7]{LangerZink2004}.

We define $W_r\Omega^n_{R,\sub{log}}$ to be the subgroup of $W_r\Omega^n_R$
consisting of elements which can be written Zariski locally as sums of dlog
forms, i.e., $H^0(\Spec R,-)$ of the image (as a Zariski sheaf) of
\[\dlog[\cdot]:\bb G_{m,\Spec R}^{\otimes n}\to W_r\Omega_{\Spec R}^n,\quad
\al_1\otimes\cdots\otimes\al_n\mapsto\dlog[\al_1]\wedge\cdots\wedge\dlog[\al_n].\]
We note that $W_r\Omega^n_{R,\sub{log}}$ would be unchanged if we were to
replace Zariski by \'etale in the previous sentence, by \cite[Cor. 
4.1(iii)]{Morrow-HW}; in particular, $W_1\Omega^n_{R,\sub{log}}$ really is the same as $\Omega^n_{R,\sub{log}}$ as defined at the start of the subsection. Similarly, if $X$ is a scheme, we let
$W_r \Omega_{X, \sub{log}}^n$ denote the sheaf $\spec R
\mapsto W_r \Omega^n_{R, \sub{log}}$. 
\end{definition}

In the case that $R$ is an ind-smooth local $\bb F_p$-algebra, Geisser--Levine
\cite{GL} established isomorphisms $K_n(R;\bb Z/p^r\bb Z)\cong W_r\Omega^n_{R,\sub{log}}$ (given by the map $\dlog[\cdot]$ on symbols) for all $n\ge0$, $r\ge1$, and proved that each group $K_n(R)$ is $p$-torsion-free. The primary goal of this subsection is to establish the following pro version of this theorem:

\begin{theorem}\label{theorem_pro_GL}
Let $R$ be a $F$-finite, regular, local noetherian $\bb F_p$-algebra and $I\subseteq R$ an ideal; fix $n,r\ge 0$. Then the pro abelian group $\{K_n(R/I^s)\}_s $ is $p$-torsion-free and there is a natural isomorphism of pro abelian groups \[\{K_n(R/I^s;\bb Z/p^r\bb Z)\}_s\isoto \{W_r\Omega^n_{R/I^s,\sub{log}}\}_s\] given by $\dlog[\cdot]$ on symbols.
\end{theorem}

This theorem was proved in \cite{Morrow-HW} under the assumption that $\Spec R/I$ was sufficiently regular (the terminology used there was ``generalised normal crossing''), and some applications were given to higher dimensional class field theory and to the deformation theory of algebraic cycles. Here we combine the results of \cite{Morrow-HW} with our main rigidity theorem to prove the theorem in general as well as a similar result in a non-local, relative case (Theorem \ref{theorem_pro_GL_relative}).

In fact, 
Theorem \ref{theorem_pro_GL_relative} below (which works for an arbitrary
henselian ideal in a regular $F$-finite $\mathbb{F}_p$-algebra) is more
fundamental. 
We can view this statement as a pro-version of the following calculation. 

\begin{proposition} 
Let $(R, I)$ be a henselian pair of $\mathbb{F}_p$-algebras. Suppose that 
$R$ and $R/I$ are ind-smooth. 
Then we have a natural isomorphism \begin{equation} \label{relhensmodp} K_n(R,
I; \mathbb{Z}/p) \simeq \Omega^n_{(R, I), \sub{log}},.\end{equation}
where $\Omega^n_{(R, I), \sub{log}}:=\op{ker}(\Omega^n_{R,\sub{log}}\rightarrow \Omega^n_{R/I,\sub{log}})$.
\end{proposition}
\begin{proof}
Indeed, this follows from our main result (which identifies the relative
$K$-theory with relative $\TC$), 
together with the Geisser--Hesselholt calculations of relative $\TC$. 
We have
the short exact sequence
\[  0 \to \widetilde{\nu}^{n+1}(R) \to \pi_n(\TC(R)/p)  \to \nu^n(R) \to 0 
\]
as in Theorem~\ref{GHGL} (see also \eqref{thisdiag}).
Using Proposition~\ref{toyFprigidityTC}, the result follows. 
\end{proof}

We begin with a couple of lemmas concerning relative log de Rham--Witt groups.

\begin{definition}
Given an $\bb F_p$-algebra $R$ and ideal $I\subseteq R$, we set $W_r\Omega^n_{(R,I),\sub{log}}:=\ker(W_r\Omega^n_{R,\sub{log}}\to W_r\Omega^n_{R/I,\sub{log}})$.\end{definition}

\begin{lemma}\label{lemma_drw1}
Let $R$ be an $\bb F_p$-algebra which is henselian along an ideal $I\subseteq R$. Then
$R-F:W_r\Omega^n_{(R,I)}\to W_{r-1}\Omega^n_{(R,I)}$ is surjective.
\end{lemma}
\begin{proof}

If $R$ is $F$-finite, noetherian, and $I$-adically complete then this was proved
in \cite[Prop.~2.20]{Morrow-HW}, and we will reduce the general case to this
one. Firstly, by taking a filtered colimit we may suppose that $R$ is the
henselization of a finite type $\bb F_p$-algebra along some ideal, and in
particular that $R$ is $F$-finite and excellent \cite[Tag 07QS]{stacks-project}. Then the $I$-adic completion $\hat R$ is geometrically regular over $R$ by \cite[7.8.3(v)]{EGA_IV_2}, whence it is a filtered colimit of smooth $R$-algebras $A$ by N\'eron--Popescu. But each structure map $R\to A$ admits a splitting
as in Remark~\ref{henspairequivtocontinuity}. 
Therefore, taking a filtered colimit shows that \[\op{Coker}(W_r\Omega^n_{(R,I)}\xto{R-F} W_{r-1}\Omega^n_{(R,I)})\To \op{Coker}(W_r\Omega^n_{(\hat R,I\hat R)}\xto{R-F} W_{r-1}\Omega^n_{(\hat R,I\hat R)})\] is injective. Since the right side vanishes by the first sentence of the paragraph, the proof is complete.
\end{proof}

A important result of Illusie states that if $X$ is a smooth variety over a
perfect field of characteristic $p$, then the canonical map of \'etale sheaves
$W_{s}\Omega^n_{X,\sub{log}}/p^r\to W_r\Omega^n_{X,\sub{log}}$ is an isomorphism
whenever $s\ge r$ and $n\ge0$, and therefore \[0\To
\{W_s\Omega^n_{X,\sub{log}}\}_s\xto{p^r}
\{W_s\Omega^n_{X,\sub{log}}\}_s\xto{\{R^{s-r}\}_s}W_r\Omega^n_{X,\sub{log}}\To
0\] is an exact sequence of pro sheaves on $X_\sub{\'et}$ \cite[\S
I.5.7]{illusie-derham-witt}. We will require the following relative form of this result for henselian ideals:

\begin{lemma}\label{lemma_drw2}
Let $R$ be an $F$-finite, regular, noetherian $\bb F_p$-algebra which is henselian along an ideal $I\subseteq R$; fix $n\ge0$. Then the sequence of pro abelian groups
\[0\To\{W_s\Omega^n_{(R,I^s),\sub{log}}\}_s\xto{p^r}\{W_s\Omega^n_{(R,I^s),\sub{log}}\}_s\xto{\{R^{s-r}\}_s}\{W_r\Omega^n_{(R,I^s),\sub{log}}\}_s\To 0\]
is exact for any $r\ge1$.
\end{lemma}
\begin{proof}
We will need to argue via the \'etale topology, so we begin by introducing the necessary sheaves. Let $X:=\Spec R$, with closed subschemes $Y_s:=\Spec R/I^s$ for $s\ge1$; let $W_r\Omega^n_X$, $W_r\Omega^n_{Y_s}$, and $W_r\Omega^n_{(X,Y_s)}:=\ker(W_r\Omega^n_X\to W_r\Omega^n_{Y_s})$ be the corresponding \'etale sheaves on $X$. As explained in Definition \ref{definition_dRW}, these sheaves have no higher \'etale cohomology (since $X$ is affine) and their global sections are respectively $W_r\Omega^n_R$, $W_r\Omega^n_{R/I^s}$, and $W_r\Omega^n_{(R,I^s)}$.

Next let $W_r\Omega^n_{X,\sub{log}}$ be the image in the \'etale topology of $\dlog[\cdot]:\bb G_{m,X}^{\otimes n}\to W_r\Omega_{X}^n$, and similarly for each $Y_s$, and set $W_r\Omega^n_{(X,Y_s),\sub{log}}:=\ker(W_r\Omega^n_{X,\sub{log}}\to W_r\Omega^n_{Y_s,\sub{log}})$. As we already mentioned in Definition \ref{definition_dRW}, the subgroup of dlog forms may be defined using either the Zariski or \'etale topology by \cite[Cor.~4.1(iii)]{Morrow-HW}, and therefore the global sections of these sheaves are respectively $W_r\Omega^n_{R,\sub{log}}$, $W_r\Omega^n_{R/I^s,\sub{log}}$, and $W_r\Omega^n_{(R,I^s),\sub{log}}$.

We claim that $\{H^1_\sub{\'et}(X,W_r\Omega^n_{(X,Y_s),\sub{log}})\}_r=0$ for
each $s\ge 1$. After replacing $I$ by $I^s$ (since a subideal of a henselian
ideal remains henselian, Remark~\ref{subideal:henspair}) we may as well assume $s=1$ and write $Y=Y_s$ for simplicity of notation; the key will be that $X$ is an affine scheme, henselian along $Y$. Appealing to \cite[Cor.~4.1(iii)]{Morrow-HW} for both $X$ and $Y$ gives rise to short exact sequences of pro sheaves on $X_\sub{\'et}$, in which the three vertical arrows are surjective:
\[\xymatrix{
0\ar[r]& \{W_r\Omega^n_{X,\sub{log}}\}_r\ar[r]\ar[d]&\{W_r\Omega^n_X\}_r\ar[r]^{R-F}\ar[d]&\{W_{r-1}\Omega^n_X\}_r\ar[r]\ar[d]&0\\
0\ar[r]& \{W_r\Omega^n_{Y,\sub{log}}\}_r\ar[r]&\{W_r\Omega^n_Y\}_r\ar[r]^{R-F}&\{W_{r-1}\Omega^n_Y\}_r\ar[r]&0
}\]
Taking the kernels of the vertical arrows gives us a short exact sequence of relative terms
\begin{equation}0\To \{W_r\Omega^n_{(X,Y),\sub{log}}\}_r\To\{W_r\Omega^n_{(X,Y)}\}_r\xto{R-F}\{W_{r-1}\Omega^n_{(X,Y)}\}_r\To0\label{equation_R-F}\end{equation} Since the middle and right terms have no higher cohomology, it is enough to check that $R-F$ is surjective on global sections; but this is even true for each fixed level $r$ by Lemma \ref{lemma_drw1}.

Finally, appealing to Illusie's result recalled immediately before the lemma (which has been extended to arbitrary regular $\bb F_p$-schemes by A.~Shiho \cite[Cor.~2.13]{Shiho2007}) and to an analogous pro version for the formal completion of $X$ along along $Y_1$ \cite[Cor.~4.8]{Morrow-HW} gives rise to short exact sequences of pro sheaves on $X_\sub{\'et}$, in which the three vertical arrows are again surjective:
\[\xymatrix{
0\ar[r]& \{W_s\Omega^n_{X,\sub{log}}\}_s\ar[r]^{p^r}\ar[d]&\{W_s\Omega^n_{X,\sub{log}}\}_s\ar[r]^{\{R^{r-s}\}_s}\ar[d]&W_{r}\Omega^n_{X,\sub{log}}\ar[r]\ar[d]&0\\
0\ar[r]& \{W_s\Omega^n_{Y_s,\sub{log}}\}_s\ar[r]^{p^r}&\{W_s\Omega^n_{Y_s,\sub{log}}\}_s\ar[r]^{\{R^{r-s}\}_s}&\{W_{r}\Omega^n_{Y_s,\sub{log}}\}_s\ar[r]&0
}\]
Taking kernels gives a short exact sequence of relative terms
\[0\To \{W_s\Omega^n_{(X,Y_s),\sub{log}}\}_s\xto{p^r}\{W_s\Omega^n_{(X,Y_s),\sub{log}}\}_s\xto{\{R^{r-s}\}_s}\{W_{r}\Omega^n_{(X,Y_s),\sub{log}}\}_s\To0.\] Taking global sections completes the proof since the previous paragraph showed that the left term has no $H^1_\sub{\'et}$.
\end{proof}

The following is our relative form of Theorem \ref{theorem_pro_GL} for henselian
ideals (and a pro-version of the identification \eqref{relhensmodp}), from which Theorem \ref{theorem_pro_GL} will immediately follow.
 
\begin{theorem}\label{theorem_pro_GL_relative}
Let $R$ be a $F$-finite, regular, noetherian $\bb F_p$-algebra which is henselian along an ideal $I\subseteq R$; fix $n,r\ge 0$. Then the trace map induces a natural isomorphism of pro abelian groups \[\{K_n(R,I^s;\bb Z/p^r\bb Z)\}_s\isoto \{W_r\Omega^n_{(R,I^s),\sub{log}}\}_s.\]
\end{theorem}
\begin{proof}
The proof is similar to the techniques of \cite[\S5.1]{Morrow-HW}, but we repeat
the necessary details here. In particular, we begin with the same recollections
on Hochschild--Kostant--Rosenberg theorems for the spectra $\TR^s$. If $R$ is
any $\bb F_p$-algebra, then the pro graded ring $\{\TR^s_*(R;p)\}_s$ is a
$p$-typical Witt complex with respect to its operators $F,V,R$; by universality
of the de Rham--Witt complex, there are therefore natural maps of graded $W_s(R)$-algebras \cite[Prop.~1.5.8]{Hesselholt} $\lambda_{s,R}:W_s\Omega_R^*\to \TR_*^s(R;p)$ for $s\ge0$, which are compatible with the Frobenius,  Verschiebung, and Restriction maps (in other words, a morphism of $p$-typical Witt complexes).

From now on in the proof assume that $R$ is regular, noetherian, and $F$-finite, and let $I\subseteq R$ be any ideal. Hesselholt's HKR theorem \cite[Thm.~B]{Hesselholt} implies that the resulting map of pro abelian groups \[\lambda_R:\{W_s\Omega_R^n\}_s\To \{\TR_n^s(R;p)\}_s\] is an isomorphism for each $n\ge 1$; similarly, the pro HKR theorem of Dundas--Morrow \cite[Cor.~4.15]{DundasMorrow} implies that \[\lambda_{R/I^\infty}: \{W_s\Omega_{R/I^s}^n\}_s\To \{\TR_n^s(R/I^s;p)\}_s \tag{pro-HKR}\] is an isomorphism of pro abelian groups for each $n\ge0$. Since $W_s\Omega^n_R\to W_s\Omega^n_{R/I^s}$ is surjective for all $s,n\ge0$, it follows that the long exact sequence associated to $\{\TR^s(R,I^s;p)\}_s\to \{\TR^s(R;p)\}_s\to \{\TR^s(R/I^s;p)\}_s$ breaks into short exact sequences and there are therefore natural induced isomorphisms of relative theories
\[\{W_s\Omega_{(R,I^s)}^n\}_s\isoto \{\TR_n^s(R,I^s;p)\}_s\] for all $n\ge0$.

Now assume that $R$ is henselian along $I$. By Lemma \ref{lemma_drw1}, the map
$R-F:W_{s}\Omega^n_{(R,I^s)}\to W_{s-1}\Omega^n_{(R,I^s)}$ is surjective for all
$n,s\ge0$, and therefore the previous isomorphism shows that the long exact
sequence associated to $\{\TC^s(R,I^s;p)\}_s\to
\{\TR^s(R,I^s;p)\}_s\xto{R-F}\{\TR^s(R,I^s;p)\}_s$ breaks into short exacts and
thereby induces natural isomorphisms \[\{\TC^s_n(R,I^s;p)\}_s\isoto
\{\ker(W_{s}\Omega^n_{(R,I^s)}\xto{R-F} W_{s-1}\Omega^n_{(R,I^s)})\}_s\] The
right side of the previous line is precisely
$\{W_s\Omega^n_{(R,I^s),\sub{log}}\}_s$ by applying \cite[Cor. 
4.1(iii)]{Morrow-HW} to $R$ and $R/I^s$ for each $s\ge1$ (in fact, we have
already made this argument: just take global sections in line
(\ref{equation_R-F}) for each $Y_s$ and then pass to the resulting diagonal of
the $\bb N^2$-indexed pro abelian group); moreover, this latter pro abelian
group is $p$-torsion-free by the injectivity in Lemma \ref{lemma_drw2} (this is
anyway easy: the pro abelian group is contained in $\{W_s\Omega^n_R\}_s$, which
is $p$-torsion-free thanks to the equality of the $p$- and canonical-filtrations
on the de Rham--Witt groups of any regular $\bb F_p$-algebra: see
\cite[Prop.~I.3.2 \& I.3.4]{illusie-derham-witt} for the case of a smooth algebra over a perfect field, then apply N\'eron--Popescu), so passing to finite coefficients gives us isomorphisms \[\{\TC^s_n(R,I^s;\bb Z/p^r\bb Z)\}_s\cong\{W_s\Omega^n_{(R,I^s),\sub{log}}/p^r\}_s\] for each fixed $r\ge1$. The right side is $\{W_r\Omega^n_{(R,I^s),\sub{log}}\}_s$ by (the hard part of) Lemma \ref{lemma_drw2}.

Finally, pro-constancy of $\{\TC^s(-;\bb Z/p^r)\}_s$ (Proposition \ref{prop_pro_constant}) implies that the left side of the previous line is $\{\TC_n(R,I^s;\bb Z/p^r\bb Z)\}_s$, which by Theorem \ref{mainthm} identifies with $\{K_n(R,I^s;\bb Z/p^r\bb Z)\}_s$ via the trace map.
\end{proof}

\begin{corollary}
Theorem \ref{theorem_pro_GL} is true.
\end{corollary}
\begin{proof}
Let $R$ be as in the statement of Theorem \ref{theorem_pro_GL}. Without loss of generality $R$ is henselian along $I$, whence $R$ is also local, and so the result follows by combining Theorem \ref{theorem_pro_GL_relative} with the isomorphism $K_n(R; \bb Z/p^r\bb Z)\isoto W_r\Omega^n_{R,\sub{log}}$ of Geisser--Levine \cite{GL}. Note that the two results are compatible since the trace map is given by $\dlog[\cdot]$ on symbols by \cite[Lem.~4.2.3 \& Cor.~6.4.1]{GH}.

We establish the $p$-torsion-freeness of $\{K_n(R/I^s)\}_s$ in the next corollary.
\end{proof}

In the following corollary $K_n^M$ denotes Milnor $K$-theory (either the
classical version or Kerz's improved variant \cite{Kerz2010}; the corollary holds for both):

\begin{corollary}[Comparison to Milnor $K$-theory; $p$-torsion-freeness]\label{corollary_Milnor}
Let $R$ be an $F$-finite, regular, noetherian $\bb F_p$-algebra and $I\subseteq R$ an ideal such that $R/I$ is local. Then, for all $n,r\ge0$:
\begin{enumerate}
\item the canonical map $\{K_n^M(R/I^s)/p^r\}_s\to\{K_n(R/I^s;\bb Z/p^r)\}_s$ is surjective and has the same kernel as $\dlog[\cdot]:\{K_n^M(R/I^s)/p^r\}_s\to \{W_r\Omega^n_{R/I^s}\}_s$;
\item the pro abelian group $\{K_n(R/I^s)\}_s$ is $p$-torsion-free.
\end{enumerate}
\end{corollary}
\begin{proof}
We can replace $R$ by its $I$-adic completion and thereby assume $R$ itself is local.  Part (1) is an immediate consequence of the commutative diagram
\[\xymatrix{
\{K_n^M(R/I^s)/p^r\}_s\ar[r]\ar@/_1cm/[rr]_{\dlog[\cdot]}& \{K_n(R/I^s;\bb Z/p^r\bb Z)\}_s\ar[r]^{\cong}&\{W_r\Omega_{R/I^s,\sub{log}}^n\}_s
}\]
which summarises the statement of Theorem \ref{theorem_pro_GL}, since the bendy arrow is surjective.

Part (1) clearly implies that $\{K_n(R/I^s)\}_s\to \{K_n(R/I^s;\bb Z/p\bb Z)\}_s$
is surjective, whence the usual exact sequence implies it is an isomorphism and that $\{K_{n-1}(R/I^s)[p]\}_s=0$.
\end{proof}

\begin{corollary}[Lifting classes]\label{corollary_lifting}
Let $R$ be an $F$-finite, regular, noetherian, local  $\bb F_p$-algebra and $I\subseteq R$ an ideal. Then, for all $n,r\ge0$:
\begin{enumerate}
\item the sequences \[0\To \{K_n(R,I^s;\bb Z/p^r\bb Z)\}_s\To K_n(R;\bb Z/p^r\bb Z)\To\{K_n(R/I^s;\bb Z/p^r\bb Z)\}_s\To 0\] are short exact;
\item there exists $s\ge 1$ with the following property: if an element of
$K_n(R/I;\bb Z/p^r\bb Z)$ lifts to $K_n(R/I^s;\bb Z/p^r\bb Z)$, then it is
symbolic and hence lifts to $K_n(R)$. 
\end{enumerate}
\end{corollary}
\begin{proof}
As we have just seen in Corollary \ref{corollary_Milnor}, $\{K_n(R/I^s;\bb Z/p^r\bb Z)\}_s$ is entirely symbolic; since $R\to R/I$ is surjective on units (using the assumption $I\subseteq\op{Jac}(R)$), we deduce that $K_n(R;\bb Z/p^r\bb Z)\to \{K_n(R/I^s;p^r\bb Z/\bb Z)\}_s$ is surjective for all $n\ge0$. Therefore the long exact relative sequences breaks into the desired short exact sequences.

Part (2) is a consequence of the surjectivity arguments of the previous paragraph, unravelling what it means for a map of pro abelian groups to be surjective.
\end{proof}

\begin{corollary}[Relative pro Geisser--Levine]
Let $R$ be an $F$-finite, regular, noetherian, local  $\bb F_p$-algebra and $I\subseteq R$ an ideal. Then, for all $n,r\ge0$, the trace map induces a natural isomorphism of pro abelian groups $\{K_n(R,I^s;\bb Z/p^r\bb Z)\}_s\cong\{W_r\Omega^n_{(R,I^s),\sub{log}}\}_s$
\end{corollary}
\begin{proof}
This follows from the short exact sequence of Corollary \ref{corollary_lifting}(1) by applying usual Geisser--Levine to the middle term and Theorem \ref{theorem_pro_GL} to the right term.
\end{proof}

\section{Comparisons of $K$ and $\TC$}

In this section, we prove two main general results: the comparison of $K$-theory
and $\TC$ in large degrees under mild finiteness hypotheses (Theorem~\ref{KTCasymptotic}) and
the \'etale local comparison of $K$ and $\TC$ (Theorem~\ref{etalekthm}). We give various
examples in subsection~\ref{KTCexamples}. Finally, we  prove an injectivity result for the
cyclotomic trace on local $\mathbb{F}_p$-algebras
(Theorem~\ref{injectivityresult}).

\subsection{\'Etale $K$-theory is $\TC$}

The results of \cite{GH} show that if $R$ is a smooth algebra over a perfect
field of characteristic $p$, then the $p$-adic \'etale $K$-theory of $R$
agrees with $\TC(R)$. We extend this result to all commutative rings which are henselian along $(p)$.  
Recall that a 
local ring is called \emph{strictly henselian} if it is henselian local and its
residue field is separably closed. 

\begin{theorem} 
\label{etalekthm}
Let $R$ be a strictly henselian local ring of residue characteristic $p$. Then $\knf(R)/p =0$, i.e., the map $K(R) \to \TC(R)$ is a
$p$-adic equivalence. 
\end{theorem} 
\begin{proof} By
Theorem~\ref{mainthm}, we may assume that $R=k$ is a field itself, which is
then separably closed. 
Thus we need to show that if $k$ is a separably closed field of characteristic
$p > 0$, then the map $K(k) \to \TC(k)$ is a $p$-adic equivalence. 

As $\mathbb{F}_p$ is perfect, $k$ is an ind-smooth $\mathbb{F}_p$-algebra (e.g., choose
a transcendence basis for $k$).  Thus it suffices to show that the terms
$\widetilde{\nu}^n(k)$ vanish by Theorem~\ref{GHGL}. 
That is, we need to show that the map
\[ 1  - C^{-1}: \Omega^n_k \to \Omega^n_k/ d \Omega^{n-1}_k  \]
is surjective. 
Given a form $\omega = a dx_1 \dots dx_n \in \Omega^n_k$, we have
$(1 - C^{-1})(u \omega) = (u - u^p a^{p-1} x_1^{p-1} \dots x_n^{p-1}) \omega$. 
Since $k$ is separably closed, 
we can solve the equation $
u - u^p a^{p-1} x_1^{p-1} \dots x_n^{p-1} = 1$ in $k$. This implies that $1 -
C^{-1}$ is surjective as desired and completes the proof. 
\end{proof} 

In the proof above, rather than using Theorem \ref{GHGL} and the groups
$\widetilde{\nu}^n(k)$, one can instead follow Suslin's arguments from
\cite{Suslinalgclosed} to show invariance of $\knf(-)/p$ under extensions of
separably closed fields.  This reduces to the case $k=\overline{\mathbb{F}_p}$,
which is easy since $K/p$ and $\TC/p$ can be directly calculated, cf.~\cite{HM97, Qui} or Example \ref{perfex}.
We record this alternative argument in the following proposition. 

\begin{proposition} 
Let $F: \mathrm{Ring} \to \mathrm{Ab}$ be a functor from rings to abelian
groups. Suppose that $F$ commutes with filtered colimits and satisfies rigidity,
i.e., for a henselian pair $(R, I)$ we have $F(R) \simeq F(R/I)$. 
Then for an extension $K \to L$ of separably closed fields, we have $F(K) \simeq
F(L)$. 
\end{proposition} 
\begin{proof}
Let $k$ be a separably closed field. 
The crux of Suslin's argument is to show that if
$q:X\rightarrow\Spec(k)$ is a connected smooth affine $k$-scheme of finite type, then
  for any class $\alpha \in F(X)$ and section $x: \Spec(k) \to X$, the
 pullback $x^* \alpha \in F(k)$ is independent of the choice of $x$. 
It suffices to fix a section $x_0: \Spec(k) \to X$,  
 and show (by connectedness, in view of the Zariski density of $k$-points in
 $X$, cf., e.g.,  \cite[Tag 04QM]{stacks-project}) that there exists a Zariski neighborhood $U$ of $x_0$ such that
 $x_0^\ast\alpha = x^\ast\alpha$ for all sections $x:\Spec(k)\rightarrow U$.
 Replacing $\alpha$ by $\alpha-q^\ast x_0^\ast \alpha$, we can assume
 $x^\ast\alpha=0$.  Then by the assumption of rigidity,
 applied to the henselization of $X$ at $x_0$, and the commutation of $F$ with
 filtered colimits, we deduce that there is an \'etale neighborhood $Y\rightarrow
 X$ of $x_0$ such that $\alpha$ pulls back to $0$ on $Y$.  Then we can take $U$ to be the image of $Y\rightarrow X$: as $k$ is separably closed, every $k$-point of $U$ lifts to $Y$.

Now we prove the proposition. 
It suffices to consider the case where $K$ is the separable closure of a prime
field, and so is in particular perfect. 
Therefore, the extension $K \to L$ is ind-smooth, i.e., $L$ is a filtered colimit of smooth
$K$-algebras $A_\alpha$, with $\alpha$ running over a filtered poset.  Then each map $K \to A_\alpha$ admits
a retraction, as $K$ is separably closed. 
It follows that $F(K) \to F(A_\alpha)$ is injective for each $\alpha$, so that
$F(K) \to F(L)$ is injective. 

We now argue surjectivity. Let $u \in F(A_\alpha)$; we show that  the image of
$u$ in $F(L)$ belongs to the
image of $F(K)$.  Passing to a component if necessary, we can assume $\Spec(A_\alpha)$ is connected.  Then $\Spec(A_\alpha\otimes_K L)$ is connected as well, since $K$ is algebraically closed.  Now consider the two maps
\[ f_1, f_2:A_\alpha   \rightrightarrows L,  \]
where the first map is the colimit structure map and where
where the second map is $A_\alpha \stackrel{s}{\to} K \to L$. 
By the first paragraph of the argument applied to the $L$-algebra $A_\alpha
\otimes_K L$, we conclude that these two maps have the same effect on $u$, verifying the claim.
\end{proof} 

We can rephrase Theorem \ref{etalekthm} in terms of homotopy group sheaves.  Let
$\pi_n(\mathcal{K}/p)$ denote the \'etale sheafification of the functor $\pi_n
(K(-)/p)$ over an arbitrary scheme $X$, and let $\pi_n( \mathcal{TC}/p)$ denote the same for $\pi_n (\TC(-)/p)$.  Then if $i$ denotes the closed inclusion $(X\times_{\Spec(\mathbb{Z})}\Spec(\mathbb{F}_p))_{et}\rightarrow X_{et}$, Theorem \ref{etalekthm} (plus the trivial fact that $\TC /p$ vanishes on rings where $p$ is invertible) is equivalent to the statement that for all $n$, there is an equivalence
\begin{equation} \label{stalkK}\pi_n(\mathcal{TC}/p)\simeq i_\ast i^\ast
\pi_n(\mathcal{K}/p)\end{equation}
adjoint to the map $i^\ast\pi_n(\mathcal{K}/p)\rightarrow i^\ast \pi_n(\mathcal{TC}/p)$ (also an equivalence) given by the restriction of the cyclotomic trace.  We can clearly also replace $(-)/p$ with $(-)/p^s$ for any $s$.

As a consequence, one obtains the following result, which will be discussed in more
detail in the forthcoming paper \cite{CM18}.
In the smooth case, this is one of the main results of \cite{GH}. 

\begin{theorem} 
\label{padicetaleKthy}
Let $X$ be a scheme proper over  $\spec R$ for a ring $R$ henselian along
$(p)$. Denote by
$\hat{K}^\sub{\'et}(-)$ the \'etale Postnikov sheafification of the $K$-theory presheaf,
meaning the limit over the \'etale sheafified Postnikov tower of
$K(-)$.\footnote{This is what Thomason's \'etale hypercohomology construction
implements. By \cite[Thm.~1.3]{CM18} it agrees with the more basic sheafification for the etale site as considered in \cite{HTT} under very mild finiteness assumptions on $X$.  We refer also to \cite[Sec.~1.3]{SAG} as a source for
sheaves of spectra and the associated $t$-structure.}  Then the natural map
$$\hat{K}^\sub{\'et}(X)/p\rightarrow \TC(X)/p$$
is an equivalence.
 \end{theorem} 
\begin{proof} The comparison map is induced by the cyclotomic trace, given that $\TC/p$ is an \'etale Postnikov sheaf (see \cite[Thm 5.16]{CM18} in full generality; for
affines, see \cite[Sec.~3]{GH}; in the semi-separated case see \cite{BMloc} for
the conclusion of Nisnevich descent).  Now assume $X$ is proper over $\Spec(R)$
henselian along $(p)$. Then the proper base change theorem together with
Gabber's affine analog of the proper base change theorem
\cite{Gabberaffine} combine to show $H^j_\sub{\'et}(X, \mathcal{F})\simeq
H^j_\sub{\'et}(X\times_{\Spec(R)}\Spec(R/pR), i^*\mathcal{F})$
for all torsion abelian sheaves $\mathcal{F}$ on $X_{et}$.  In particular,
$X_{et}$ has finite $p$-cohomological dimension by induction on an affine cover and the following Lemma~\ref{finiteetale}, so by comparing descent spectral sequences for $\hat{K}^\sub{\'et}(X)$ and $\TC(X)$ it suffices to show that $\pi_n(\mathcal{K}/p)$ and $\pi_n(\mathcal{TC}/p)$ have the same cohomology on $X_\sub{\'et}$.  But by the same base change results, this follows from the fact (see above) that they have the same restriction to $X\times_{\Spec(R)}\Spec(R/pR)$.\end{proof} 

Above we used the following standard lemma.

\begin{lemma} 
\label{finiteetale}
Let $R$ be an $\mathbb{F}_p$-algebra. Then the mod $p$ \'etale cohomological
dimension of $R$ is $\leq 1$. 
\end{lemma} 
\begin{proof} 
This is classical in the noetherian case from the Artin-Schreier sequence (see
\cite[Exp.~X, Thm.~5.1]{SGA4}). 
To obtain the non-noetherian case, we use the criterion from \cite[Exp. IX,
Prop.~5.5]{SGA4} together with the fact that cohomology commutes with filtered
colimits \cite[Tag 03Q4]{stacks-project}. 
\end{proof}

\subsection{Asymptotic comparison of $K$ and $ \TC$}

Next, 
we show that $K/p$ and $\TC/p$ agree in large degrees for $p$-adic rings
satisfying mild finiteness conditions. 
In view of Theorems~\ref{padicetaleKthy} and \ref{mainthm}, this yields a general $p$-adic
Lichtenbaum--Quillen statement for rings which are henselian along $(p)$. 
Note that for smooth algebras, the result follows from the
calculations of Geisser--Levine and Geisser--Hesselholt (Theorem~\ref{GHGL}) and 
for singular curves, a slight strengthening of this result appears in
\cite{GHexc}. 

\begin{theorem} 
\label{KTCasymptotic}
Let $R$ be a commutative ring and $p$ be a prime number. Suppose that $d
\geq 1$ and:  
\begin{enumerate}
\item  $R$ is henselian along $(p)$. 
\item The ring $R/p$ has finite Krull dimension.  
\item For any $x \in \spec (R/p)$, the residue field $k(x)$ has the property that
$[k(x): k(x)^p] \leq p^d$. \end{enumerate}
The map $K(R)/p^r
\to \TC(R)/p^r$ is an equivalence in degrees $\geq d$ for any $r$. 
\end{theorem} 
\begin{proof} 
We use throughout the following basic observation: if $T \to T'$ is a map of
spectra which is an equivalence in degrees $\geq d$, then for any spectrum $T''$
with a map $T'' \to T'$, the map $T \times_{T'} T'' \to T''$ is an equivalence
in degrees $\geq d$. 
Using the pullback square from Theorem~\ref{mainthm}
(involving the spectra $K(R)/p^r, \TC(R)/p^r, K(R/p)/p^r,
\TC(R/p)/p^r$), we reduce to the case where $R$ is an
$\mathbb{F}_p$-algebra. 
It now suffices to see that $K(R)/p^r \to \TC(R)/p^r$ is an equivalence
in degrees $\geq d$.

Note also that the result is clearly equivalent if we replace $K(R)/p$ with
$\mathbb{K}(R)/p$. 
By the theorems of Thomason--Trobaugh \cite{TT90} and
Blumberg--Mandell \cite{BMloc} respectively, 
$\mathbb{K}(-)/p^r$ and $\TC(-)/p^r$ 
are Nisnevich sheaves 
on $\spec R$ with 
values in the $\infty$-category $\Sp$.

By \cite[Th.~3.17]{CM18}, the finiteness of Krull dimension implies that
the Nisnevich topos of $\spec R$ has \emph{finite homotopy dimension}  in the sense of \cite[Def. 7.2.2.1]{HTT}.
In the case where $R$ is noetherian, the finiteness of homotopy dimension
appears in \cite[Thm.~3.7.7.1]{SAG}. 
As a consequence, Postnikov towers in the $\infty$-category of Nisnevich sheaves of
spectra on $\spec R$ are convergent and one has a descent spectral sequence. 
Therefore, it suffices to see that 
the maps on stalks induce isomorphisms in degrees $\geq d$. 
The maps on stalks are 
\[  \mathbb{K}(A)/p^r \to \TC(A)/p^r,  \]
as $A$ ranges over the connected finite \'etale algebras over henselizations of $R$ at prime ideals. 
Since the map $K(A)/p \to \mathbb{K}(A)/p$ is an equivalence in degrees $\geq
1$, 
it suffices to see that $K(A)/p \to \TC(A)/p$ is an 
equivalence in degrees $\geq d$ for each such $A$. This in turn follows from the fiber square 
of Theorem~\ref{mainthm}  (applied to the henselian local ring $A$) 
and the fact that if $k$ is a field of characteristic $p$, then the map $K(k)/p^r
\to \TC(k)/p^r$ is an equivalence in degrees $\geq \log_p [k: k^p]$ (which
follows from Theorem~\ref{GHGL} and the theory of $p$-bases, which implies $\dim
\Omega^1_k \leq d$ and so $\Omega^n_k  = 0$ for $n > d$ \cite[Thm.~26.5]{Matsumura}). 
Note also that the invariant $\log_p [k:k^p] = \dim \Omega^1_k$ is invariant under finite
separable extensions of fields of characteristic $p$. 
\end{proof} 

We immediately conclude the following $p$-adic Lichtenbaum--Quillen isomorphism.

\begin{corollary}
Let $R$ be a commutative ring which is henselian along $(p)$ with $\Spec(R/p)$
of finite Krull dimension, and suppose that $d\ge1$ is such that $[k(x): k(x)^p]
\leq p^d$ for all $x\in\Spec (R/p)$. Then the map $K(R)/p^r\to \hat{K}^\sub{\'et}(R)/p^r$ is an equivalence in degrees $\ge d$ for any $r$.
\end{corollary}
\begin{proof}
Combine Theorems \ref{KTCasymptotic} and \ref{padicetaleKthy}.
\end{proof}

\subsection{Examples}
\label{KTCexamples}
Finally, we include several explicit examples of comparisons between $K$ and $ \TC$. 
We begin with an example of Theorem~\ref{KTCasymptotic}. 
In particular, the standard rings occurring in algebra or algebraic geometry
over a perfect field of characteristic $p$ are covered by this example. 
\begin{example} 
Let $R$ be a 
noetherian $\mathbb{F}_p$-algebra
which is \emph{$F$-finite,} i.e., the Frobenius map is finite
(Definition~\ref{Ffinite}). 
Then $R$ has finite Krull dimension \cite[Prop. 1.1]{Kunz}. 
Moreover, if $R$ is generated
as a module by $p^d$ elements over the Frobenius, this passes to any localization. 
It follows that if $k$ is any residue field of $R$, then $[k: k^p] \leq p^d$. 
Therefore, the map $K(R)/p \to \TC(R)/p $ is an equivalence in degrees $\geq d$
thanks to Theorem~\ref{KTCasymptotic}. Note that Theorem~\ref{KTCasymptotic}  
assumes that $d \geq 1$, but if $d = 0$ then $R$ is a finite product of perfect
fields, so that the result   
follows from Example~\ref{perfex} below. 
\end{example}

We can also show that $K$-theory and $\TC$ agrees in connective degrees for
certain large rings. 
\begin{example} 
\label{perfex}
Let $R$ be a perfect $\mathbb{F}_p$-algebra, i.e., such that the Frobenius is
an isomorphism. Then the map $K(R)/p^r \to \TC(R)/p^r$ is an
equivalence on connective covers for each $r \geq 0$.

In fact, 
$\TC_i(R)$ vanishes for $i > 0$. 
This follows as in the calculation of $\TC(\mathbb{F}_p)$ in \cite[Sec.~IV-4]{nikolaus-scholze}. 
Indeed, by a variant of B\"okstedt's calculation, one finds
$\THH(R)_* \simeq R[\sigma]$ for $|\sigma| = 2$ and $$ \TC^-_*(R) =
W(R)[x, \sigma]/(x \sigma - p), \quad \TP_*(R) \simeq W(R)[x^{\pm 1}], \quad
|x| = -2.$$
If $R$ has no nontrivial idempotents, it follows that $\pi_0(\TC(R)/p^r) =
\mathbb{Z}/p^r \mathbb{Z}$
and that $\pi_{-1}(\TC_{}(R)/p^r)$ is the cokernel of $F -1 $ on $W_r(R)$. 
We refer to the work of Bhatt--Morrow--Scholze \cite{BMS2} for more details. 

Moreover, $K_i(R)$ is a $\mathbb{Z}[1/p]$-module for
$i > 0$. This follows from the existence of Adams operations in higher
$K$-theory and the fact that $\psi^p$ is the Frobenius, cf.~\cite{Hiller, Kratzer}. 
If $R$ has  no
nontrivial idempotents, it follows by Zariski descent that the kernel of the map $K_0(R) \to
\mathbb{Z}$ is a $\mathbb{Z}[1/p]$-module, so that $K_0(R)/p \simeq
\mathbb{Z}/p$.  
Combining all these observations, the claim for $R$ follows provided $\operatorname{Spec}(R)$ is connected.  To reduce the general case to that one, note that everything commutes with filtered colimits, so we can assume $R$ is the perfection of a finite type $\mathbb{F}_p$-algebra.  Then $\operatorname{Spec}(R)$ is noetherian, hence a finite disjoint union of connected affines, giving the reduction.
\end{example} 

We obtain the following corollary.
\begin{corollary}
Let $R$ be a ring henselian along $(p)$. 
Suppose $R/p$ is a \emph{semiperfect} ring, i.e., an $\mathbb{F}_p$-algebra 
such that the Frobenius map on $R/p$ is a surjection. 
The map $K(R)/p^r \to \TC(R)/p^r$ is an equivalence in degrees $\geq 0$ for any
$r$. 
\label{Kpfd}
\end{corollary}
\begin{proof} 
Using Theorem~\ref{mainthm}, we immediately reduce to the case where $R$ itself is a
semiperfect $\mathbb{F}_p$-algebra. 
In this case, let $I \subset R$ be the nilradical. It follows that $R/I$ is
perfect and that $I$ is locally nilpotent, so that $(R, I)$ is a henselian pair. 
By Theorem~\ref{mainthm} and Example~\ref{perfex}, it follows that the map
$K(R)/p^r \to \TC(R)/p^r$ is an equivalence in degrees $\geq 0$. 
\end{proof} 

\begin{example} Let $C$ be a complete nonarchimedean field whose residue field is perfect of
characteristic $p$. 
Let $\mathcal{O}_C \subset C$ be the ring of integers, and let $\pi \in
\mathcal{O}_C$ be a nonzero element of positive valuation. 
Then $\mathcal{O}_C$ is $\pi$-adically complete, and the image of the maximal ideal 
$\mathfrak{m}_C \subset \mathcal{O}_C$  is a locally nilpotent ideal in
$\mathcal{O}_C/\pi$: in particular, $\mathfrak{m}_C \subset \mathcal{O}_C$ is
henselian. We conclude 
that the map $K( \mathcal{O}_C)/p^r \to \TC(
\mathcal{O}_C)/p^r$ exhibits the former as the connective cover of the latter
for $r \geq 0$ since we know the analogous statement for the residue field (as
in Example~\ref{perfex}).

Given a perfectoid field $C$ (in the sense of \cite{Scholze}), the ring of
integers $\mathcal{O}_C$ has perfect residue field, so the above conclusion
holds. 
When $C = \mathbb{C}_p$ is the completed algebraic closure of $\mathbb{Q}_p$,
the $p$-adic $K$-theory of $\mathcal{O}_C$ 
was calculated by Nizio\l, cf.~\cite[Lem~3.1]{Niziol-crystalline} which shows that
$K(\mathcal{O}_C; \mathbb{Z}_p) \simeq K(C; \mathbb{C}_p)$, which in turn is
$p$-adic connective topological $K$-theory by \cite{Suslinalgclosed, Suslinlocal}. 
Similarly, $\TC(\mathcal{O}_C; \mathbb{Z}_p)$ was calculated by Hesselholt
\cite{HesselholtOC} and shown to agree with the $K$-theory. 
A description of $\TC(\mathcal{O}_C; \mathbb{Z}_p)$
in general has been given in \cite{BMS2}. 
\end{example}

\subsection{Split injectivity}
As we recalled in Theorem \ref{GHGL}, results of Geisser--Levine and
Geisser--Hesselholt show that the trace map $K(R)/p\to \TC(R)/p$ induces
split injections on homotopy groups whenever $R$ is an ind-smooth local $\bb
F_p$-algebra.  In fact, the same argument shows that this also holds with mod
$p^r$ coefficients for any $r$, if we use the logarithmic de Rham-Witt groups
(see Definition \ref{definition_dRW}) as the mod $p^r$ generalizations of the
$\nu^n(R)$.  The splitting comes from the \'{e}tale descent spectral sequence
for $\TC$, so that the complementary summand of $\pi_n(K(R)/p^r)$ in
$\pi_n(\TC(R)/p^r)$ is given by
$H^1(\Spec(R)_\sub{\'{e}t};\pi_{n+1}(\mathcal{TC}/p^r))$, where
$\pi_{n+1}(\mathcal{TC}/p^r)$ denotes the \'etale sheafification of the presheaf $\pi_n(\TC(-)/p^r)$.

We can use our main result to extend this to arbitrary local $\bb F_p$-algebras.

\begin{theorem}
Let $R$ be any local $\bb F_p$-algebra, $n\ge-1$, and $r\ge 1$. Then the
trace map $\pi_n (K(R)/p^r) \to \pi_n (\TC(R)/p^r)$ is split injective. More
precisely,  there is a functor $\widetilde{\nu^{n+1}_r}:\operatorname{CAlg}_{\mathbb{F}_p}\rightarrow\operatorname{Ab}$ and a natural transformation $\widetilde{\nu^{n+1}_r}(-)\rightarrow \pi_n(\TC(-)/p^r)$ such that, for $R$ local, the induced map
$$\widetilde{\nu^{n+1}_r}(R)\oplus \pi_n(K(R)/p^r)\To\pi_n(\TC(R)/p^r)$$
is an isomorphism.  This is also compatible with the natural transition maps as $r$ varies.

Moreover, $\widetilde{\nu^{n+1}_r}$ commutes with filtered colimits; in particular, the direct sum decomposition above holds for arbitrary $R$ after Zariski sheafification.
\label{injectivityresult}
\end{theorem}
\begin{proof}
Recall that if $R$ is an $\mathbb{F}_p$-algebra, we can make a functorial surjection to $R$ from a polynomial algebra, namely take the polynomial algebra $\mathbb{F}_p[x_a]_{a\in R}$ on variables indexed by the elements of $R$ with the map $x_a\mapsto a$.  Let $R'\rightarrow R$ denote the henselization of this surjection at its kernel; thus $R'$ is a functorial ind-smooth $\mathbb{F}_p$-algebra surjecting onto $R$ with henselian kernel.  Moreover, $R\mapsto R'$ evidently commutes with filtered colimits.

Define
for $m \geq 0$
$$\widetilde{\nu^m_r}(R) := H^1(\operatorname{Spec}(R')_\sub{\'{e}t};\pi_{m}(\mathcal{TC}/p^r)).$$
Since the coefficient presheaf $\pi_{m}(\TC(-)/p^r)$ commutes with filtered
colimits by Corollary \ref{TCring}, this \'etale cohomology group commutes with filtered colimits in $R'$ by standard cocontinuity arguments.  Combining with the previous, we find that $\widetilde{\nu^m_r}(-)$ commutes with filtered colimits.

Furthermore, since $\operatorname{Spec}(R')$ has \'{e}tale $p$-cohomological dimension $\leq 1$, the descent spectral sequence for $\TC(-)/p^r$ gives a natural map $\widetilde{\nu^{n+1}_r}(R)\rightarrow \pi_n(\TC(R')/p^r)$ (compare with the proofs of Theorem \ref{GHGL} and Theorem \ref{padicetaleKthy}).  Composing with the map on $\TC$ induced by $R'\rightarrow R$ defines the desired natural transformation $\widetilde{\nu^{n+1}_r}(R)\rightarrow \pi_n(\TC(R)/p^r)$.

Now assume $R$ is local.  Then $R'$ is too, since a henselian ideal is radical, and it is moreover ind-smooth.  Thus the argument recalled before the statement of the theorem, based on the results of Geisser--Levine and Geisser--Hesselholt, shows that
$$\widetilde{\nu^{n+1}_r}(R)\oplus \pi_n(K(R')/p^r)\overset{\sim}{\rightarrow}\pi_n(\TC(R')/p^r).$$
On the other hand our main rigidity theorem, Theorem \ref{mainthm}, gives a long exact sequence
$$\ldots\rightarrow\pi_n(K(R')/p^r)\rightarrow \pi_n(\TC(R')/p^r)\oplus \pi_n(K(R)/p^r)\rightarrow \pi_n(\TC(R)/p^r)\rightarrow\ldots.$$
Combining shows both that this long exact sequence breaks up into short exact sequences and that 
$$\widetilde{\nu^{n+1}_r}(R)\oplus\pi_n(K(R)/p^r)\overset{\sim}{\rightarrow}\pi_n(\TC(R)/p^r),$$
as claimed.
\end{proof}

We next proceed to identify these 
constructions $\widetilde{\nu^{m}_r}$. 
When $r=1$, the proof of Theorem \ref{GHGL} shows that
$\widetilde{\nu^m_1}(R)= \widetilde{\nu^m}(R')$, and this combines with Proposition \ref{toyFprigidityTC} to give an identification
$$\widetilde{\nu^m_1}(R)= \widetilde{\nu^m}(R):=\operatorname{coker}( 1 - C^{-1}: \Omega^m_R \to \Omega^m_R / d \Omega^{m-1}_R)$$
for arbitrary $\mathbb{F}_p$-algebras $R$ and $m \geq 0$.  
More generally, we can obtain a similar description of the
$\widetilde{\nu_r^m}$ for $r> 1$ as follows. 
As in \cite[Sec.~4]{Morrow-HW}, 
we have a natural map 
\[ \overline{F}: W_r \Omega_R^{m} \to W_r \Omega_R^{m}/ d V^{r-1}
\Omega_R^{m}  \]
for an arbitrary $\mathbb{F}_p$-algebra $R$ which factors the Frobenius $F: W_{r+1}
\Omega_R^{m} \to W_r \Omega_R^{m}$.  Let $\pi: W_r \Omega_R^{m} \to 
 W_r \Omega_R^{m}/ d V^{r-1} \Omega_R^{m}$ be the natural projection and 
consider the map 
 \begin{equation} \label{piminusF} W _r \Omega_R^{m} \xrightarrow{\pi - \overline{F}}  
 W_r \Omega_R^{m}/ (d V^{r-1} \Omega_R^{m})
.\end{equation} 
By \cite[Cor. 4.2(iii)]{Morrow-HW}, the map $\pi - \overline{F}$
has kernel given by the logarithmic forms, i.e., we have a short exact sequence 
 \begin{equation} \label{piminusF2} 
 0 \to 
W_r \Omega_{R, \mathrm{log}}^{m} \to 
 W _r \Omega_R^{m} \xrightarrow{\pi - \overline{F}}  
 W_r \Omega_R^{m}/ (d V^{r-1} \Omega_R^{m})
.\end{equation}

\begin{proposition} 
\label{identifytwiddle}
For any $\mathbb{F}_p$-algebra $R$, we have a natural identification
of graded abelian groups
\[ \widetilde{\nu_r^m}(R) \simeq \mathrm{coker}( \pi - \overline{F}).  \]
\end{proposition} 
\begin{proof} 

We claim that the map $\pi - \overline{F}$ is surjective locally in the \'etale topology 
and the cokernel satisfies rigidity for henselian pairs. 
The first claim follows because the composite map 
\[ W_{r+1} \Omega_R^{m} \stackrel{R}{\to} W_r \Omega_R^{m}
\xrightarrow{\pi  - \overline{F}} W_r \Omega_R^{m}/dV^{r-1} \Omega_R^{m} \]
actually lifts to the map $W_{r+1}\Omega_R^{m} \xrightarrow{R - F} W_r
\Omega_R^{m}$, which is a surjection in the \'etale topology thanks to
\cite[Cor. 4.1(ii)]{Morrow-HW}. 

Second, 
let $(R, I)$ be a henselian pair of $\mathbb{F}_p$-algebras. 
To see that $\mathrm{coker}( \pi - \overline{F})$ is rigid, we can imitate the
strategy of Proposition~\ref{toyFprigidityTC}. 
Namely, we 
let $L$ denote the kernel of the surjection
$W_r \Omega_R^{m} / dV^{r-1} \Omega_R^{m} \to W_r \Omega_{R/I}^{m}/d
V^{r-1} \Omega_{R/I}^{m}$ and 
contemplate the commutative diagram
\[ \xymatrix{
 & & W_r \Omega_{(R, I)}^{m} \ar[d]  \\
W_{r+1} \Omega_{(R, I)}^{m} \ar[d] \ar[rru]^{R - F} \ar[r]^R &   W_r
\Omega_{(R, I)}^{m} \ar[d]  \ar[r]^{\psi} &  L \ar[d]  \\
W_{r+1} \Omega_R^{m} \ar[d]  \ar[r]^R &  W_r \Omega_R^{m}\ar[d]
\ar[r]^-{\pi - \overline{F}} &  W_r \Omega_R^{m} / dV^{r-1}
\Omega_R^{m} \ar[d]  \\ 
W_{r+1} \Omega_{R/I}^{m} \ar[r]^R &  W_{r} \Omega_{R/I}^{m} \ar[r]^-{\pi -
\overline{F}} &  W_r
\Omega_{R/I}^{m}/ d V^{r-1} \Omega_{R/I}^{m}
}.\]

To prove rigidity, it suffices by the snake lemma  to show that $\psi:W_r \Omega_{(R,
I)}^{m} \to L$ is a surjection. 
Note that the kernel $L$ is surjected upon by the relative forms $W_r
\Omega_{(R, I)}^{m}$; this follows easily from the fact that $\Omega_R^{m}
\to \Omega_{R/I}^{m}$ is a surjection. 
Thus, the surjectivity of $\psi$ now follows from the surjectivity of $R - F:
W_{r+1} \Omega^{m}_{(R, I)} \to W_r \Omega^{m}_{(R, I)}$ 
given by Lemma~\ref{lemma_drw1}. 

Now we prove the proposition.  Let $R$ be an $\mathbb{F}_p$-algebra and recall that, by definition,
$$\widetilde{\nu_r^{m}}(R) = H^1( \mathrm{Spec}(R')_{\mathrm{et}}; \pi_m (
\mathcal{TC}/p^r)),$$
where $R'\rightarrow R$ is a surjective map from an ind-smooth
$\mathbb{F}_p$-algebra with henselian kernel. Since $\pi_m ( \mathcal{TC}/p^r)
\simeq W_r \Omega^{m}_{\mathrm{log}}$ (as
\'etale sheaves) over the ind-smooth $\operatorname{Spec}(R')$ by the
results of Geisser--Levine and Geisser--Hesselholt, we can use the resolution of 
\eqref{piminusF2} in the \'etale topology and the fact that the two terms of
\eqref{piminusF2} have no higher cohomology on affines to calculate this etale cohomology group. It follows that 
$\widetilde{\nu_r^{m}}(R)$ identifies with $\mathrm{coker}( \pi - \overline{F})$ on $R'$. Since we have just seen this cokernel satisfies
rigidity, we can replace $R'$ by $R$, whence the claim. 
\end{proof} 

\bibliographystyle{abbrv}
\bibliography{TCfinite}

\begin{thebibliography}{10}

\bibitem{ARloc}
J.~Ad\'amek and J.~Rosick\'y.
\newblock {\em Locally presentable and accessible categories}, volume 189 of
  {\em London Mathematical Society Lecture Note Series}.
\newblock Cambridge University Press, Cambridge, 1994.

\bibitem{AhKu}
S.~T. Ahearn and N.~J. Kuhn.
\newblock Product and other fine structure in polynomial resolutions of mapping
  spaces.
\newblock {\em Algebr. Geom. Topol.}, 2:591--647, 2002.

\bibitem{AMN}
B.~Antieau, A.~Mathew, and T.~Nikolaus.
\newblock On the {B}lumberg-{M}andell {K}\"{u}nneth theorem for {TP}.
\newblock {\em Selecta Math. (N.S.)}, 24(5):4555--4576, 2018.

\bibitem{SGA4}
M.~Artin, A.~Grothendieck, and J.~L. Verdier.
\newblock {\em Th{\'e}orie des topos et cohomologie {\'e}tale des sch{\'e}mas
  ({SGA} 4)}.
\newblock Springer-{V}erlag, 1973.

\bibitem{ayala-mg-rozenblyum}
D.~Ayala, A.~Mazel-Gee, and N.~Rozenblyum.
\newblock The geometry of the cyclotomic trace.
\newblock {\em ar{X}iv eprints}, 2017.

\bibitem{BG}
C.~Barwick and S.~Glasman.
\newblock Cyclonic spectra, cyclotomic spectra, and a conjecture of {K}aledin.
\newblock {\em ar{X}iv eprints}, 2016.

\bibitem{Bass}
H.~Bass.
\newblock {\em Algebraic {$K$}-theory}.
\newblock W. A. Benjamin, Inc., New York-Amsterdam, 1968.

\bibitem{BLM}
B.~Bhatt, J.~Lurie, and A.~Mathew.
\newblock The de {R}ham-{W}itt complex revisited.
\newblock {\em {A}st{\'e}risque}, to appear.

\bibitem{BMS2}
B.~Bhatt, M.~Morrow, and P.~Scholze.
\newblock Topological {H}ochschild homology and integral {$p$}-adic {H}odge
  theory.
\newblock {\em Publ. Math. Inst. Hautes \'{E}tudes Sci.}, 129:199--310, 2019.

\bibitem{BlochEsnaultKerz2013}
S.~{Bloch}, H.~{Esnault}, and M.~{Kerz}.
\newblock {Deformation of algebraic cycle classes in characteristic zero}.
\newblock {\em Algebraic Geometry}, 1(3):290--310, 2014.

\bibitem{BEK}
S.~Bloch, H.~Esnault, and M.~Kerz.
\newblock {$p$}-adic deformation of algebraic cycle classes.
\newblock {\em Invent. Math.}, 195(3):673--722, 2014.

\bibitem{BMloc}
A.~J. Blumberg and M.~A. Mandell.
\newblock Localization theorems in topological {H}ochschild homology and
  topological cyclic homology.
\newblock {\em Geom. Topol.}, 16(2):1053--1120, 2012.

\bibitem{BMcyc}
A.~J. Blumberg and M.~A. Mandell.
\newblock The homotopy theory of cyclotomic spectra.
\newblock {\em Geom. Topol.}, 19(6):3105--3147, 2015.

\bibitem{BHM}
M.~B\"okstedt, W.~C. Hsiang, and I.~Madsen.
\newblock The cyclotomic trace and algebraic {$K$}-theory of spaces.
\newblock {\em Invent. Math.}, 111(3):465--539, 1993.

\bibitem{Borceux1}
F.~Borceux.
\newblock {\em Handbook of categorical algebra. 1}, volume~50 of {\em
  Encyclopedia of Mathematics and its Applications}.
\newblock Cambridge University Press, Cambridge, 1994.
\newblock Basic category theory.

\bibitem{Borceux}
F.~Borceux.
\newblock {\em Handbook of categorical algebra. 2}, volume~51 of {\em
  Encyclopedia of Mathematics and its Applications}.
\newblock Cambridge University Press, Cambridge, 1994.
\newblock Categories and structures.

\bibitem{Cartier1957}
P.~Cartier.
\newblock Une nouvelle op\'{e}ration sur les formes diff\'{e}rentielles.
\newblock {\em C. R. Acad. Sci. Paris}, 244:426--428, 1957.

\bibitem{CM18}
D.~Clausen and A.~Mathew.
\newblock Hyperdescent and \'etale {$K$}-theory.
\newblock {\em ar{X}iv eprints}, 2019.

\bibitem{cortinas}
G.~Corti\~nas.
\newblock Infinitesimal {$K$}-theory.
\newblock {\em J. Reine Angew. Math.}, 503:129--160, 1998.

\bibitem{dundas}
B.~I. Dundas.
\newblock Relative {$K$}-theory and topological cyclic homology.
\newblock {\em Acta Math.}, 179(2):223--242, 1997.

\bibitem{Dundascont}
B.~I. Dundas.
\newblock Continuity of {$K$}-theory: an example in equal characteristics.
\newblock {\em Proc. Amer. Math. Soc.}, 126(5):1287--1291, 1998.

\bibitem{dgm}
B.~I. Dundas, T.~G. Goodwillie, and R.~McCarthy.
\newblock {\em The local structure of algebraic {K}-theory}, volume~18 of {\em
  Algebra and Applications}.
\newblock Springer-Verlag London, Ltd., London, 2013.

\bibitem{DKexcint}
B.~I. Dundas and H.~{\O}. Kittang.
\newblock Integral excision for {$K$}-theory.
\newblock {\em Homology Homotopy Appl.}, 15(1):1--25, 2013.

\bibitem{DundasMorrow}
B.~I. Dundas and M.~Morrow.
\newblock Finite generation and continuity of topological {H}ochschild and
  cyclic homology.
\newblock {\em Ann. Sci. \'Ec. Norm. Sup\'er. (4)}, 50(1):201--238, 2017.

\bibitem{elkik}
R.~Elkik.
\newblock Solutions d'\'equations \`a coefficients dans un anneau hens\'elien.
\newblock {\em Ann. Sci. \'Ecole Norm. Sup. (4)}, 6:553--603 (1974), 1973.

\bibitem{Emmanouil}
I.~Emmanouil.
\newblock Mittag-{L}effler condition and the vanishing of {$\varprojlim^1$}.
\newblock {\em Topology}, 35(1):267--271, 1996.

\bibitem{gabber}
O.~Gabber.
\newblock {$K$}-theory of {H}enselian local rings and {H}enselian pairs.
\newblock In {\em Algebraic {$K$}-theory, commutative algebra, and algebraic
  geometry ({S}anta {M}argherita {L}igure, 1989)}, volume 126 of {\em Contemp.
  Math.}, pages 59--70. Amer. Math. Soc., Providence, RI, 1992.

\bibitem{Gabberaffine}
O.~Gabber.
\newblock Affine analog of the proper base change theorem.
\newblock {\em Israel J. Math.}, 87(1-3):325--335, 1994.

\bibitem{Geissersurvey}
T.~Geisser.
\newblock Motivic cohomology, {$K$}-theory and topological cyclic homology.
\newblock In {\em Handbook of {$K$}-theory. {V}ol. 1, 2}, pages 193--234.
  Springer, Berlin, 2005.

\bibitem{GH}
T.~Geisser and L.~Hesselholt.
\newblock Topological cyclic homology of schemes.
\newblock In {\em Algebraic {$K$}-theory ({S}eattle, {WA}, 1997)}, volume~67 of
  {\em Proc. Sympos. Pure Math.}, pages 41--87. Amer. Math. Soc., Providence,
  RI, 1999.

\bibitem{GHexc}
T.~Geisser and L.~Hesselholt.
\newblock Bi-relative algebraic {$K$}-theory and topological cyclic homology.
\newblock {\em Invent. Math.}, 166(2):359--395, 2006.

\bibitem{GHlocal}
T.~Geisser and L.~Hesselholt.
\newblock On the {$K$}-theory and topological cyclic homology of smooth schemes
  over a discrete valuation ring.
\newblock {\em Trans. Amer. Math. Soc.}, 358(1):131--145, 2006.

\bibitem{GHcomplete}
T.~Geisser and L.~Hesselholt.
\newblock On the {$K$}-theory of complete regular local {$\Bbb F_p$}-algebras.
\newblock {\em Topology}, 45(3):475--493, 2006.

\bibitem{GL}
T.~Geisser and M.~Levine.
\newblock The {$K$}-theory of fields in characteristic {$p$}.
\newblock {\em Invent. Math.}, 139(3):459--493, 2000.

\bibitem{GT}
H.~A. Gillet and R.~W. Thomason.
\newblock The {$K$}-theory of strict {H}ensel local rings and a theorem of
  {S}uslin.
\newblock {\em J. Pure Appl. Algebra}, 34(2-3):241--254, 1984.

\bibitem{Goodrel}
T.~G. Goodwillie.
\newblock Relative algebraic {$K$}-theory and cyclic homology.
\newblock {\em Ann. of Math. (2)}, 124(2):347--402, 1986.

\bibitem{GooIII}
T.~G. Goodwillie.
\newblock Calculus. {III}. {T}aylor series.
\newblock {\em Geom. Topol.}, 7:645--711, 2003.

\bibitem{EGA_IV_2}
A.~Grothendieck.
\newblock {\'E}l\'ements de g\'eom\'etrie alg\'ebrique. {IV}. \'etude locale
  des sch\'emas et des morphismes de sch\'emas. {II}.
\newblock {\em Inst. Hautes \'Etudes Sci. Publ. Math.}, (24):231, 1965.

\bibitem{Hesselholt}
L.~Hesselholt.
\newblock On the {$p$}-typical curves in {Q}uillen's {$K$}-theory.
\newblock {\em Acta Math.}, 177(1):1--53, 1996.

\bibitem{HesselholtOC}
L.~Hesselholt.
\newblock On the topological cyclic homology of the algebraic closure of a
  local field.
\newblock In {\em An alpine anthology of homotopy theory}, volume 399 of {\em
  Contemp. Math.}, pages 133--162. Amer. Math. Soc., Providence, RI, 2006.

\bibitem{hesselholt-tp}
L.~Hesselholt.
\newblock Topological {H}ochschild homology and the {H}asse-{W}eil zeta
  function.
\newblock In {\em An alpine bouquet of algebraic topology}, volume 708 of {\em
  Contemp. Math.}, pages 157--180. Amer. Math. Soc., Providence, RI, 2018.

\bibitem{HM97}
L.~Hesselholt and I.~Madsen.
\newblock On the {$K$}-theory of finite algebras over {W}itt vectors of perfect
  fields.
\newblock {\em Topology}, 36(1):29--101, 1997.

\bibitem{HMlocal}
L.~Hesselholt and I.~Madsen.
\newblock On the {$K$}-theory of local fields.
\newblock {\em Ann. of Math. (2)}, 158(1):1--113, 2003.

\bibitem{Hiller}
H.~L. Hiller.
\newblock {$\lambda $}-rings and algebraic {$K$}-theory.
\newblock {\em J. Pure Appl. Algebra}, 20(3):241--266, 1981.

\bibitem{Hovey}
M.~Hovey.
\newblock {\em Model categories}, volume~63 of {\em Mathematical Surveys and
  monographs}.
\newblock American {M}athematical {S}ociety, 2007.

\bibitem{illusie-derham-witt}
L.~Illusie.
\newblock Complexe de de\thinspace {R}ham-{W}itt et cohomologie cristalline.
\newblock {\em Ann. Sci. \'Ecole Norm. Sup. (4)}, 12(4):501--661, 1979.

\bibitem{Katz70}
N.~M. Katz.
\newblock Nilpotent connections and the monodromy theorem: {A}pplications of a
  result of {T}urrittin.
\newblock {\em Inst. Hautes \'{E}tudes Sci. Publ. Math.}, (39):175--232, 1970.

\bibitem{Kerz2010}
M.~Kerz.
\newblock Milnor {$K$}-theory of local rings with finite residue fields.
\newblock {\em J. Algebraic Geom.}, 19(1):173--191, 2010.

\bibitem{Kratzer}
C.~Kratzer.
\newblock Op\'erations d'{A}dams en {$K$}-th\'eorie alg\'ebrique.
\newblock {\em C. R. Acad. Sci. Paris S\'er. A-B}, 287(5):A297--A298, 1978.

\bibitem{Kunz}
E.~Kunz.
\newblock On {N}oetherian rings of characteristic {$p$}.
\newblock {\em Amer. J. Math.}, 98(4):999--1013, 1976.

\bibitem{land-tamme}
M.~Land and G.~Tamme.
\newblock On the {$K$}-theory of pullbacks.
\newblock {\em Ann. of Math. (2)}, 190(3):877--930, 2019.

\bibitem{LangerZink2004}
A.~Langer and T.~Zink.
\newblock De {R}ham-{W}itt cohomology for a proper and smooth morphism.
\newblock {\em J. Inst. Math. Jussieu}, 3(2):231--314, 2004.

\bibitem{Loday}
J.-L. Loday.
\newblock {\em Cyclic homology}, volume 301 of {\em Grundlehren der
  Mathematischen Wissenschaften [Fundamental Principles of Mathematical
  Sciences]}.
\newblock Springer-Verlag, Berlin, 1992.
\newblock Appendix E by Mar\'\i a O. Ronco.

\bibitem{SAG}
J.~Lurie.
\newblock {\em Spectral algebraic geometry}.
\newblock Available at \url{https://www.math.ias.edu/~lurie/}.

\bibitem{HTT}
J.~Lurie.
\newblock {\em Higher topos theory}, volume 170 of {\em Annals of Mathematics
  Studies}.
\newblock Princeton University Press, Princeton, NJ, 2009.

\bibitem{HA}
J.~Lurie.
\newblock {\em Higher Algebra}.
\newblock Available at \url{https://www.math.ias.edu/~lurie/}, 2014.

\bibitem{Mthick}
A.~Mathew.
\newblock A thick subcategory theorem for modules over certain ring spectra.
\newblock {\em Geom. Topol.}, 19(4):2359--2392, 2015.

\bibitem{MNN17}
A.~Mathew, N.~Naumann, and J.~Noel.
\newblock Nilpotence and descent in equivariant stable homotopy theory.
\newblock {\em Adv. Math.}, 305:994--1084, 2017.

\bibitem{Matsumura}
H.~Matsumura.
\newblock {\em Commutative ring theory}, volume~8 of {\em Cambridge Studies in
  Advanced Mathematics}.
\newblock Cambridge University Press, Cambridge, 1986.
\newblock Translated from the Japanese by M. Reid.

\bibitem{Matsumura1989}
H.~Matsumura.
\newblock {\em Commutative ring theory}, volume~8 of {\em Cambridge Studies in
  Advanced Mathematics}.
\newblock Cambridge University Press, Cambridge, second edition, 1989.
\newblock Translated from the Japanese by M. Reid.

\bibitem{Maygeom}
J.~P. May.
\newblock {\em The geometry of iterated loop spaces}.
\newblock Springer-Verlag, Berlin-New York, 1972.
\newblock Lectures Notes in Mathematics, Vol. 271.

\bibitem{mccarthy}
R.~McCarthy.
\newblock Relative algebraic {$K$}-theory and topological cyclic homology.
\newblock {\em Acta Math.}, 179(2):197--222, 1997.

\bibitem{Morrowpro}
M.~Morrow.
\newblock Pro unitality and pro excision in algebraic {$K$}-theory and cyclic
  homology.
\newblock {\em J. Reine Angew. Math.}, 736:95--139, 2018.

\bibitem{Morrow-HW}
M.~Morrow.
\newblock {$K$}-theory and logarithmic {H}odge--{W}itt sheaves of formal scheme
  in characteristic {$p$}.
\newblock {\em Ann. Sci. \'Ecole Norm. Sup. (4)}, 52(6):1537--1601, 2019.

\bibitem{nikolaus-scholze}
T.~Nikolaus and P.~Scholze.
\newblock On topological cyclic homology.
\newblock {\em Acta Math.}, 221(2):203--409, 2018.

\bibitem{Niziol-crystalline}
W.~Nizio\l.
\newblock Crystalline conjecture via {$K$}-theory.
\newblock {\em Ann. Sci. \'{E}cole Norm. Sup. (4)}, 31(5):659--681, 1998.

\bibitem{Panin}
I.~A. Panin.
\newblock The {H}urewicz theorem and {$K$}-theory of complete discrete
  valuation rings.
\newblock {\em Izv. Akad. Nauk SSSR Ser. Mat.}, 50(4):763--775, 878, 1986.

\bibitem{Popescu1}
D.~Popescu.
\newblock General {N}\'eron desingularization.
\newblock {\em Nagoya Math. J.}, 100:97--126, 1985.

\bibitem{Popescu2}
D.~Popescu.
\newblock General {N}\'eron desingularization and approximation.
\newblock {\em Nagoya Math. J.}, 104:85--115, 1986.

\bibitem{Qui}
D.~Quillen.
\newblock On the cohomology and {$K$}-theory of the general linear groups over
  a finite field.
\newblock {\em Ann. of Math. (2)}, 96:552--586, 1972.

\bibitem{QuillenHA}
D.~G. Quillen.
\newblock {\em Homotopical algebra}.
\newblock Lecture Notes in Mathematics, No. 43. Springer-Verlag, Berlin-New
  York, 1967.

\bibitem{Raynaud-henselian}
M.~Raynaud.
\newblock {\em Anneaux locaux hens\'{e}liens}.
\newblock Lecture Notes in Mathematics, Vol. 169. Springer-Verlag, Berlin-New
  York, 1970.

\bibitem{Rei87}
L.~Reid.
\newblock {$N$}-dimensional rings with an isolated singular point having
  nonzero {$K_{-N}$}.
\newblock {\em $K$-Theory}, 1(2):197--205, 1987.

\bibitem{Scholze}
P.~Scholze.
\newblock Perfectoid spaces.
\newblock {\em Publ. Math. Inst. Hautes \'Etudes Sci.}, 116:245--313, 2012.

\bibitem{Condensed}
P.~Scholze.
\newblock Lectures on condensed mathematics.
\newblock 2019.
\newblock Available at
  \url{https://www.math.uni-bonn.de/people/scholze/Condensed.pdf}.

\bibitem{Serresurvey}
J.-P. Serre.
\newblock Arithmetic groups.
\newblock In {\em Homological group theory ({P}roc. {S}ympos., {D}urham,
  1977)}, volume~36 of {\em London Math. Soc. Lecture Note Ser.}, pages
  105--136. Cambridge Univ. Press, Cambridge-New York, 1979.

\bibitem{Shiho2007}
A.~Shiho.
\newblock On logarithmic {H}odge-{W}itt cohomology of regular schemes.
\newblock {\em J. Math. Sci. Univ. Tokyo}, 14(4):567--635, 2007.

\bibitem{Arf}
V.~P. Snaith.
\newblock {\em Stable homotopy around the {A}rf-{K}ervaire invariant}, volume
  273 of {\em Progress in Mathematics}.
\newblock Birkh\"auser Verlag, Basel, 2009.

\bibitem{stacks-project}
T.~{Stacks Project Authors}.
\newblock {Stacks Project}.
\newblock \url{http://stacks.math.columbia.edu}, 2017.

\bibitem{Suslinalgclosed}
A.~Suslin.
\newblock On the {$K$}-theory of algebraically closed fields.
\newblock {\em Invent. Math.}, 73(2):241--245, 1983.

\bibitem{Suslinlocal}
A.~A. Suslin.
\newblock On the {$K$}-theory of local fields.
\newblock In {\em Proceedings of the {L}uminy conference on algebraic
  {$K$}-theory ({L}uminy, 1983)}, volume~34, pages 301--318, 1984.

\bibitem{ThomasonICM}
R.~W. Thomason.
\newblock The local to global principle in algebraic {$K$}-theory.
\newblock In {\em Proceedings of the {I}nternational {C}ongress of
  {M}athematicians, {V}ol.\ {I}, {II} ({K}yoto, 1990)}, pages 381--394. Math.
  Soc. Japan, Tokyo, 1991.

\bibitem{TT90}
R.~W. Thomason and T.~Trobaugh.
\newblock Higher algebraic {$K$}-theory of schemes and of derived categories.
\newblock In {\em The {G}rothendieck {F}estschrift, {V}ol.\ {III}}, volume~88
  of {\em Progr. Math.}, pages 247--435. Birkh\"auser Boston, Boston, MA, 1990.

\bibitem{vanderKallen1977}
W.~van~der Kallen.
\newblock The {$K\sb{2}$} of rings with many units.
\newblock {\em Ann. Sci. \'Ecole Norm. Sup. (4)}, 10(4):473--515, 1977.

\bibitem{vdK}
W.~van~der Kallen.
\newblock Homology stability for linear groups.
\newblock {\em Invent. Math.}, 60(3):269--295, 1980.

\bibitem{vanderKallen1986}
W.~van~der Kallen.
\newblock Descent for the {$K$}-theory of polynomial rings.
\newblock {\em Math. Z.}, 191(3):405--415, 1986.

\bibitem{Wei91}
C.~A. Weibel.
\newblock Pic is a contracted functor.
\newblock {\em Invent. Math.}, 103(2):351--377, 1991.

\bibitem{KBook}
C.~A. Weibel.
\newblock {\em The {$K$}-book}, volume 145 of {\em Graduate Studies in
  Mathematics}.
\newblock American Mathematical Society, Providence, RI, 2013.
\newblock An introduction to algebraic $K$-theory.

\end{thebibliography}

\end{document}